\documentclass[12pt,a4paper]{amsart}
\usepackage{amsthm,amsfonts,amsmath,amssymb,latexsym}
\usepackage{epsfig,graphics,color}
\usepackage[dvipsnames]{xcolor}
\usepackage[all]{xy}
\usepackage[breaklinks=true]{hyperref}
\usepackage{mathrsfs}
\usepackage{stmaryrd}
\usepackage{verbatim}
\usepackage{bm}
\usepackage{mathabx}
\usepackage{enumitem}
\usepackage{pgf,amsmath,tikz,type1cm,fix-cm}
\usetikzlibrary{arrows.meta,bending,positioning}

\newlength{\XHeight}
\newlength{\XWidth}
\setlist[itemize,1]{leftmargin=\dimexpr 26pt-.1in}

\newtheorem{PARA}{}[section]
\newtheorem{theorem}[PARA]{Theorem}
\newtheorem{corollary}[PARA]{Corollary}
\newtheorem{lemma}[PARA]{Lemma}
\newtheorem{proposition}[PARA]{Proposition}
\newtheorem{definition}[PARA]{Definition}
\newtheorem{definition-proposition}[PARA]{Definition-Proposition}

\theoremstyle{definition}
\newtheorem{remark}[PARA]{Remark}
\theoremstyle{theorem}
\newtheorem{example}[PARA]{Example}
\newcommand{\para}{\begin{PARA}\rm}
\newcommand{\arap}{\end{PARA}\rm}
\newcommand{\dfn}{\begin{definition}\rm}
\newcommand{\nfd}{\end{definition}\rm}
\newcommand{\rmk}{\begin{remark}\rm}
\newcommand{\kmr}{\end{remark}\rm}
\newcommand{\xmpl}{\begin{example}\rm}
\newcommand{\lpmx}{\end{example}\rm}
\newcommand{\cA}{\mathcal{A}}

\newcommand{\cC}{\mathcal{C}}
\newcommand{\cD}{\mathcal{D}}

\newcommand{\cH}{\mathcal{H}}

\newcommand{\cM}{\mathcal{M}}

\newcommand{\cP}{\mathcal{P}}

\renewcommand{\H}{{\mathbb{H}}}

\newcommand{\R}{{\mathbb{R}}}

\newcommand{\Z}{{\mathbb{Z}}}
\newcommand{\coker}{\mathrm{coker}\, }  
\newcommand{\colim}{\mathrm{ colim}\, }  
\newcommand{\im}{\mathrm{im}\,}        
\newcommand{\Ext}{\mathrm{ Ext }}  

\newcommand{\ind}{\mathrm{ind}}

\newcommand{\ev}{\mathrm{ev}}

\newcommand{\Hom}{\mathrm{Hom}}

\newcommand{\eps}{{\varepsilon}}
\newcommand{\om}{{\omega}}
\newcommand{\Om}{{\Omega}}

\newcommand{\CZ}{\mathrm{CZ}}
\def\NABLA#1{{\mathop{\nabla\kern-.5ex\lower1ex\hbox{$#1$}}}}
\def\Nabla#1{\nabla\kern-.5ex{}_{#1}}
\def\Tabla#1{\Tilde\nabla\kern-.5ex{}_{#1}}
\renewcommand{\Tilde}{\widetilde}

\newcommand{\p}{{\partial}}

\newcommand{\la}{\langle}
\newcommand{\ra}{\rangle}
\newcommand{\wh}{\widehat}
\newcommand{\ol}{\overline}

\newcommand*\circled[1]{\tikz[baseline=(char.base)]{
            \node[shape=circle,draw,inner sep=2pt] (char) {#1};}}

\newcommand{\boldmu}{{\boldsymbol{\mu}}}
\newcommand{\boldlambda}{{\boldsymbol{\lambda}}}
\newcommand{\boldbeta}{{\boldsymbol{\beta}}}
\newcommand{\boldeta}{{\boldsymbol{\eta}}}

\newcommand{\boldc}{{\boldsymbol{c}}}

\newcommand{\boldB}{{\boldsymbol{B}}}

\newcommand{\um}{{\underline{m}}}
\newcommand{\utau}{{\underline{\tau}}}
\newcommand{\usigma}{{\underline{\sigma}}}

\newcommand*\Hamlambdamuaxis[2]{\ensuremath{
       \settowidth{\XWidth}{$\lambda$}
       \setlength{\XHeight}{\XWidth}
       \tikz[baseline]{
         \draw[thick] (-\XWidth,2\XHeight) -- (0,2\XHeight) -- (2\XWidth,0) -- (3\XWidth,\XHeight); 
         \draw[arrows={[sep] - Stealth[length=3pt]}] (-\XWidth,-.5\XHeight) -- (-\XWidth,3\XHeight);
         \draw[arrows={[sep] - Stealth[length=3pt]}] (-1.5\XWidth,0) -- (3.5\XWidth,0);
         \node at (1.5\XWidth,1.5\XHeight)[scale=0.5] {#1};
         \node at (3.2\XWidth,.5\XHeight)[scale=0.5] {#2};
       }
     }
}

\newcommand*\minusHamlambdamuaxis[2]{\ensuremath{
       \settowidth{\XWidth}{$\lambda$}
       \setlength{\XHeight}{\XWidth}
       \tikz[baseline]{
         \draw[thick] (-\XWidth,-2\XHeight) -- (0,-2\XHeight) -- (2\XWidth,0) -- (3\XWidth,-\XHeight); 
         \draw[arrows={[sep] - Stealth[length=3pt]}] (-\XWidth,-3\XHeight) -- (-\XWidth,3\XHeight);
         \draw[arrows={[sep] - Stealth[length=3pt]}] (-1.5\XWidth,0) -- (3.5\XWidth,0);
         \node at (1.5\XWidth,-1.4\XHeight)[scale=0.5] {#1};
         \node at (3.2\XWidth,-.4\XHeight)[scale=0.5] {#2};
       }
     }
}

\newcommand*\Hammuaxis[1]{\ensuremath{
       \settowidth{\XWidth}{$\lambda$}
       \setlength{\XHeight}{\XWidth}
       \tikz[baseline]{
         \draw[thick] (-\XWidth,0) -- (2\XWidth,0) -- (3\XWidth,\XHeight); 
         \draw[arrows={[sep] - Stealth[length=3pt]}] (-\XWidth,-.5\XHeight) -- (-\XWidth,3.5\XHeight);
         \draw[arrows={[sep] - Stealth[length=3pt]}] (-1.5\XWidth,0) -- (3.5\XWidth,0);
         \node at (3.2\XWidth,.5\XHeight)[scale=0.5] {#1};
       }
     }
}

\newcommand*\minusHammuaxis[1]{\ensuremath{
       \settowidth{\XWidth}{$\lambda$}
       \setlength{\XHeight}{\XWidth}
       \tikz[baseline]{
         \draw[thick] (-\XWidth,0) -- (2\XWidth,0) -- (3\XWidth,-\XHeight); 
         \draw[arrows={[sep] - Stealth[length=3pt]}] (-\XWidth,-.5\XHeight) -- (-\XWidth,3.5\XHeight);
         \draw[arrows={[sep] - Stealth[length=3pt]}] (-1.5\XWidth,0) -- (3.5\XWidth,0);
         \node at (3.4\XWidth,-.5\XHeight)[scale=0.5] {#1};
       }
     }
}

\newcommand*\HamLlambda[1]{\ensuremath{
       \settowidth{\XWidth}{$\lambda$}
       \setlength{\XHeight}{\XWidth}
       \tikz[baseline]{
         \draw[thick] (-\XWidth,2\XHeight) -- (0,2\XHeight) -- (3\XWidth,-\XHeight);
         \draw[arrows={[sep] - Stealth[length=3pt]}] (-\XWidth,-.5\XHeight) -- (-\XWidth,3\XHeight);
         \draw[arrows={[sep] - Stealth[length=3pt]}] (-1.5\XWidth,0) -- (3.5\XWidth,0);
         \node at (1.5\XWidth,1.5\XHeight)[scale=0.5] {#1};
       }
     }
}

\newcommand*\minusHamLlambda[1]{\ensuremath{
       \settowidth{\XWidth}{$\lambda$}
       \setlength{\XHeight}{\XWidth}
       \tikz[baseline]{
         \draw[thick] (-\XWidth,-2\XHeight) -- (0,-2\XHeight) -- (3\XWidth,\XHeight);
         \draw[arrows={[sep] - Stealth[length=3pt]}] (-\XWidth,-3\XHeight) -- (-\XWidth,3\XHeight);
         \draw[arrows={[sep] - Stealth[length=3pt]}] (-1.5\XWidth,0) -- (3.5\XWidth,0);
         \node at (1.5\XWidth,-1.4\XHeight)[scale=0.5] {#1};
       }
     }
}

\newcommand*\homotopyHamminusmutomu[2]{\ensuremath{
       \settowidth{\XWidth}{$\lambda$}
       \setlength{\XHeight}{\XWidth}
       \tikz[baseline]{
	\draw[thick] (-\XWidth,0) -- (2\XWidth,0) -- (3\XWidth,\XHeight); 
	\draw[thick] (-\XWidth,0) -- (2\XWidth,0) -- (3\XWidth,-\XHeight); 
	\draw [arrows={[sep] - Stealth[length=3pt]}] (2.9\XWidth,-.5\XHeight)  .. controls (3.4\XWidth,0) .. (2.9\XWidth,.5\XHeight);
         \node at (3.4\XWidth,-.5\XHeight)[scale=0.5] {#1};
         \node at (2.9\XWidth,.2\XHeight)[scale=0.5] {#2};
        }
     }
}

\newcommand*\homotopyHamLlambdatoHamlambdamu[2]{\ensuremath{
       \settowidth{\XWidth}{$\lambda$}
       \setlength{\XHeight}{\XWidth}
       \tikz[baseline]{
        \draw[thick] (-\XWidth,2\XHeight) -- (0,2\XHeight) -- (3\XWidth,-\XHeight);
	\draw[thick] (-\XWidth,2\XHeight) -- (0,2\XHeight) -- (2\XWidth,0) -- (3\XWidth,\XHeight); 
	\draw [arrows={[sep] - Stealth[length=3pt]}] (2.9\XWidth,-.5\XHeight)  .. controls (3.4\XWidth,0) .. (2.9\XWidth,.5\XHeight);
         \node at (3.4\XWidth,-.5\XHeight)[scale=0.5] {#1};
         \node at (2.9\XWidth,.2\XHeight)[scale=0.5] {#2};
        }
     }
}

\newcommand*\homotopyminusHamlambdamutominusHamLlambda[2]{\ensuremath{
       \settowidth{\XWidth}{$\lambda$}
       \setlength{\XHeight}{\XWidth}
       \tikz[baseline]{
         \draw[thick] (-\XWidth,-2\XHeight) -- (0,-2\XHeight) -- (2\XWidth,0) -- (3\XWidth,-\XHeight); 
         \draw[thick] (-\XWidth,-2\XHeight) -- (0,-2\XHeight) -- (3\XWidth,\XHeight);
	\draw [arrows={[sep] - Stealth[length=3pt]}] (2.9\XWidth,-.5\XHeight)  .. controls (3.4\XWidth,0) .. (2.9\XWidth,.5\XHeight);
         \node at (1.5\XWidth,-1.4\XHeight)[scale=0.5] {#1};
         \node at (3.2\XWidth,-.4\XHeight)[scale=0.5] {#2};
        }
     }
}

\newcommand{\alex}{\color{black}}

\newcommand{\refone}{\color{black}}
\newcommand{\reftwo}{\color{black}}

\begin{document}

\title[Reduced symplectic homology]{Reduced symplectic homology and string topology}
\author{Kai Cieliebak}
\address{\!\!\!\!\!\!\!Universit\"at Augsburg \newline Universit\"atsstrasse 14, D-86159 Augsburg, Germany}
\email{kai.cieliebak@math.uni-augsburg.de}
\author{Alexandru Oancea}
\address{ 
\!\!\!\!\!\!\!Universit\'e de Strasbourg \newline 
Institut de recherche math\'ematique avanc\'ee, IRMA \newline
Strasbourg, France}
\email{oancea@unistra.fr}
\date{\today}


\begin{abstract}
We introduce a common domain of definition for the  
loop product and the loop coproduct, reduced loop homology, on which they combine to a unital infinitesimal anti-symmetric bialgebra structure. In particular, a relation conjectured by Sullivan holds with an extra term. The structure depends on choices governed by secondary continuation maps. These results on string topology are proved in the more general context of reduced symplectic homology for a suitable class of Weinstein manifolds. 
\end{abstract}

\maketitle

\setcounter{tocdepth}{1}
{\footnotesize \tableofcontents}


\section{Introduction}\label{sec:introduction}

The initial motivation for this paper lies in string topology. 
Given a closed orientable manifold $M$ of dimension $n$, denote $\Lambda = \Lambda M$ its free loop space, $\Lambda_0\subset \Lambda$ the subspace of constant loops, $\H_*\Lambda = H_{*+n}\Lambda$ the degree shifted loop homology and $\H^*\Lambda = H^{*+n}\Lambda$ the degree shifted loop cohomology of $M$. Chas and Sullivan~\cite{CS} constructed the \emph{loop product} $\mu :\H_*\Lambda \otimes \H_*\Lambda\to \H_*\Lambda$ of shifted degree $0$, and Sullivan~\cite{Sullivan-open-closed} and Goresky-Hingston~\cite{Goresky-Hingston} constructed the \emph{loop coproduct} $\lambda:\H_*(\Lambda,\Lambda_0)\to \H_*(\Lambda,\Lambda_0) \otimes \H_*(\Lambda,\Lambda_0)$ (with coefficients in a field) of shifted degree $1-2n$, respectively the algebraically dual \emph{cohomology product} $\lambda^\vee:\H^*(\Lambda,\Lambda_0)\otimes \H^*(\Lambda,\Lambda_0)\to \H^*(\Lambda,\Lambda_0)$ of shifted degree $2n-1$ (with arbitrary coefficients). Very quickly the question arose as to whether it is possible to find a common domain of definition for the product and the coproduct, motivated in particular by the following relation conjectured by Sullivan:  
\begin{equation} \label{eq:Sullivan}
\lambda\mu = (\mu\otimes 1)(1\otimes\lambda) + (1\otimes \mu)(\lambda\otimes 1). 
\end{equation}

We gave in~\cite{CHO-PD} one answer to this question, replacing $\H_*\Lambda$ by the \emph{Rabinowitz loop homology} $\widehat{\H}_*\Lambda$, which is a homology group of Tate flavor consisting roughly of one copy of the loop homology and one copy of the loop cohomology, where both $\mu$ and $\lambda$ admit extensions which fit into the structure of a graded Frobenius algebra satisfying Poincar\'e duality~\cite{CHO-algebra}. 

It is the purpose of this paper to give a different answer to this question, replacing $\H_*\Lambda$ with \emph{reduced loop homology}
$$
\ol{\H}_*\Lambda=\H_*\Lambda/\chi(M)[pt],
$$ 
where $\chi(M)$ is the Euler characteristic and $[pt]$ is the class of a point seen as a constant loop. This homology group has a much more geometric flavor than $\widehat{\H}_*\Lambda$, since it lies between $\H_*\Lambda$ and $\H_*(\Lambda,\Lambda_0)$ in the sense that the canonical projection 
factors as $\H_*\Lambda\to \ol{\H}_*\Lambda\to \H_*(\Lambda,\Lambda_0)$. Reduced loop homology is 
``half" of Rabinowitz loop homology and serves as an excellent approximation for loop homology (to the point that it coincides with the latter whenever the Euler characteristic vanishes, e.g., in odd dimensions).    
We prove in this paper the following result on reduced loop homology.

\begin{theorem} \label{thm:main-intro} 
Assume $\dim M= 1$ or $\dim M \ge 3$.
The loop product on $\H_*\Lambda$ descends to $\ol\H_*\Lambda$ and the loop coproduct on $\H_*(\Lambda,\Lambda_0)$ extends to $\ol\H_*\Lambda$ (canonically if $H_1M=0$). Each such extension $\lambda$ defines together with the loop product $\mu$ the structure of a commutative cocommutative unital infinitesimal anti-symmetric
bialgebra on $\ol \H_*\Lambda$. In particular, the following relation holds 
\begin{equation}\label{eq:Sullivan-unital}
\lambda\mu = (\mu\otimes 1)(1\otimes\lambda) + (1\otimes\mu)(\lambda\otimes 1) - (\mu\otimes\mu)(1\otimes\lambda\eta \otimes 1),
\end{equation}
where $1$ denotes the identity map and $\eta$ the unit for the product $\mu$. 
\end{theorem}

See Definition~\ref{defi:secondary-unital} for the meaning of commutative and cocommutative unital infinitesimal anti-symmetric bialgebra, cf. also~\cite{CHO-algebra}. 

As shown in~\S\ref{sec:odd-spheres}, the term $\lambda\eta$ can be nonzero in some examples, and zero in others. This term always vanishes if $H_1M=0$, in which case the unital infinitesimal relation~\eqref{eq:Sullivan-unital} reduces to Sullivan's relation~\eqref{eq:Sullivan}. 

We prove Theorem~\ref{thm:main-intro} as Theorem~\ref{thm:main3} in the case $n\ge 3$ as a consequence of a more general statement of a symplectic nature. The case $n=1$, i.e., $M=S^1$, is proved by direct computation in~\S\ref{sec:odd-spheres}.
{\reftwo We expect that the statement also holds for $n=2$, although our current method does not cover this dimension.}

The previous discussion for free loops has a counterpart for based loops. The resulting algebraic structure is the same, except for commutativity and cocommutativity. See Theorem~\ref{thm:main-based}.

Besides~\cite{CHO-PD}, the setup considered in this paper is to the best of our knowledge the only instance in which the fundamental operations of string topology, i.e., the loop product and the loop coproduct, are brought together into a meaningful bialgebra structure. The previous ``extension by zero" of the coproduct on $\H_*(\Lambda,\Lambda_0)$ due to Hingston-Wahl~\cite{Hingston-Wahl} does not feature any meaningful compatibility with the product. In contrast, the unital infinitesimal bialgebra structure which we establish in this paper is a versatile tool for computations. As an example, we recall in~\S\ref{sec:odd-spheres} the full bialgebra structure in reduced loop homology for odd dimensional spheres. This was computed in~\cite{CHO-MorseFloerGH} using the unital infinitesimal relation from knowledge of the algebra structure and of the value of the coproduct on the algebra generators. This mechanism is very general and we summarize it in the following proposition.

\begin{proposition}
The coproduct on $\ol{\H}_*\Lambda$ is determined by its values on the unit and on a set of generators of the loop algebra $\ol{\H}_*\Lambda$. The coproduct on $\ol{\H}_*\Lambda$ moreover determines the coproduct on $\H_*(\Lambda,\Lambda_0)$.  
\end{proposition}

\begin{proof}
The first statement is proved directly by induction from the unital infinitesimal relation~\eqref{eq:Sullivan-unital}. The second statement follows from the fact that the coproduct on $\ol{\H}_*\Lambda$ extends the coproduct on $\H_*(\Lambda,\Lambda_0)$, i.e., the projection $\ol{\H}_*\Lambda\to \H_*(\Lambda,\Lambda_0)$ is a coalgebra map.
\end{proof}

Naef and Willwacher have independently constructed in~\cite{Naef-Willwacher} extensions of the loop coproduct to $\H_*\Lambda$ in the case $\chi(M)=0$. Their extensions depend on a trivialization of the Euler class, ours depend on the choice of a non-vanishing vector field {\reftwo (cf.~Remark~\ref{rem:cont-data}).} We believe that, when interpreted correctly, these two constructions should yield the same result. 

In this paper we prove all these structural results in a more general symplectic situation, and we deduce the statements in string topology from their symplectic counterparts. Namely, we exhibit the class of \emph{strongly $R$-essential Weinstein domains $W$} (Definition~\ref{defi:R-essential}) of dimension $2n$, where $R$ is a principal ideal domain.\footnote{{\refone This terminology has nothing to do with the notion of $R$-essential closed manifold as defined by Gromov in~\cite{Gromov-filling}, which stands for a different meaning: a closed $R$-orientable $n$-manifold is $R$-essential in the sense of Gromov if $0\neq f_*[M]\in H_n(B\pi_1(M);R)$, where $f:M\to B\pi_1(M)$ is the classifying map of its universal covering $\widetilde M$. In our case, $R$-essentiality means that the Weinstein domain admits a Morse function all of whose critical points of maximal index are homologically visible, and therefore ``essential''.}} 
For such Weinstein domains we can define the (degree shifted) \emph{reduced symplectic homology $\ol{S\H}_*(W) = \ol{SH}_{*+n}(W)$}. This class of Weinstein domains includes in particular disc cotangent bundles, in which case reduced symplectic homology is a model for reduced loop homology. Reduced symplectic homology features the same properties as those of reduced loop homology above: it lies between symplectic homology $S\H_*(W)$ and positive symplectic homology $S\H^{>0}_*(W)$, in the sense that the canonical map from the first to the second factors as $S\H_*(W)\to \ol{S\H}_*(W)\to SH_*^{>0}(W)$, and the following structural result holds. 

\begin{theorem}  \label{thm:main-symp-intro}
Let $W$ be a strongly $R$-essential Weinstein domain of dimension $2n\ge 6$. The pair of pants product on $S\H_*(W)$ descends canonically to a unital product $\mu$ on $\ol{S\H}_*(W)$, the secondary coproduct on $S\H_*^{>0}(W)$ extends to a coproduct $\lambda$ on $\ol{S\H}_*(W)$ (canonically if $H^{n-1}(W)=0$), and these two operations fit together into the structure of a commutative cocommutative unital infinitesimal anti-symmetric
bialgebra on $\ol{S\H}_*(W)$. In particular, relation~\eqref{eq:Sullivan-unital} holds.  
\end{theorem}

The non-canonical extensions of the coproduct from $S\H_*^{>0}(W)$ to $\ol{S\H}_*(W)$ are determined by a choice of \emph{continuation data $\cD$}, consisting of an \emph{$R$-essential Morse function} $K$ on $W$ (Definition~\ref{defi:R-essential}) and a homotopy interpolating between $-K$ and $K$. The non-uniqueness of the extension is controlled by so-called \emph{secondary continuation maps}, constructed from homotopies between such continuation data. These secondary continuation maps carry the same amount of information as the term $\lambda\eta$ in equation~\eqref{eq:Sullivan-unital} (Lemma~\ref{lem:cDopDlambda}).

The definition of reduced symplectic homology emanates from the canonical exact sequence of the pair $(W,\p W)$ described in~\cite{CO}. Denoting $S\H^*(W)=SH^{*+n}(W)$ the shifted symplectic cohomology of $W$, we have a long exact sequence 
\begin{equation} \label{eq:les-intro}
{\footnotesize
\xymatrix
@C=10pt
{
S\H_*(W,\p W) \simeq S\H^{-*-2n}(W) \ar[r]^-\eps & S\H_*(W) \ar[r]^-{\iota} & 
   S\H_*(\p W) \ar[r]^-{\pi}  & 
   S\H^{1-2n-*}(W) 
   \\
   }
}   
\end{equation}
 in which $\iota$ intertwines the products and $\pi$ intertwines the coproducts (Proposition~\ref{prop:les+}). \emph{Reduced symplectic homology and cohomology} are defined as  
 $$
 \ol{S\H}_*(W)= \coker \eps,\qquad \ol{S\H}^*(W)=\ker\eps. 
 $$
Note that the map $\eps$ factors through the 0-energy sector, so that these groups differ from $S\H_*(W)$ and $S\H^*(W)$ only by some finite dimensional factor. As a consequence of the definition we have a short exact sequence 
\begin{equation} \label{eq:short_ex_seq_intro}
{\small
\xymatrix
@C=12pt
{
   0 \ar[r] & (\ol{S\H}_*(W),\mu) \ar[r]^-{\iota} & 
   (S\H_*(\p W),\boldsymbol{\mu},\boldsymbol{\lambda}) \ar[r]^-{\pi} & 
   (\ol{S\H}^{1-2n-*}(W),{\mu}^\vee) \ar[r] & 0 
}
}
\end{equation}
The remarkable fact is that this exact sequence splits via maps which respect the extensions of the coproduct to $\ol{S\H}_*(W)$, and of the product to $\ol{S\H}^*(W)$ (Proposition~\ref{prop:splittings}). These maps are non-canonical to the same extent to which those extensions are non-canonical. In particular, the ambiguity is controlled by the same secondary continuation maps. 

A similar discussion can be led for the open string case, involving exact Lagrangians $L\subset W$ with Legendrian boundary $\p L\subset \p W$.
The open and closed case fit together in an open-closed theory of unital infinitesimal anti-symmetric bialgebras, akin to the graded open-closed TQFT discussed in~\cite{CHO-algebra}. We did not develop the axioms of such a theory and invite the 
interested reader to do so.  

The fundamental relations which define a unital infinitesimal anti-symmetric bialgebra stem from the study of boundary degenerations of Floer problems parametrized by certain remarkable polytopes. The polytopes which are used in the definitions of the operations $\mu$, $\lambda$, and in the proofs of the unital infinitesimal relation, of the anti-symmetry relation, as well as of coassociativity, are all of the form $K_d\times \Delta^k$, products between an associahedron $K_d$ of dimension $d-3$ and a simplex $\Delta^k$ of dimension $k$, and are endowed with boundary subdivisions that refine the product subdivision. (In the previous enumeration we see in order $K_3$, $K_3\times \Delta^1$, $K_4\times \Delta^1$ (in 2 flavors), $K_4\times \Delta^2$.) This is the beginning of a hierarchy which seems to coincide with that of ``assocoipahedra" of Poirier and Tradler~\cite{Poirier-Tradler}, as was observed and is being further investigated in the broader context of symplectic homology by Mazuir~\cite{Mazuir-in-progress}. This should lead on the one hand to a description of assocoipahedra as moduli spaces of curves decorated with Floer data, and on the other hand to the clarification of the definition and properties of unital infinitesimal anti-symmetric bialgebras up to homotopy.   

The plan of the paper is the following. In~\S\ref{sec:Ressential} we define $R$-essential Weinstein domains and study their first homological properties. In~\S\ref{sec:two_flavors} we define reduced symplectic homology $\ol{S\H}_*(W)$ and also symplectic homology relative to the continuation map $S\H_*(W;\im c)$. In~\S\ref{sec:coproductsSHWrelimc} we define coproducts on $S\H_*(W;\im c)$ and secondary continuation maps, and explain how the latter control the ambiguity in the definition of the former. In~\S\ref{sec:Weinstein_strong_R_essential} we define strongly $R$-essential Weinstein domains, which are better behaved algebraically. In~\S\ref{sec:bialgebra} we prove that reduced symplectic homology is a unital infinitesimal anti-symmetric bialgebra. In~\S\ref{sec:splittings} we discuss splittings of the exact sequence~\eqref{eq:short_ex_seq_intro}. Here we use in a crucial way the cone description of Rabinowitz Floer homology from~\cite{CO-cones}, and we show in~\S\ref{sec:unital-from-ass} how the unital infinitesimal relation~\eqref{eq:Sullivan-unital} is implied by the associativity of the product on the cone. In~\S\ref{sec:open_strings} we discuss the Lagrangian case. In~\S\ref{sec:examples_reduced} we specialize to the case of cotangent bundles and deduce our statements in string topology from their symplectic counterparts.  We illustrate these structures on the  example of reduced loop homologies of odd dimensional spheres.

{\bf Acknowledgements}.
The first author thanks Stanford University, Institut Mittag--Leffler, and the Institute for Advanced Study for their hospitality over the duration of this project.
The second author was partially funded by ANR grants ENUMGEOM 18-CE40-0009, COSY 21-CE40-0002, and by a Fellowship of the University of Strasbourg Institute for Advanced Study (USIAS) within the French national programme "Investment for the future" (IdEx-Unistra). 
We thank the anonymous referees for their critical comments on the notion of $R$-essentiality.

\section{$R$-essential Morse functions} \label{sec:Ressential}

Let $W$ be a compact manifold with boundary of dimension $\dim W=2n$. A {\em defining function} on $W$ is a smooth function $K:W\to\R$ having the boundary $\p W$ as its regular maximum set.

Let $R$ be a principal ideal domain.

\begin{definition} \label{defi:R-essential} 
An \emph{$R$-essential Morse function} on $W$ is a defining Morse function $K:W\to\R$ without critical points of index $>n$ such that the number of its index $n$ critical points is equal to the rank of $H_n(W;R)$. 
We say that $W$ is {\em $R$-essential} if it admits an $R$-essential Morse function $K$.
\end{definition}

Note that a necessary condition for $W$ to be $R$-essential is that $H_k(W;\Z)=0$ for all $k>n$, and $H_n(W;R)$ and $H_{n-1}(W;R)$ are free $R$-modules. This follows from the fact that the singular homology groups admit a description as Morse homology groups.  

\begin{remark}
(a) For $R=\Z$, a necessary condition for $W$ to be $\Z$-essential is that $H_k(W;\Z)=0$ for all $k>n$, and $H_n(W;\Z)$ and $H_{n-1}(W;\Z)$ are free $\Z$-modules. By the h-cobordism theorem, this condition is also sufficient if $n\geq 3$ and $W$ is simply connected~\cite[Theorem~6.1]{Smale}, while in the non-simply connected case it is in general not sufficient~\cite[Proposition 1.4]{Cieliebak-Eliashberg15}. 

(b) Suppose that $R$ is a field of characteristic $p>0$. Then a necessary condition for $W$ to be $R$-essential is that $H_k(W;\Z)=0$ for $k>n$, $H_n(W;\Z)$ is free, and $H_{n-1}(W;\Z)$ has only $p$-torsion. Indeed, given an $R$-essential Morse function, the Morse differential in degree $n$ over $\Z$ must vanish upon tensoring with $R$, so $H_{n-1}(W;\Z)$ has only $p$-torsion. By the h-cobordism theorem, this condition is also sufficient if $n\ge 3$ and $W$ is simply connected: By~\cite[Theorem~6.1]{Smale} the manifold $W$ admits a defining Morse function with no critical points of index $>n$ and the number of critical points of index $n$ equal to $\mathrm{rk}_\Z H_n(W;\Z)$ plus the number of torsion generators of $H_{n-1}(W;\Z)$. If $H_{n-1}(W;\Z)$ has only $p$-torsion the universal coefficient theorem shows that this is the same as $\mathrm{rk}_R H_n(W;R)$ q.e.d. Again, this condition is in general not sufficient in the non-simply connected case.
\end{remark}
 
We will be interested in $R$-essentiality for Weinstein domains. 
Recall from~\cite{Cieliebak-Eliashberg-book} that a {\em Weinstein domain} $(W,\omega,X,\phi)$ is a compact manifold $W$ with boundary equipped with a symplectic form $\omega$, a defining generalized Morse function $\phi:W\to \R$, and a vector field $X$ which is Liouville for $\omega$ and gradient-like for $\phi$. Weinstein domains are considered up to homotopy of $(\om,X,\phi)$ with these properties. After such a Weinstein homotopy, we may assume that the function $\phi$ is Morse, i.e. it has no embryonic critical points. Then $\phi$ has no critical points of index $>n$, so $R$-essentiality of $\phi$ reduces to the number of its index $n$ critical points being equal to the rank of $H_n(W;R)$. Let us emphasize, however, that in all our constructions involving $R$-essentiality we can use instead of $\phi$ an $R$-essential Morse function $K$ which is unrelated to the Weinstein structure, and whose existence is a purely topological question as {\reftwo discussed} above.

\begin{remark}
A theorem of O.~Lazarev~\cite{Lazarev} implies that, on an $R$-essential Weinstein domain $W$ of dimension $2n\geq 6$ with $H_n(W;R)\neq 0$, one can actually find a Weinstein--Morse function $\phi$ as above which is $R$-essential. 
\end{remark}

\begin{example} Examples of $\Z$-essential Weinstein domains include the following compact manifolds (with {\em any} Weinstein structure): 

(i) disc cotangent bundles of closed orientable manifolds (in the non-orientable case they are $\Z/2$-essential);

(ii) plumbings of cotangent bundles of closed orientable manifolds (without the orientability assumption they are $\Z/2$-essential);

(iii) Milnor fibers of isolated singularities;

(iv) subcritical Weinstein domains.  
\end{example}

{\bf Convention on coefficients. }
Throughout this paper we use the principal ideal domain $R$ as the ring of coefficients. We will not specify it in the notation of homology or cohomology groups, unless we want to place special emphasis on it. For example $H_n(W)$ stands for $H_n(W;R)$.

Let us return to a general $2n$-dimensional compact manifold $W$ with boundary.
In the sequel we denote by $(MC_*,\p)$ the Morse chain complex with coefficients in $R$ and by $(MC^*,\delta)$ the Morse cochain complex with coefficients in $R$ (associated to a pair consisting of a defining Morse function $K$ and a Morse-Smale gradient-like vector field). The condition of $K$ being $R$-essential can be rephrased as an isomorphism 
$$
H_n(W)\simeq MC_n(K),
$$
which is equivalent to 
$
\p|_{MC_n(K)}=0.
$
This is in turn equivalent to 
\begin{equation} \label{eq:R-essential-cochains}
\delta|_{MC^{n-1}(K)}=0,\quad \mbox{or equivalently,}\quad \delta|_{MC^n(-K)}=0.
\end{equation}
These conditions do not depend on the choice of the Morse-Smale gradient-like vector field for $K$. 

Let $K$ and $K'$ be $R$-essential Morse functions and choose Morse-Smale gradient-like vector fields for $K$ and $K'$ in order to form the corresponding Morse complexes. By \emph{continuation data from $-K$ to $K'$}, denoted 
$$
\cD=\cD_{-K,K'},
$$ 
we mean the choice of a regular pair consisting of a homotopy from $-K$ to $K'$ via functions which are linear near the boundary with increasing slopes, and of a homotopy between the gradient-like vector fields for $-K$ and $K'$. Such a choice of continuation data determines a \emph{continuation chain map}
$$
c=c_{\cD}:MC^*(-K)\to MC^*(K'). 
$$

{\bf Terminology convention.} In the sequel, whenever we speak of a \emph{chain level continuation map} between Morse or Floer complexes, we understand not only the algebraic data of the continuation chain map, but also the geometric continuation data consisting of a homotopy of functions and gradient-like vector fields, respectively of Hamiltonians and almost complex structures.

\begin{lemma} \label{lem:R-essential-continuation-maps-are-equal}
Given $R$-essential Morse functions $K,K'$ together with Morse-Smale gradient-like vector fields, any two choices of continuation data determine equal continuation maps $MC^*(-K)\to MC^*(K')$.  
\end{lemma} 

\begin{proof} Consider two continuation maps $c$ and $\tilde c$. We have 
$$
\tilde c-c=\delta_{K'}h+h\delta_{-K}
$$ 
for some degree $-1$ map $h:MC^*(-K)\to MC^*(K)$ determined by a homotopy of continuation data. Now $MC^*(-K)$ is supported in degrees $n,\dots,2n$, and $MC^*(K')$ is supported in degrees $0,\dots,n$. Therefore $c$ and $\tilde c$ can only be nonzero in degree $n$. We evaluate the relation $\tilde c-c=\delta_{K'}h+h\delta_{-K}$ on $MC^n(-K)$ and apply~\eqref{eq:R-essential-cochains} in order to obtain 
$$
\tilde c -c =0
$$
in degree $n$. Thus $\tilde c=c$.
\end{proof}

\begin{remark}
The equality $\tilde c=c$ implies in particular that any chain homotopy $h$ as above is a chain map. This is the \emph{secondary continuation map} which we will further explore below, and it \emph{does} depend on choices. 
\end{remark}

\begin{proposition}\label{prop:int-form}
Let $K$ be an $R$-essential Morse function on $W$ with $\dim W=2n$. 
Then the continuation map $c:MC^*(-K)\to MC^*(K)$ is only nonzero in degree $n$, where it coincides via the identifications $MC^n(\pm K)\equiv MH^n(\pm K)$ with the map induced on homology
{\small
\begin{align*}
  c_*:H_n(W)  = MH^n(-K) \to  MH^n(K)=H^n(W) = \Hom_R(H_n(W),R).
\end{align*}}
Moreover, the latter is given by $\la c_*a,b\ra = a\bullet b$ for the intersection pairing $\bullet$ on $H_n(W)$.
\end{proposition}

\begin{proof}
That $H^n(W)$ is identified with $\Hom_R(H_n(W),R)$ follows from the Universal Coefficient Theorem for cohomology with coefficients in the principal ideal domain $R$~\cite[Theorem~5.5.3]{Spanier}, namely the existence of a short exact sequence $0\to \Ext_R(H_{n-1}(W),R)\to H^n(W)\to \Hom_R(H_n(W),R)\to 0$ which splits non-canonically. Using that the $R$-module $H_{n-1}(W)$ is free we obtain $\Ext_R(H_{n-1}(W),R)=0$ and therefore the canonical map $H^n(W)\to \Hom_R(H_n(W),R)$ is an isomorphism.
That $c$ lives only in degree $n$ was shown in the proof of Lemma~\ref{lem:R-essential-continuation-maps-are-equal}, and the same proof shows that it coincides in degree $n$ with the map $c_*$. Consider now two index $n$ critical points $p,q$ of $K$. By~\eqref{eq:R-essential-cochains} they are closed and define Morse homology classes $[p],[q]\in MH_*(W)$. By definition of the continuation map, $\la c_*[p],[q]\ra$ is the intersection number of their stable manifolds (with respect to some gradient-like vector field), which is the intersection number of the  corresponding singular homology classes. 
\end{proof}

Next, we study the cokernel $MC^*(K)/\im c$ of a continuation map.

\begin{lemma}\label{lem:R-essential-chain-homotopy-mod-imc} Consider two $R$-essential Morse functions $K_1,K_2$ together with Morse-Smale gradient-like vector fields, and continuation maps $c_i:MC^*(-K_i)\to MC^*(K_i)$, $i=1,2$. Any two continuation maps 
  $$
MC^*(K_1)\stackrel{c,\, c'}\longrightarrow MC^*(K_2)
$$ 
induce chain homotopic maps 
$$
MC^*(K_1)/\im c_1 \stackrel{[c],\, [c']}\longrightarrow MC^*(K_2)/\im c_2.
$$ 
\end{lemma}

\begin{proof} Fix a continuation map $MC^*(-K_1)\to MC^*(-K_2)$. The diagram
$$
\xymatrix{
MC^*(K_1)\ar[r]^-{c,\, c'}& MC^*(K_2) \\
MC^*(-K_1) \ar[u]^-{c_1}\ar[r]& MC^*(-K_2) \ar[u]_-{c_2}
}
$$
is commutative by Lemma~\ref{lem:R-essential-continuation-maps-are-equal}.
This implies that $c(\im c_1)=c'(\im c_1)\subset \im c_2$, {\reftwo so we get} induced maps $MC^*(K_1)/\im c_1 \stackrel{[c],\, [c']}\longrightarrow MC^*(K_2)/\im c_2$. 

To prove that these maps are chain homotopic we start from a chain homotopy $c'-c=\delta_{K_2}h+h\delta_{K_1}$, where $h:MC^*(K_1)\to MC^*(K_2)$ has degree $-1$. By $R$-essentiality this implies 
$$
c'=c \qquad \mbox {on }\qquad MC^n(K_1)
$$
and 
$$
c'-c=\delta_{K_2}h\qquad \mbox{on }\qquad MC^{n-1}(K_1).
$$ 
Define now $[h]:MC^*(K_1)/\im c_1\to MC^*(K_2)/\im c_2$ of degree $-1$ by 
$$
[h]=h \qquad \mbox{ on }\qquad MC^{*<n}(K_1)
$$
and 
$$
[h]=0 \qquad \mbox{on }\qquad MC^n(K_1)/\im c_1.
$$
We then have  
$$
[c']-[c]=\delta_{K_2}[h]+[h]\delta_{K_1}.
$$
Indeed, this holds on $MC^{*<n-1}(K_1)$ from the homotopy relation $c'-c=\delta_{K_2}h+h\delta_{K_1}$. It holds on $MC^{n-1}(K_1)$ from the relation $c'-c=\delta_{K_2}h$. And it holds on $MC^n(K_1)/\im c_1$ from the relation $c'=c$. 
\end{proof}

\begin{corollary} \label{cor:R-essential-chain-homotopy-mod-imc}
Given two $R$-essential Morse functions $K_1,K_2$ together with Morse-Smale gradient-like vector fields, and given continuation maps $c_i:MC^*(-K_i)\to MC^*(K_i)$, $i=1,2$, there is a canonical isomorphism 
\begin{equation} \label{eq:R-essential-mod-c}
H^*(MC^*(K_1)/\im c_1)\simeq H^*(MC^*(K_2)/\im c_2).
\end{equation}
More precisely, any continuation map $c:MC^*(K_1)\to MC^*(K_2)$ induces a chain homotopy equivalence 
$$
MC^*(K_1)/\im c_1\simeq MC^*(K_2)/\im c_2,
$$
and the induced isomorphism~\eqref{eq:R-essential-mod-c} does not depend on $c$.  
\end{corollary}

\begin{proof} In view of Lemma~\ref{lem:R-essential-chain-homotopy-mod-imc} it remains to argue that the canonical map 
$[c]:H^*(MC^*(K_1)/\im c_1) \longrightarrow H^*(MC^*(K_2)/\im c_2)$ is an isomorphism. This is a consequence of the fact that it is induced by a continuation map, and continuation maps are stable under composition. We can thus compose $[c]$ with a continuation map $H^*(MC^*(K_2)/\im c_2)\longrightarrow H^*(MC^*(K_1)/\im c_1)$, and this composition must be the identity. 
\end{proof}

\begin{remark} \label{rmk:cokernel-vs-homotopy-cokernel}
As a consequence of the previous discussion,
the \emph{co\-kernel} of the continuation map $c:MC^*(-K)\to MC^*(K)$ associated to an $R$-essential Morse function $K$ is independent of $K$ up to chain homotopy equivalence.
This comes as a surprise in as much as continuation maps are only defined up to homotopy and thus the only homotopically meaningful notion in general would be that of \emph{homotopy cokernel}, i.e. the cone of $c$. 
\end{remark}

\section{Two flavors of reduced symplectic homology} \label{sec:two_flavors}

{\reftwo In this section we introduce two variants of reduced symplectic homology and study their relation. Here and in the rest of the paper, for symplectic homology we follow the conventions in~\cite{CO}.}

\subsection{Reduced symplectic homology $\ol{SH}_*(W)$} \label{sec:reduced_homology}

Let $W$ be a Liouville domain and consider the following part of the symplectic homology exact sequence of the pair $(W,\p W)$~\cite{CO},  
$$
SH_*(W,\p W)\cong SH^{-*}(W)\stackrel{\eps_*}\longrightarrow SH_*(W)\stackrel {\iota_*}\longrightarrow SH_*(\p W).
$$

\begin{definition}
\emph{Reduced symplectic homology} is defined as
$$
\ol{SH}_*(W)=\coker \eps_* .
$$ 
\end{definition}

\begin{proposition}
Reduced symplectic homology $\ol{SH}_*(W)$ carries a unital algebra structure induced by that of $SH_*(W)$ and the projection  
$$
SH_*(W)\to \ol{SH}_*(W)
$$ 
is a map of unital algebras. Moreover, the canonical map $SH_*(W)\to SH_*^{>0}(W)$ factors through reduced symplectic homology  
$$
SH_*(W)\to \ol{SH}_*(W)\to SH_*^{>0}(W).
$$
(Note however that $SH_*^{>0}(W)$ is \emph{not} a ring.)
\end{proposition}

\begin{proof}
Since $\iota_*$ is an algebra map~{\refone \cite[Theorem~10.2.e]{CO}} we have that $\ker \iota_*=\im \eps_*\subset SH_*(W)$ is an ideal.  As a consequence $\ol{SH}_*(W) = SH_*(W)/\im \eps_*$ carries an induced unital algebra structure, and this proves the first part of the proposition. 

For the second part we use the commutative diagram (see~\cite{CO})
$$
\xymatrix{
SH_*^{>0}(W,\p W)=0 \ar[r] & SH_*^{>0}(W)\ar[r]^-\simeq & SH_*^{>0}(\p W) \\
SH_*(W,\p W)\ar[u] \ar[r]^-{\eps_*}& SH_*(W) \ar[u]_-{j_*} \ar[r]^-{\iota_*}& SH_*(\p W) \ar[u]
}
$$
which implies that $\im \eps_*\subset \ker j_*$.  
\end{proof}

\subsection{Symplectic homology relative to the continuation map $SH_*(W;\im c)$} \label{sec:SHrelimc}

We explain in this subsection how the Morse theoretic construction of the cokernel of the continuation map in~\S\ref{sec:Ressential} extends to define a symplectic homology group relative to the image of the continuation map. 

Let $W$ be an $R$-essential Weinstein domain. We consider Hamiltonians $H:\wh W\to\R$ with the following properties: 
\begin{itemize}
\item on $W$ they are $C^2$-small $R$-essential Morse functions, which in a neighborhood of $\{1\}\times \p W$ are linear with respect to $r\in(0,\infty)\times \p W$ of constant slope, denoted $\eps_0$; 
\item on $[1,\infty)\times \p W$ they are strictly increasing functions of $r$, linear of constant and noncritical slope $\mu$ away from the boundary, strictly convex in the region where the slope varies from $\eps_0$ to $\mu$. 
\end{itemize}
We denote $\cH_{R-ess}$ the class of such Hamiltonians, which we call \emph{$R$-essential Hamiltonians}.
Given such a Hamiltonian $H$ we denote $K$ the $R$-essential Morse function obtained by restricting $H$ to $W$. We will sometimes view $K$ as being defined on $\wh W$ by extending it as a linear function of $r$ with slope $\eps_0$. We refer to $K$ as being the \emph{Morse truncation of $H$}. 

By convention we only consider homotopies $K\to H$ which are constant in $W$ and strictly increasing functions of $r$ in $[1,\infty)\times\p W$. These induce inclusions $FC_*(K)\to FC_*(H)$ {\reftwo of Floer chain complexes. Here $FC_*(H)$ also depends on generic choices of a time-dependent perturbation $H_t$ of $H$ and a time-dependent compatible almost complex structure $J_t$, but we will not include these in the notation. In this context, \emph{continuation data} $\cD=\cD_{-K,K'}$ from $-K$ to $K'$ is a pair $(K_s,J_{s,t})$ consisting of functions $K_s$ and time-dependent compatible almost complex structures $J_{s,t}$ parametrized by $s\in\R$ which agree with $(-K,J_t)$ for 
large $s$ and with $(K',J_t')$ for 
small $s$. Composing it with homotopies $-H\to-K$ and $K'\to H'$ as above and time-dependent perturbations, $\cD$ can also be viewed as a pair $(H_{s,t},J_{s,t})$ from $(-H_t,J_t)$ to $(H_t',J_t')$. A choice of continuation data determines a \emph{continuation chain map}
$$
c=c_{\cD}:FC_*(-K)\to FC_*(H'). 
$$
Note that $\cD$ induces continuation data in the sense of~\S\ref{sec:Ressential} and thus a chain map $FC_*(-K)= MC^{n-*}(-K)\to FC_*(K')= MC^{n-*}(K')$. 

\begin{remark}\label{rem:cont-data}
In the above continuation data for Floer complexes, the almost complex structures can be replaced by inhomogeneous terms that are added to the Floer equation. This variant is used in~\cite{CHO-MorseFloerGH}, and all our results continue to hold (with identical proofs) in this framework. Consider now the disc cotangent bundle $W=D^*M$ of a closed oriented $n$-manifold $M$. In this case, as explained in~\cite[\S3.1]{CHO-MorseFloerGH}, vector fields (or equivalently $1$-forms) on $M$ with nondegenerate zeroes give rise to inhomogeneous terms that can be used for continuation data. If $M$ has Euler characteristic zero we can use vector fields without zeroes. 
\end{remark}
}

\begin{lemma}
Let $W$ be an $R$-essential Weinstein domain. 

(i) Given an $R$-essential Hamiltonian $H$ with Morse truncation $K$, any two continuation maps 
$$
c:FC_*(-K)\to  FC_*(H)
$$
induced by homotopies which are $C^2$-small on $W$ are equal. 

(ii) Given $R$-essential Hamiltonians $H_1,H_2$ with the same asymptotic slope, and given continuation maps $c_i:FC_*(-K_i)\to FC_*(H_i)$, $i=1,2$ as above, any homotopy between $H_1$ and $H_2$ which is $C^2$-small on $W$ induces a chain homotopy equivalence 
$$
FC_*(H_1)/\im c_1\simeq FC_*(H_2)/\im c_2. 
$$
\end{lemma}

\begin{proof}
The proof is similar to that of Lemma~\ref{lem:R-essential-chain-homotopy-mod-imc} and Corollary~\ref{cor:R-essential-chain-homotopy-mod-imc}. We omit the details.
\end{proof}

\begin{corollary} \label{cor:FC*Himc}
Let $W$ be an $R$-essential Weinstein domain. Given an $R$-essential Hamiltonian $H$ with Morse truncation $K$ and a continuation map $c:FC_*(-K)\to FC_*(H)$ as above, the chain homotopy type of 
$$
FC_*(H;\im c):=FC_*(H)/\im c
$$
depends only on the asymptotic slope of $H$. \qed
\end{corollary}

\begin{definition}
Let $W$ be an $R$-essential Weinstein domain. The \emph{symplectic homology group of $W$ relative to the continuation map} is defined to be 
$$
SH_*(W;\im c) := \lim\limits_{\stackrel{\longrightarrow}{\scriptsize \left.\begin{array}{c}H\in\cH_{R-ess}\\\mathrm{slope}(H)\to\infty \end{array}\right.}} FH_*(H;\im c).
$$
\end{definition}

\begin{remark} \label{rmk:cokernel-vs-homotopy-cokernel-2} As a follow-up to Remark~\ref{rmk:cokernel-vs-homotopy-cokernel}, we note that the cokernel of $c$ gives rise to symplectic homology rel $\im c$, 
$SH_*(W;\im c)$, whereas the homotopy cokernel, i.e. the cone of $c$, gives rise to non-negative symplectic homology, $SH_*^{\ge 0}(\p W)$. 
\end{remark}

\subsection{Comparison between $\ol{SH}_*(W)$ and $SH_*(W;\im c)$}\label{sec:olSHvsSHrelimc}

Let $W$ be an $R$-essential Weinstein domain. 

Given an $R$-essential Morse function $K$, an $R$-essential Hamiltonian $H$, and a continuation map 
$c:FC_*(-K)\to FC_*(H)$, the composition 
$$
FC_*(-K)\stackrel{c}\longrightarrow FC_*(H)\stackrel\pi\longrightarrow FC_*(H)/\im c
$$ 
is zero, and thus the composition 
$$
FH_*(-K)\stackrel{c_*}\longrightarrow FH_*(H)\stackrel{\pi_*}\longrightarrow H_*(FC_*(H)/\im c). 
$$
is also zero. We infer a map $\coker c_*\to H_*(FC_*(H)/\im c)$. In the limit over $H$ we obtain a map 
\begin{equation} \label{eq:map_from_reduced_to_modimc}
\ol{SH}_*(W)\to SH_*(W;\im c).
\end{equation}

\begin{proposition}\label{prop:olSH_relSH_sufficient_condition} Let $W$ be an $R$-essential Weinstein domain, and consider the factorization 
$$
SH_*(W,\p W)\longrightarrow SH_*^{=0}(W,\p W) \stackrel{c_*^{=0}}\longrightarrow SH_*^{=0}(W)\longrightarrow SH_*(W)
$$ 
of the map $c_*:SH_*(W,\p W)\to SH_*(W)$. 

The map $\ol{SH}_*(W)\to SH_*(W;\im c)$ is an isomorphism if 
$$
\im c_*^{=0}\to SH_0(W)
$$ 
is injective, and in particular if $H^n(W)\to SH_0(W)$ is injective. 
\end{proposition}

\begin{proof} Let $K$ be an $R$-essential Morse function as above, let $H$ be an $R$-essential Hamiltonian, and let $c:FC_*(-K)\to FC_*(H)$ be a continuation map. We have a commutative diagram 
$$
\xymatrix{
FH_{*+1}^{>0}(H) \ar[d] \ar[r]_{\mathrm{Id}}^= & H_{*+1}(FC_*^{>0}(H)) \ar[d]\\
FH_*(K)/\im c_* \ar[d]\ar[r]_{f_{K*}}^\simeq & H_*(FC_*(K)/\im c) \ar[d]\\
FH_*(H)/\im c_* \ar[d]\ar[r]_{f_{H*}}&  H_*(FC_*(H)/\im c) \ar[d]\\
FH_*^{>0}(H) \ar[d]\ar[r]_{\mathrm{Id}}^= & H_*(FC_*^{>0}(H))\ar[d]\\
FH_{*-1}(K)/\im c_* \ar[r]_{f_{K*}}^\simeq & H_{*-1}(FC_*(K)/\im c) 
}
$$
The maps $f_{K*}$ and $f_{H*}$ are the fixed-slope analogues of~\eqref{eq:map_from_reduced_to_modimc}. The right column is an exact sequence. That the map $f_{K*}$ is an isomorphism follows from the assumption of $R$-essentiality for $K$ and the resulting conditions~\eqref{eq:R-essential-cochains}. We are seeking conditions under which $f_{H*}$ is an isomorphism as well. By the 5-lemma, it is enough that the left column be an exact sequence. 

We are thus seeking a condition under which the bottom line in the diagram below is exact. Note that the middle line in this diagram is exact. 

$$
\xymatrix
@C=20pt
{
& FH_*(-K) \ar@{=}[r] \ar[d]^{c_*^{=0}}& FH_*(-K) \ar[d]^{c_*}& \\
FH_{*+1}^{>0}(H) \ar[r]^-A\ar@{=}[d] & FH_*^{=0}(H) \ar[r]^-B\ar@{->>}[d]& FH_*(H)\ar[r]^-C\ar@{->>}[d]& FH_*^{>0}(H) \ar@{=}[d] \\
FH_{*+1}^{>0}(H) \ar[r]^-{A'} & FH_*^{=0}(H)/\im c_*^{=0}\ar[r]^-{B'}& FH_*(H)/\im c_* \ar[r]^-{C'}& FH_*^{>0}(H) 
}
$$

The maps $A,B,C$ are the canonical maps from the action truncation exact sequence, and the maps $A',B',C'$ are induced by $A,B,C$. Exactness of the bottom line at $FH_*^{=0}(H)/\im c_*^{=0}$ and at $FH_*(H)/\im c_*$ is automatic. Exactness at $FH_*^{>0}(H)$ follows from our assumption: we have $\im C'=\im C$ and $\ker A'=A^{-1}(\im c_*^{=0})$. Thus we have exactness if and only if $A^{-1}(\im c_*^{=0})=\ker A=A^{-1}(0)$, i.e. $\im A\cap \im c_*^{=0}=0$, i.e. $\ker B\cap \im c_*^{=0}=0$. This happens for all $H$ if and only if $\im c_*^{=0}\to SH_*(W)$ is injective. 
\end{proof}

\section{Coproducts on $SH_*(W;\im c)$} \label{sec:coproductsSHWrelimc}

\subsection{Coproduct defined by a choice of continuation data} 

Let $W$ be an $R$-essential Weinstein domain. We explain in this subsection
that symplectic homology relative to the continuation map constitutes a natural domain of definition for the secondary coproduct.

In this subsection we fix 
\begin{itemize}
\item an $R$-essential Morse function $K$ and a Morse-Smale gradient-like vector field; 
\item continuation data $\cD=\cD_{-K,K}$ from $-K$ to $K$, which determines a continuation map $c_\cD :FC_*(-K)\to FC_*(K)$. \end{itemize}

We consider $R$-essential Hamiltonians $H$ which extend $K$ and continuation maps $FC_*(-K)\to FC_*(H)$ which factor as 
$$
FC_*(-K)\stackrel{c_\cD}\longrightarrow FC_*(K)\longrightarrow FC_*(H),
$$
where $FC_*(K)\to FC_*(H)$ is induced by a homotopy which is constant on $W$, and is therefore an inclusion. We denote these continuation maps also by $c_\cD:FC_*(-K)\to FC_*(H)$. Note that the map $FC_*(K)\to FC_*(2H)$ induced by composing a homotopy $K\to 2K$ on $W$ with a homotopy $2K\to 2H$ constant on $W$ is also naturally an inclusion. We assume in the sequel without loss of generality that the slope of $K$ at the boundary is smaller than $1$ (otherwise we increase the weight of the Hamiltonian at the outputs).  

We consider coproducts 
$$
\boldsymbol{\lambda}=\boldsymbol{\lambda}_\cD:FC_*(H)\to FC_*(2H)\otimes FC_*(2H)
$$
defined via $1$-parameter families which, at the endpoints, factor as a primary coproduct with outputs in $FC_*(-K)$ and $FC_*(2H)$, followed by a continuation map $c_\cD:FC_*(-K)\to FC_*(H)$.  
See Figure~\ref{fig:Coproduct-SHrelimc}.
In particular the induced map $FC_*(H)\to FC_*(2H;\im c_\cD)^{\otimes 2}$ is a chain map. Passing to the limit over $R$-essential Hamiltonians $H$ which extend $K$ we obtain a map  
$$
\boldlambda_\cD:SH_*(W)\to SH_*(W;\im c)^{\otimes 2}.
$$

\begin{figure}
\begin{center}
\includegraphics[width=.8\textwidth]{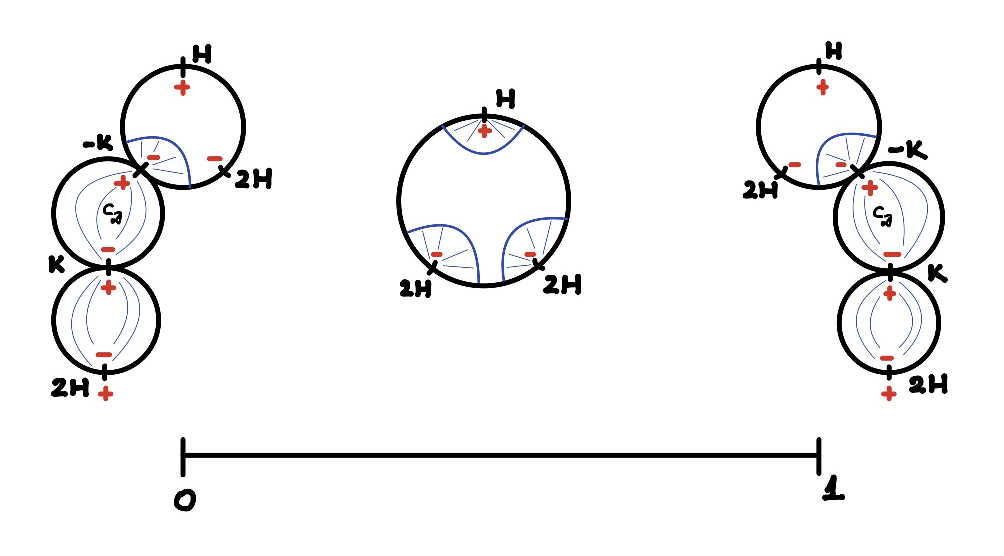}
\caption{Definition of the coproduct on $SH_*(W;\im c)$.}
\label{fig:Coproduct-SHrelimc}
\end{center}
\end{figure}

We will now investigate conditions under which this map descends to $SH_*(W;\im c)$ or $\ol{SH}_*(W)$.

\begin{lemma} \label{lem:lambda-imc}
  Assume that the dimension of the $R$-essential Weinstein domain is $2n\ge 6$.
  Then the chain level map $\boldlambda_\cD$ satisfies 
$$
\boldlambda_\cD(\im c_\cD)\subset \im c_\cD\otimes FC_*(2H) + FC_*(2H)\otimes \im c_\cD,
$$ 
i.e. the pair $(\im c_\cD,\im c_\cD)$ is a coideal pair for the map $\boldlambda_\cD$. 
\end{lemma}

\begin{proof}
For action reasons $\boldlambda_\cD(\im c_\cD)\subset FC_*^{=0}(2H)\otimes FC_*^{=0}(2H)$, and we have natural identifications $FC_*^{=0}(2H)=FC_*(2K)\simeq FC_*(K)=MC^{n-*}(K)$. {\refone (See~\cite[Proposition~2.10]{CO} for the last identification.)} Note that $\boldsymbol{\lambda}_\cD c_\cD$ appears as a term in the boundary of an operation $\boldbeta=\boldbeta_\cD$ of degree $-2$ (with respect to the Conley-Zehnder grading), and the other terms in the boundary lie in $\im c_\cD\otimes FC_*(2H) + FC_*(2H)\otimes \im c_\cD$ (cf.~\cite{CO-cones}). See Figure~\ref{fig:Beta-coproduct}, {\reftwo in which the arrows go upwards in order to signify that we represent a Morse {\em cochain} complex}.
\begin{figure}
\begin{center}
\includegraphics[width=.8\textwidth]{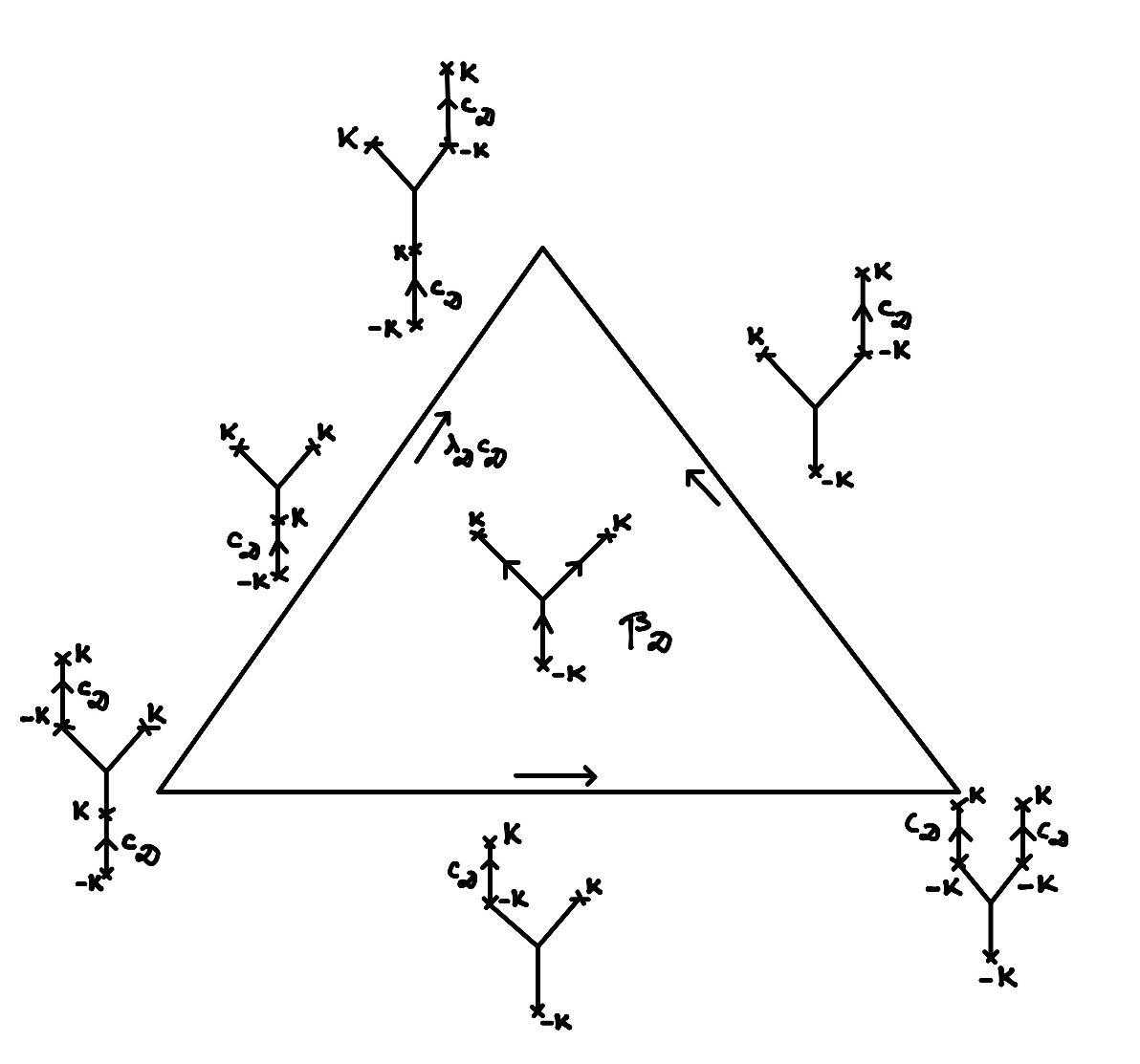}
\caption{The operation $\boldbeta_\cD$.}
\label{fig:Beta-coproduct}
\end{center}
\end{figure}

The operation $\boldsymbol{\beta}$ counts $2$-parameter families of cohomological $Y$-graphs with one input $x$ in $MC^*(-K)$ and two outputs $a,b$ in $MC^*(K)$. The condition for $0$-dimensional moduli spaces is 
$\ind_K a + \ind_K b +(2n - \ind_{-K}x) - 4n +2 =0$, which is equivalent to 
$$
\ind_K a + \ind_K b = -2 + 2n + \ind_{-K}x.
$$
The left hand side of this equality is $\leq 2n$, while the right hand side is $\geq -2+3n$. The equality is therefore never satisfied if $n\ge 3$, which implies $\boldsymbol{\beta}=0$ and thus the lemma.   
\end{proof}

\begin{proposition} \label{prop:coproduct_on_SH_rel_imc} Assume that the dimension of the $R$-essential Weinstein domain is $2n\ge 6$. A choice of $R$-essential Morse function $K$, a choice of Morse-Smale gradient-like vector field, and a choice of continuation data $\cD$ from $-K$ to $K$, induce a degree $-n+1$ coproduct 
$$
\boldsymbol{\lambda}_\cD:SH_*(W;\im c)\to SH_*(W;\im c)\otimes SH_*(W;\im c).
$$  
\end{proposition}

\begin{proof}
This follows from Lemma~\ref{lem:lambda-imc} by passing to the limit over $R$-essential Hamiltonians which extend $K$.  
\end{proof}

The dimensional assumption in the previous Proposition can be discarded if one wishes to only reduce the domain to $\ol{SH}_*(W)$. 

\begin{proposition} \label{prop:lambda_olSH_SHimc} A choice of $R$-essential Morse function $K$, a choice of Morse-Smale gradient-like vector field, and a choice of continuation data $\cD$ from $-K$ to $K$, induce a degree $-n+1$ linear map 
$$
\boldsymbol{\lambda}_\cD:\ol{SH}_*(W)\to SH_*(W;\im c)\otimes SH_*(W;\im c).
$$  
\end{proposition}

\begin{proof}
Consider the chain map
$$
\boldsymbol{\lambda}_\cD:FC_*(H)\to FC_*(2H;\im c_\cD)^{\otimes 2}.
$$
In the proof of Lemma~\ref{lem:lambda-imc} we needed to prove that this map vanishes on $\im c_\cD$. We now need to prove something weaker, namely that the map $\boldsymbol{\lambda}_{\cD\, *}$ induced in homology vanishes on $\im c_*$. Equivalently, it is enough to evaluate $\boldsymbol{\lambda}_\cD\circ c_\cD$ on an element $x\in FC_*(-K)$ which is a \emph{cycle} and show that it is a \emph{boundary} in $FC_*(2H;\im c_\cD)^{\otimes 2}$.

Denote 
$$
I_\cD= \im c_\cD\otimes FC_*(2H)+ FC_*(2H)\otimes \im c_\cD.
$$
The analysis of the moduli space defining the operation $\boldsymbol{\beta}=\boldsymbol{\beta}_\cD$ from the proof of Lemma~\ref{lem:lambda-imc} shows that we have 
\begin{equation}\label{eq:lambdared_beta}
\boldsymbol{\lambda}_\cD\circ c_\cD = [\p,\boldsymbol{\beta}_\cD]\, \mbox{ mod }\, I_\cD.
\end{equation}
Evaluating this on a cycle $x\in FC_*(-K)$ we obtain  
\begin{align*}
\boldsymbol{\lambda}_\cD\circ c_\cD(x) & = \p \boldsymbol{\beta}_\cD(x) \, \mbox{ mod }\, I_\cD \\
& = \p\big( \boldsymbol{\beta}_\cD(x) \, \mbox{ mod }\, I_\cD \big),
\end{align*}
which is a boundary in $FC_*(2H;\im c_\cD)^{\otimes 2}$. 

In the limit over $R$-essential Hamiltonians $H$ we obtain the map 
$$
\boldlambda_\cD:\ol{SH}_*(W)\to SH_*(W;\im c)^{\otimes 2}.
$$
\end{proof}
 
\begin{remark}
In the definition of $\boldsymbol{\lambda}_\cD$ we used the \emph{same} continuation data $\cD$ for the continuation map $c_\cD:FC_*(-K)\to FC_*(K)$ in the factorization at both ends of the $1$-parameter Floer problem. The previous results would still hold for a more general definition using \emph{different} continuation data at the two ends. We will however not discuss it.  
\end{remark}

The coproducts $\boldsymbol{\lambda}_\cD$ depend on the choice of $\cD$. Next, we will explain how this dependence can be expressed in terms of \emph{secondary continuation maps}. 

\subsection{Secondary continuation map}  \label{sec:secondary_cont}

This map is defined via a para\-me\-trized Floer problem with parameter space $[0,1]$. For readability, we lead the discussion in two steps: we first consider a simplified setup in which the Morse function and gradient-like vector field at the endpoints are assumed to be the same, then we discuss the general case in which we allow them to be different. 

In the sequel $W$ is an $R$-essential Weinstein domain.
 
\noindent {\bf (1) Simplified setup: equal Morse data at the endpoints.} Recall that, given an $R$-essential Morse function $K:W\to \R$ with Morse-Smale gradient-like vector field $\xi$, a choice of continuation data $\cD$ from $(-K,-\xi)$ to $(K,\xi)$ induces a continuation map $c_\cD:FC_*(-K)\to FC_*(K)$, and the map $c_\cD$ does not depend on $\cD$ (Lemma~\ref{lem:R-essential-continuation-maps-are-equal}). 	

\begin{definition}[secondary continuation map] Let $\cD_0,\cD_1$ be choices of continuation data from $(-K,-\xi)$ to $(K,\xi)$. The \emph{secondary continuation map} 
$$
\vec\boldc_{\cD_0,\cD_1}:FC_*(-K)\to FC_{*+1}(K)
$$
is the degree $1$ map induced by a homotopy between $\cD_0$ and $\cD_1$. 

The \emph{secondary continuation element}, or \emph{copairing} 
$$
\boldc_{\cD_0,\cD_1}\in \Big( FC_0(K)\otimes FC_1(K) \Big) \oplus \Big( FC_1(K)\otimes FC_0(K)\Big)
$$
is defined from the secondary continuation map by viewing the input in $FC_*(-K)\equiv FC_*(K)^\vee$ as the first output in $FC_*(K)$, i.e. $\vec\boldc_{\cD_0,\cD_1}= (\ev\otimes 1)(1\otimes \boldc_{\cD_0,\cD_1})$, with $\ev:FC_*(K)^\vee\otimes FC_*(K)\to R$ the evaluation {\refone (which is a degree $0$ chain map)}.
\end{definition}

That the range of bi-degrees for the components of $\boldc_{\cD_0,\cD_1}$ is restricted to $(0,1)$ and $(1,0)$ is a consequence of the fact that the Floer complex of an $R$-essential Hamiltonian is supported in degrees $0,\dots,n$. 

\begin{lemma}  \label{lem:sec_c_well_defined}
(i) The secondary continuation map is a chain map which is well-defined up to chain homotopy. 

(ii) The secondary continuation element is a cycle which is well-defined up to chain homotopy. 
\end{lemma} 

\begin{proof} (i) By definition $\vec\boldc_{\cD_0,\cD_1}$ is a chain homotopy between $c_{\cD_0}$ and $c_{\cD_1}$. Since $c_{\cD_1}=c_{\cD_0}$, we infer that $\vec\boldc_{\cD_0,\cD_1}$ is a chain map. 

That the map induced by $\vec\boldc_{\cD_0,\cD_1}$ in homology depends only on $\cD_0$ and $\cD_1$, and not on the choice of homotopy between the two, follows by a standard interpolation argument. 

(ii) This is equivalent to (i) since the continuation element is obtained from the continuation map by dualization. 
\end{proof}

Given an $R$-essential Hamiltonian $H$ with Morse truncation $K$, we consider the composition $FC_*(-H)\to FC_*(-K)\to FC_{*+1}(K)\to FC_{*+1}(H)$, where the middle map is $\vec\boldc_{\cD_0,\cD_1}$ and the extremal maps are the projection onto the quotient complex $FC_*(-K)$ and the inclusion of the subcomplex $FC_*(K)$. We call this composition also ``secondary continuation map" and denote it $\vec\boldc_{\cD_0,\cD_1}:FC_*(-H)\to FC_{*+1}(H)$. After passing to homology, we obtain in the limit a map 
$$
\vec\boldc_{\cD_0,\cD_1} : SH^{-*}(W)\to SH_{*+1}(W).
$$
Equivalently, viewing the input as an output we obtain a degree $1$ ``secondary continuation element"  
$$
\boldc_{\cD_0,\cD_1} \in \big(SH_0(W)\otimes SH_1(W)\Big) \oplus \Big(SH_1(W)\otimes SH_0(W)\Big)\subset SH_*(W)^{\otimes 2}.  
$$ 

For the next lemma we introduce the following notation. The \emph{reduced continuation element}  
$$
\bar\boldc_{\cD_0,\cD_1} \in SH_*(W;\im c)^{\otimes 2}
$$
is the projection of $\boldc_{\cD_0,\cD_1}$ under the map $SH_*(W)\to SH_*(W;\im c)$. Given a choice of continuation data $\cD$ from $(-K,-\xi)$ to $(K,\xi)$ consisting of a homotopy $(H_{s,t}, J_{s,t})$ of Hamiltonians and almost complex structures, the \emph{reverse continuation data} $\cD^{op}$ is given by the homotopy $(-H_{-s,1-t},J_{-s,1-t})$. Note that minus signs appear both in front of $H$ and the continnuation parameter $s$, so $\cD^{op}$ is again continuation data from $(-K,-\xi)$ to $(K,\xi)$. 

Recall also the coproduct $\boldlambda_\cD:SH_*(W)\to SH_*(W;\im c)^{\otimes 2}$ and the unit $\boldeta$ for the product on $SH_*(W)$. 
 
\begin{lemma} \label{lem:cDopDlambda} We have 
$$
\bar \boldc_{\cD^{op},\cD} = \boldlambda_{\cD} \boldeta.
$$
\end{lemma}

\begin{proof}
This follows by a gluing argument from the fact that $\boldeta$ is given by a count of genus $0$ curves with one negative puncture. See Figure~\ref{fig:csec-lambdaeta}. {\refone At the left end of the interval of interpolation, the middle continuation cylinder with continuation data $\cD^{op}$ can be interpreted, by reversing the direction of the coordinates $(s,t)$ and exchanging the signs of the punctures, as a continuation cylinder with continuation data $\cD$. The outcome is a configuration of Floer cylinders that represents a homological count equivalent to that of $\boldlambda_{\cD}\boldeta$.} 
\end{proof}

\begin{figure}
\begin{center}
\includegraphics[width=1\textwidth]{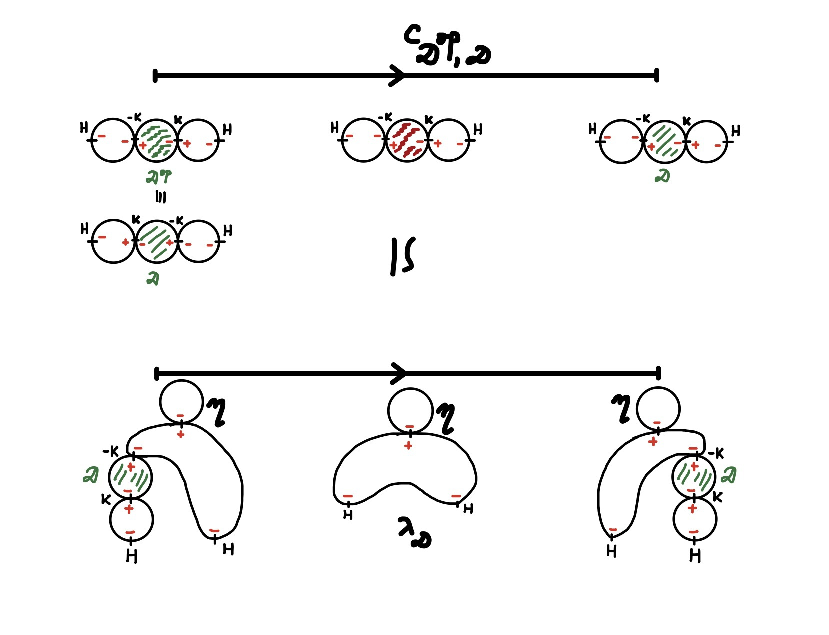}
\caption{Secondary continuation map from coproduct.}
\label{fig:csec-lambdaeta}
\end{center}
\end{figure}

\noindent {\bf (2) General setup: different Morse data at the endpoints.}
Let $(K,\xi,\cD)$ and $(K',\xi',\cD')$ be two choices of $R$-essential Morse functions, Morse-Smale gradient-like vector fields and continuation data $\cD$ from $(-K,-\xi)$ to $(K,\xi)$ and $\cD'$ from $(-K',-\xi')$ to $(K',\xi')$. 

In addition, we choose continuation data $\cC$ from $(K,\xi)$ to $(K',\xi')$. This determines also continuation data from $(-K,-\xi)$ to $(-K',-\xi')$, denoted $-\cC$. This choice of continuation data $\cC$ is necessary, as it will reflect the dependence of the coproduct on the choice of Morse function. 

\begin{definition}
The \emph{secondary continuation map}
$$
\vec\boldc_{\cD,\cD',\cC}:FC_*(-K)\to FC_{*+1}(K')
$$ 
is the degree $1$ map defined in three steps as follows: 

(i) we form continuation data $\tilde \cD=\cD\#\cC$ from $-K$ to $K'$ by appending $\cC$ to $\cD$ at the output; 

(ii) we form continuation data ${\tilde \cD}'=-\cC\#\cD'$ from $-K$ to $K'$ by appending $\cD'$ to $-\cC$ at the output; 

(iii) we choose a homotopy from $\tilde\cD$ to ${\tilde\cD}'$ and define $\vec\boldc_{\cD,\cD',\cC}=\vec\boldc_{\tilde\cD,{\tilde\cD}'}$ to be the induced map $FC_*(-K)\to FC_{*+1}(K')$. 
\end{definition}

\begin{definition}
The \emph{secondary continuation element}, or \emph{copairing} 
$$
\boldc_{\cD,\cD',\cC}\in \Big( FC_0(K) \otimes FC_1(K') \Big) \oplus \Big( FC_1(K)\otimes FC_0(K')\Big)
$$
is defined from the secondary continuation map by viewing the input in $FC_*(-K)\equiv FC_*(K)^\vee$ as the first output in $FC_*(K)$. 
\end{definition}

The following is the analogue of Lemma~\ref{lem:sec_c_well_defined}.

\begin{lemma} 
(i) The secondary continuation map is a chain map which is well-defined up to chain homotopy. 

(ii) The secondary continuation element is a cycle which is well-defined up to chain homotopy. \qed
\end{lemma}

Given $R$-essential Hamiltonians $H$, $H'$ with Morse truncations $K$, $K'$, we consider the composition 
$$
FC_*(-H)\to FC_*(-K)\to FC_{*+1}(K')\to FC_{*+1}(H'),
$$ 
where the middle map is $\vec\boldc_{\cD,\cD',\cC}$ and the extremal maps are the projection onto the quotient complex $FC_*(-K)$ and the inclusion of the subcomplex $FC_*(K')$. We call this composition also ``secondary continuation map" and denote it $\vec\boldc_{\cD,\cD',\cC}:FC_*(-H)\to FC_{*+1}(H')$. After passing to homology, we obtain in the limit a map 
$$
\vec\boldc_{\cD,\cD',\cC} : SH^{-*}(W)\to SH_{*+1}(W).
$$
Equivalently, viewing the input as an output we obtain a degree $1$ ``secondary continuation element" 
$$
\boldc_{\cD,\cD',\cC} \in \big(SH_0(W)\otimes SH_1(W)\Big) \oplus \Big(SH_1(W)\otimes SH_0(W)\Big)\subset SH_*(W)^{\otimes 2}.
$$
Projecting via the map $SH_*(W)\to SH_*(W;\im c)$ we get 
$$
\bar \boldc_{\cD,\cD',\cC} \in SH_*(W;\im c)^{\otimes 2}.
$$

\subsection{Dependence of the coproduct on the choice of continuation data} \label{sec:dependence_on_D}

Let $(K,\xi,\cD)$ and $(K',\xi',\cD')$ be two choices of $R$-essential Morse functions, Morse-Smale gradient-like vector fields and continuation data from $(-K,-\xi)$ to $(K,\xi)$ and from $(-K',-\xi')$ to $(K',\xi')$. Choose also continuation data $\cC$ from $(K,\xi)$ to $(K',\xi')$, which determines also continuation data $-\cC$ from $(-K,-\xi)$ to $(-K',-\xi')$.

Recall the coproducts $\boldlambda_\cD, \boldlambda_{\cD'}:\ol{SH}_*(W)\to SH_*(W;\im c)^{\otimes 2}$. 
Let {\refone $\tau$ be the graded twist on tensor products given by $\tau(a\otimes b)=(-1)^{|a||b|}b\otimes a$.} 
We will use the continuation element $\boldc_{\cD,\cD',\cC}$ 
defined in the previous subsection.

\begin{proposition}\label{prop:lambdaDprimeDC}
We have 
$$
\boldlambda_{\cD'} = \boldlambda_\cD + (\boldmu\otimes 1)(1\otimes \boldc_{\cD,\cD',\cC}) - (1\otimes\boldmu)
({\alex \tau\boldc_{\cD,\cD',\cC}}\otimes 1)
$$
as maps $\ol{SH}_*(W)\to SH_*(W;\im c)^{\otimes 2}$.
\end{proposition}

The terms on the right hand side have to be interpreted as follows: the multiplication involving $\boldmu$ is performed in $\ol{SH}_*(W)$, and the result is further reduced to $SH_*(W;\im c)$. 

{\refone 
\begin{remark}
Given continuation data $\cD=(H_{s,t},J_{s,t})$, recall the notation $\cD^{op}=(-H_{-s,1-t},J_{-s,1-t})$ from the previous subsection. Given continuation data $\cC=(K_s,J_{s,t})$ from $K$ to $K'$, let $-\cC=(-K_s,J_{s,t})$ be continuation data from $-K$ to $-K'$. Then $\cC^{op}$ is continuation data from $-K'$ to $-K$ and $-\cC^{op}$ is continuation data from $K'$ to $K$. Moreover  
$$
\tau\boldc_{\cD,\cD',\cC}= -\boldc_{{\cD'}^{op},\cD^{op},-\cC^{op}}.
$$
\end{remark}
}

{\refone 
\begin{proof}[Proof of Proposition~\ref{prop:lambdaDprimeDC}]
We set up a Floer problem parametrized by a rectangle as in Figure~\ref{fig:LambdaDDprime-small}. (See also Proposition~\ref{prop:cocomm-coprod} for a similar reasoning.) Along the horizontal segments we consider Floer data that corresponds to $\boldlambda_\cD$ and $\boldlambda_{\cD'}$. (In Figure~\ref{fig:LambdaDDprime-small} we split each horizontal segment in three parts in order to emphasize that the endpoints have to be understood as splittings at $-K'$ for the top horizontal segment, and at $-K$ for the bottom horizontal segment.) Along the vertical segments we consider Floer data that corresponds to the secondary continuation map $\boldc_{\cD,\cD',\cC}$, by interpreting 
the configuration consisting of a curve with 
\begin{itemize}
\item one input, 
\item two outputs, and 
\item a continuation-map-curve attached at one of the outputs, 
\end{itemize}
as a configuration consisting of 
\begin{itemize}
\item a curve with two inputs, 
\item one output, and 
\item a continuation-element-curve attached at one of the inputs.
\end{itemize} 
Finally, we extend this choice of Floer data to the interior of the rectangle. The count of $0$-dimensional moduli spaces of solutions of this parametrized Floer problem provides a chain homotopy between $0$ and $-\boldlambda_{\cD'} +\boldlambda_\cD + (\boldmu\otimes 1)(1\otimes \boldc_{\cD,\cD',\cC}) - (1\otimes\boldmu)
(\tau\boldc_{\cD,\cD',\cC}\otimes 1)$. The twist in the last term arises from our convention that, when interpreting an operation $\vec\boldc$ with 1 input and 1 output into an operation $\boldc$ with 2 outputs, we convert the input of $\vec\boldc$ into the 1st output of $\boldc$. 
\end{proof}
}

\begin{figure}
\begin{center}
\includegraphics[width=1\textwidth]{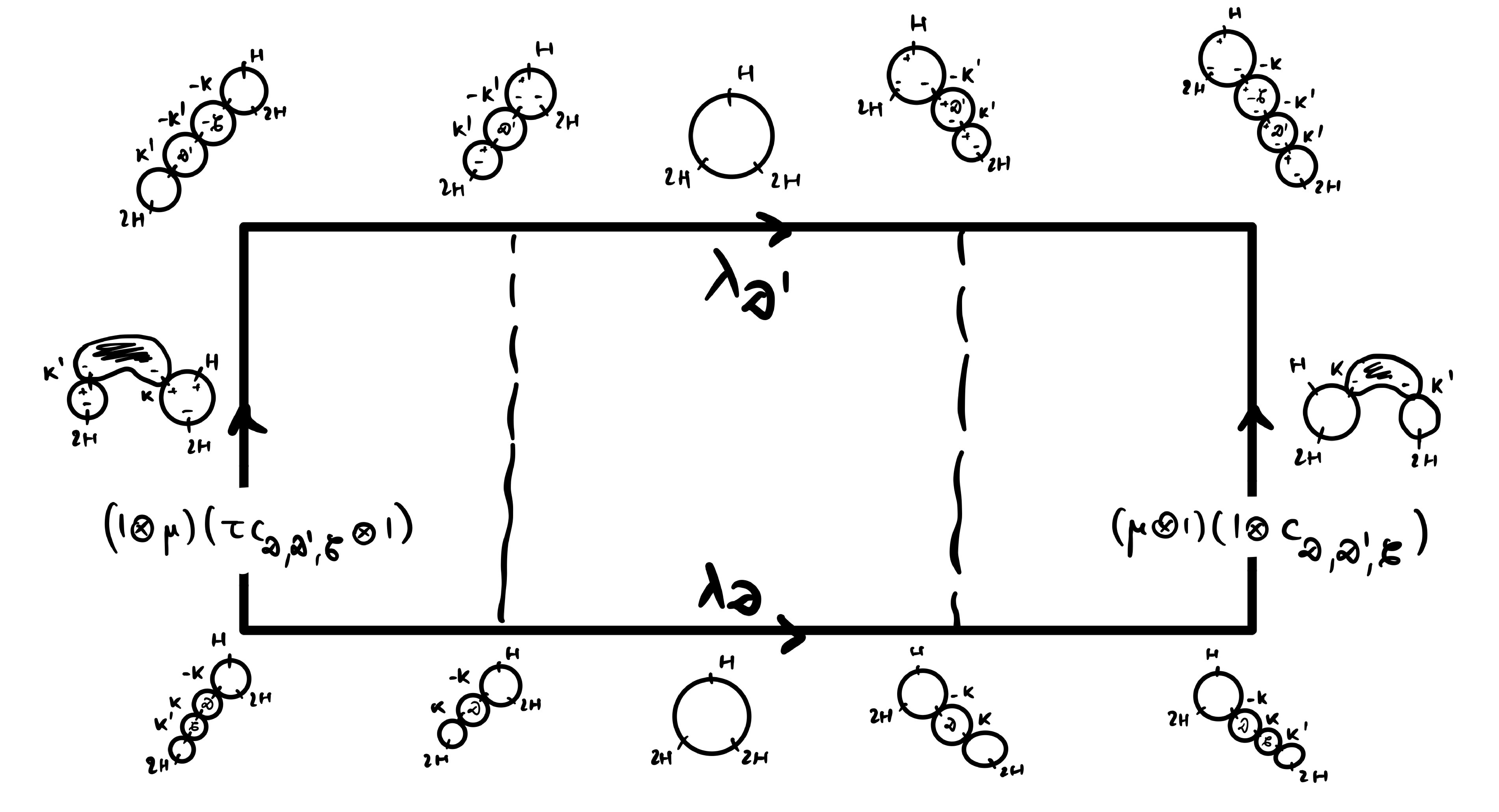}
\caption{The relation $\boldlambda_{\cD'} = \boldlambda_\cD + (\boldmu\otimes 1)(1\otimes \boldc_{\cD,\cD',\cC})-(1\otimes\boldmu)(\tau\boldc_{\cD,\cD',\cC}\otimes 1)$.}
\label{fig:LambdaDDprime-small} 
\end{center}
\end{figure}

In the special case $\cD'=\cD^{op}$ and with $\cC$ the constant continuation data (which we drop from the notation) we obtain the following. 

\begin{corollary}\label{cor:lambdaDprimeD} We have 
$$
\boldlambda_{\cD^{op}} = \boldlambda_\cD + (\boldmu\otimes 1)(1\otimes \boldc_{\cD,\cD^{op}}) + (1\otimes\boldmu)(\boldc_{\cD,\cD^{op}}\otimes 1)
$$
as maps $\ol{SH}_*(W)\to SH_*(W;\im c)^{\otimes 2}$.
\end{corollary}

\begin{proof}
Using ${\alex \tau \boldc_{\cD,\cD^{op}}}=-\boldc_{\cD,\cD^{op}}$ and Proposition~\ref{prop:lambdaDprimeDC} we find
\begin{align*}
\boldlambda_{\cD^{op}} & = \boldlambda_\cD + (\boldmu\otimes 1)(1\otimes \boldc_{\cD,\cD^{op}}) - (1\otimes\boldmu)({\alex \tau \boldc_{\cD,\cD^{op}}}\otimes 1)\\
& = \boldlambda_\cD + (\boldmu\otimes 1)(1\otimes \boldc_{\cD,\cD^{op}}) + (1\otimes\boldmu)(\boldc_{\cD,\cD^{op}}\otimes 1).
\end{align*}
\end{proof}

\begin{corollary} \label{cor:independence} For an $R$-essential Weinstein domain $W$ such that $H^{n-1}(W)=0$, the coproduct $\boldlambda_\cD$ is independent of the choice of continuation data $\cD$. 
\end{corollary}

\begin{proof} Given another choice of continuation data $\cD'$, we claim that the continuation map $\vec\boldc_{\cD,\cD',\cC}$ 
vanishes. This implies $\boldlambda_{\cD'}=\boldlambda_\cD$ by Proposition~\ref{prop:lambdaDprimeDC}. 

We now prove the claim. The continuation map acts as $SH^{-*}(W)\to SH_{*+1}(W)$ and it factors through the energy zero part $H_{n+*}(W)\to H^{n-*-1}(W)$. Since the homology and cohomology of $W$ are supported in degrees $0,\dots,n$, this map can only be nonzero in degrees $*=-1,0$. We are thus left to consider the two maps 
$$
H_{n-1}(W)\to H^n(W)\qquad \mbox{and}\qquad H_n(W)\to H^{n-1}(W). 
$$
The condition $H^{n-1}(W)=0$ ensures the vanishing of the second map. This condition also implies that $\Hom_R(H_{n-1}(W),R)$ vanishes. Indeed, this group is a surjective image of $H^{n-1}(W)$ by the universal coefficient theorem (this uses that $R$ is a principal ideal domain). On the other hand, by the condition of $R$-essentiality the Morse differential $\p:MC_n(W)\to MC_{n-1}(W)$ is zero, hence $H_{n-1}(W)$ is free. The vanishing of $\Hom_R(H_{n-1}(W),R)$ is therefore equivalent to the vanishing of $H_{n-1}(W)$, and this ensures the vanishing of the first map. 
\end{proof}
 
\begin{corollary} \label{cor:lambdaeta=0}
For an $R$-essential Weinstein domain $W$ such that $H^{n-1}(W)=0$, the coproduct $\boldlambda=\boldlambda_\cD$ satisfies 
$$
\boldlambda\boldeta=0,
$$
where $\boldeta$ is the unit for the product $\boldmu$. 
\end{corollary}

\begin{proof}
We have seen in the proof of Corollary~\ref{cor:independence} that, under the assumption $H^{n-1}(W)=0$, all the continuation maps vanish. In particular $\boldc_{\cD^{\mathrm{op}},\cD}=0$, and we conclude using the equality $\boldc_{\cD^{\mathrm{op}},\cD}=\boldlambda_\cD\boldeta$ from Lemma~\ref{lem:cDopDlambda}. 
\end{proof}

\begin{remark}
For the cotangent bundle $W=D^*M$ of an $R$-orientable manifold $M$, this recovers the corresponding result from~\cite[\S4]{CHO-MorseFloerGH}. Indeed, we have $H^{n-1}(W)\simeq H^{n-1}(M)\simeq H_1(M)$ by Poincar\'e duality.
\end{remark}

\section{Strongly $R$-essential Weinstein domains}   \label{sec:Weinstein_strong_R_essential}

In this section we introduce a more restrictive class of $R$-essential Weinstein domains which have the following nice features:
\begin{itemize}
\item the dimension condition $2n\ge 6$ from the definition of the coproducts $\boldsymbol{\lambda}=\boldsymbol{\lambda}_\cD$ can be dropped. 
\item reduced symplectic homology and symplectic homology relative to the continuation map are equal, so that they provide a common domain of definition for the pair of pants product and for the pair of pants secondary coproducts. 
\end{itemize}
 
\begin{definition}
A \emph{strongly $R$-essential Weinstein domain} is an $R$-essential Weinstein domain $W$
such that the canonical map 
$$
\ol{SH}_*(W)\to SH_*(W;\im c)
$$
is an isomorphism. 
\end{definition}

This terminology is motivated by Proposition~\ref{prop:olSH_relSH_sufficient_condition}. Given an $R$-essential Weinstein domain $W$, a sufficient condition which ensures that it is strongly $R$-essential is injectivity of the map $H^n(W;R)\to SH_0(W;R)$. This means that the top dimensional cells of the skeleton are homologically essential not only in the Morse (zero energy) sector, but also inside the full symplectic homology group.\footnote{In view of Example~\ref{ex:strongly_R_essential}.(i), such Weinstein domains were called ``cotangent-like" in a previous version of this paper. We opted for a different terminology because, from the arboreal perspective, all Weinstein domains are ``cotangent-like".}

\begin{example} \label{ex:strongly_R_essential}
Examples of strongly $\Z$-essential Weinstein domains include the following (with their canonical Weinstein structure). 

(i) Disc cotangent bundles of closed orientable manifolds (in the non-orientable case they are strongly $\Z/2$-essential). 

(ii) Subcritical Weinstein domains. 

(iii) $\Z$-essential Weinstein domains $W$ which have vanishing first Chern class and which admit a Liouville form whose closed Reeb orbits on the boundary $\p W$ which are contractible in the interior $W$ have Conley-Zehnder index $\neq 0,1$ (this ensures that nonconstant orbits cannot kill any generators of $H^n(W)$). For example: 
\begin{itemize}
\item Plumbings of cotangent bundles of simply-connected closed manifolds of dimension $\ge 3$. Indeed, as shown in~\cite[Theorem~54]{Ekholm-Lekili}, their boundaries admit contact forms whose closed Reeb orbits are nondegenerate and have Conley-Zehnder index $\ge 2$.  
\item Milnor fibers of isolated singularities which admit contact forms on the boundary as above. 
Specific examples are the Milnor fillings of Brieskorn manifolds $\Sigma(\ell,2,\dots,2)$ of dimension $2n-1\ge 5$. The indices have been computed explicitly by Ustilovsky~\cite{Ustilovsky} and van Koert~\cite{vanKoert-Brieskorn}, see also Uebele~\cite{Uebele-Brieskorn} and Fauck~\cite{Fauck-thesis}. That all indices are $\ge 2$ in this case can be seen using the formula from~\cite[Proposition~103]{Fauck-thesis}. 
Many other Brieskorn manifolds satisfy the assumption on the indices, although it is unclear whether they can be characterized in a simple manner. 
\end{itemize}
\end{example}

The following example shows that, in contrast to $R$-essentiality, the notion of strong $R$-essentiality depends on the symplectic structure. 

\begin{example}
In this example we use as ring of coefficients $R=\Z$.
Consider a Weinstein domain $W$ with $SH_*(W)=0$. For instance, this is the case if $W$ is {\em flexible}~\cite{Cieliebak-Eliashberg-book}, or more generally {\em subflexible}~\cite{Murphy-Siegel}. Then $W$ is strongly $\Z$-essential if and only if the intersection form on $H_n(W)$ is trivial. 
To see this, note first that $\ol{SH}_*(W)=SH_*(W)/\im c_*=0$. So $W$ is strongly $\Z$-essential iff $SH_*(W;\im c)=0$, or equivalently (by the long exact sequence of the pair $\im c\subset SC_*(W)$), $H_*(\im c)=0$. Since the chain complex $\im c$ is supported in a single degree, this holds iff $c=0$ on chain level. By~\eqref{eq:R-essential-cochains}, this is equivalent to $c_*=0$ on homology, which in turn is equivalent to vanishing of the intersection form by Proposition~\ref{prop:int-form}. 

For example, let $W=D^*M$ be the disc cotangent bundle of a closed oriented manifold $M$. With its canonical Weinstein structure this is strongly $\Z$-essential by Example~\ref{ex:strongly_R_essential}(i). On the other hand, by the preceding discussion, $D^*M$ with its unique flexible Weinstein structure (see~\cite{Cieliebak-Eliashberg-book}) is strongly $\Z$-essential if and only if the intersection form on $H_n(D^*M)=\Z$ is trivial, i.e., if and only if the Euler characteristic of $M$ vanishes. 
\end{example}

\begin{proposition} \label{prop:coproduct_reduced_bar}
Let $W$ be a strongly $R$-essential Weinstein domain. The choice of an $R$-essential Morse function $K$, 
of a Morse-Smale gradient-like vector field, and of continuation data $\cD$ from $-K$ to $K$, determines a coproduct
\footnote{The coproduct depends a priori on all these choices, though we use only the shorthand notation $\boldsymbol{\lambda}=\boldsymbol{\lambda}_\cD$.}
$$
\boldsymbol{\lambda}=\boldsymbol{\lambda}_\cD:\ol{SH}_*(W)\to \ol{SH}_*(W)^{\otimes 2}.
$$  
\end{proposition}

\begin{proof}
The map $\boldlambda_\cD:\ol{SH}_*(W)\to SH_*(W;\im c)^{\otimes 2}$ from Proposition~\ref{prop:lambda_olSH_SHimc} lifts to 
$$
\boldlambda_\cD:\ol{SH}_*(W)\to \ol{SH}_*(W)^{\otimes 2}
$$
by the strong $R$-essentiality assumption.
\end{proof}

\begin{remark}
It is instructive to ponder the role played by strong $R$-essentiality in the previous proof, allowing to remove the dimensionality assumption $2n\ge 6$ from Proposition~\ref{prop:coproduct_on_SH_rel_imc}. 
In the sequel we will use strong $R$-essentiality as a standing assumption and prove robust algebraic properties of the coproduct in that setup. All the sequel results have counterparts for $R$-essential Weinstein domains, which however require additional dimensional assumptions.   
\end{remark}

\section{Bialgebra structure on reduced symplectic homology} \label{sec:bialgebra}

\subsection{Unital infinitesimal anti-symmetric bialgebras}

Let $A$ be a graded module over a unital commutative ring $R$. We assume that $A$ is of \emph{finite type}, meaning that it is free and finite dimensional in each degree. We recall from~\cite{CHO-algebra} the following definition. 
 
\begin{definition} \label{defi:secondary-unital} 
A \emph{unital infinitesimal anti-symmetric bialgebra}\footnote{In~\cite{CHO-algebra} we describe a variation of  this algebraic structure for which the coproduct lands in a completed tensor product. While that structure is relevant for Rabinowitz Floer homology, it is not needed for reduced symplectic homology, whose definition involves only a direct limit over the action filtration and no inverse limit.} 
is a graded module $A$ endowed with a product $\boldmu:A\otimes A\to A$, a coproduct $\boldlambda:A\to A\otimes A$ and an element $\boldeta\in A$ which satisfy the following relations:
\begin{itemize}
\item {\sc (unit)} the element $\boldeta$ is the unit for the product $\boldmu$.
\item {\sc (associativity)} the product $\boldmu$ is associative. 
\item {\sc (coassociativity)} the coproduct $\boldlambda$ is coassociative. 
\item {\sc (unital infinitesimal relation)} 
\begin{align*}
\boldlambda\boldmu = (-1)^{|\boldlambda||\boldmu|} \Big((1\otimes\boldmu)&(\boldlambda\otimes 1) + (\boldmu\otimes 1)(1\otimes\boldlambda)\Big) \\
& - (-1)^{|\boldmu|}(\boldmu\otimes \boldmu)(1\otimes \boldlambda\boldeta\otimes 1).
\end{align*}
\item {\sc (unital anti-symmetry)} 
\begin{align*}
&(-1)^{|\boldmu| (|\boldlambda|+1)} (1\otimes\boldmu) (\tau\boldlambda\otimes 1)  
+ (-1)^{|\boldlambda| (|\boldmu|+1)}(\boldmu\tau\otimes 1)(1\otimes\boldlambda) \\
& \hspace{5cm}- (-1)^{|\boldlambda|+|\boldmu|}(\boldmu\tau\otimes \boldmu)(1\otimes\boldlambda\boldeta\otimes 1) \\
&= (-1)^{|\boldlambda| |\boldmu|} \tau(1\otimes\boldmu\tau)(\boldlambda\otimes 1) 
- (-1)^{(|\boldlambda|+1)(|\boldmu|+1)} \tau (\boldmu\otimes 1)(1\otimes\tau\boldlambda) \\ 
& \hspace{5cm} - (-1)^{|\boldmu|}\tau(\boldmu\otimes \boldmu\tau)(1\otimes\boldlambda\boldeta\otimes 1).
\end{align*}
\end{itemize}
\end{definition}

{\refone We recall that graded associativity for $\boldmu$ is expressed by $\boldmu(\boldmu\otimes 1)=(-1)^{|\boldmu|}\boldmu(1\otimes\boldmu)$, and graded coassociativity for $\boldlambda$ is expressed by $(\boldlambda\otimes 1)\boldlambda=(-1)^{|\boldlambda|}(1\otimes\boldlambda)\boldlambda$.} 

Evaluating the {\sc (unital anti-symmetry)} relation on $\boldeta\otimes \boldeta$ we obtain in particular
$$
\tau \boldlambda\boldeta = (-1)^{|\boldlambda|}\boldlambda\boldeta. 
$$

The notion of a unital infinitesimal anti-symmetric bialgebra is invariant under shift. 

It is useful to give a graphical interpretation of the {\sc (unital infinitesimal relation)} and of {\sc (unital anti-symmetry)}. Let us represent $\boldmu$ and $\boldlambda$ in the form of {\sf Y}-shaped graphs, 
with the inputs depicted in clockwise order with respect to the output for $\boldmu$, and the outputs depicted in counterclockwise order with respect to the input for $\boldlambda$. See Figure~\ref{fig:mu-and-lambda}.
\begin{figure}
\begin{center}
\includegraphics[width=.7\textwidth]{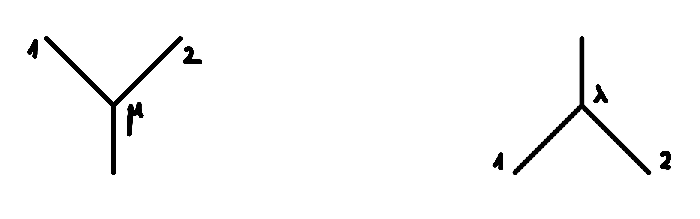}
\caption{The product $\boldmu$ and the coproduct $\boldlambda$.}
\label{fig:mu-and-lambda} 
\end{center}
\end{figure}
Then the {\sc (unital infinitesimal relation)} takes the form depicted in Figure~\ref{fig:infinitesimal-schematic}, and {\sc (unital anti-symmetry)} takes the form depicted in Figure~\ref{fig:4-term-new-schematic}.

\begin{figure}
\begin{center}
\includegraphics[width=\textwidth]{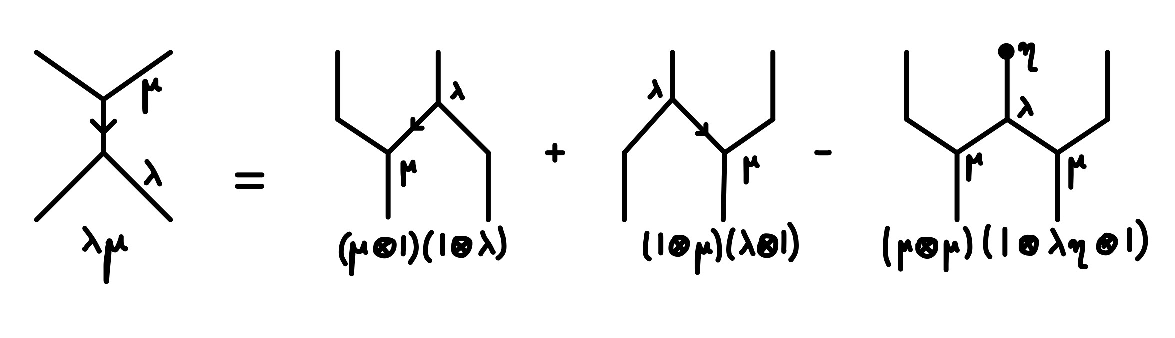}
\caption{The unital infinitesimal relation for $|\boldmu|=0$.}
\label{fig:infinitesimal-schematic} 
\end{center}
\end{figure}

\begin{figure}
\begin{center}
\includegraphics[width=\textwidth]{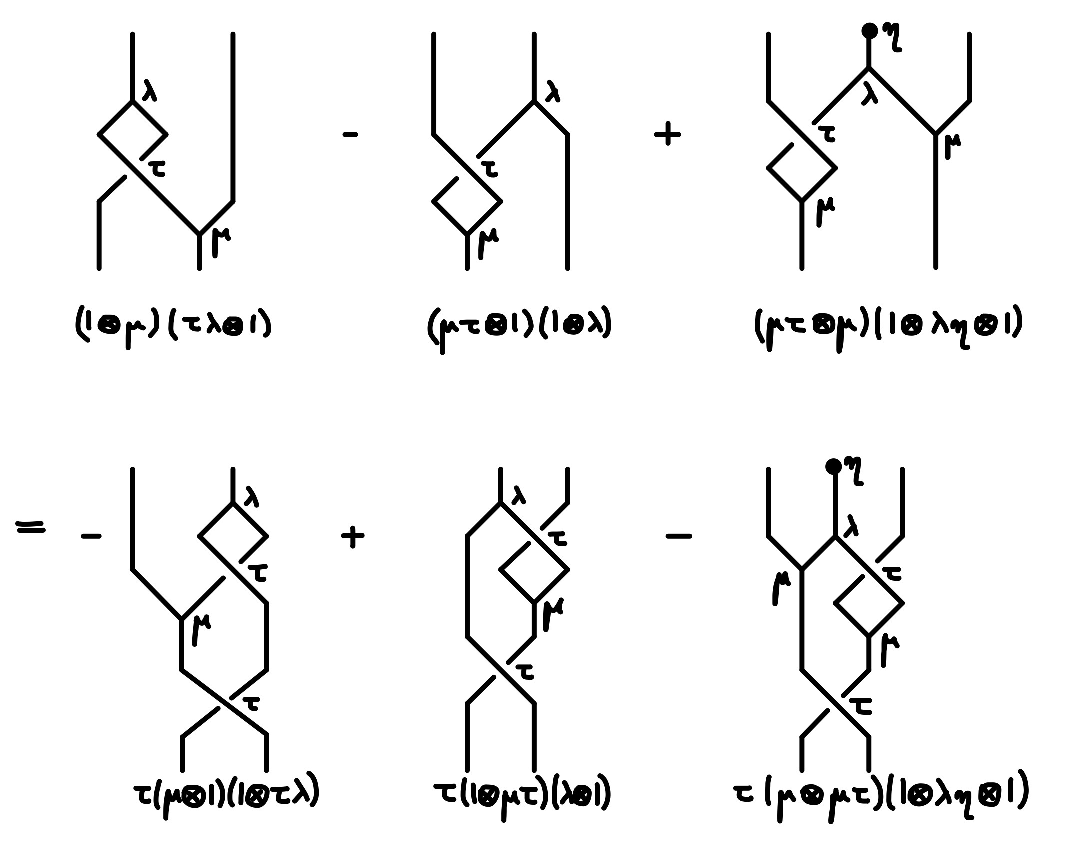}
\caption{Unital anti-symmetry for $|\boldmu|=0$ and $|\boldlambda|$ odd.}
\label{fig:4-term-new-schematic} 
\end{center}
\end{figure}

\begin{remark}[The commutative and cocommutative case]\label{rem:ccIAB}
If $\boldmu$ is commutative and $\boldlambda$ is cocommutative,
$$
   \boldmu\tau = (-1)^{|\boldmu|}\boldmu\quad\text{and}\quad
   \tau\boldlambda = (-1)^{|\boldlambda|}\boldlambda,
$$ 
then {\sc (unital anti-symmetry)} is a consequence of the {\sc (unital infinitesimal relation)}. To see this, simply observe that the unital infinitesimal relation transforms the left hand side of the unital anti-symmetry relation to $(-1)^{|\boldmu|+|\boldlambda|}\boldlambda\boldmu$ and the right hand side to $(-1)^{|\boldmu|}\tau\boldlambda\boldmu$, so the two sides are equal. 
\end{remark}

\begin{remark}[Involutivity, cf.~{\cite{CHO-algebra}}] 
If $2\neq 0$ in the ring $R$, and if $\boldmu$ and $\boldlambda$ are
commutative resp. cocommutative of opposite parity, then the following {\sc (involutivity)} relation holds:
$$
\boldmu\boldlambda =0.
$$
Indeed, we have 
$$
\boldmu\boldlambda = (-1)^{|\boldlambda|}\boldmu(\tau \boldlambda) = (-1)^{|\boldlambda|}(\boldmu\tau)\boldlambda=(-1)^{|\boldlambda|+|\boldmu|}\boldmu\boldlambda=-\boldmu\boldlambda.
$$
\end{remark}

\subsection{Bialgebra structure on reduced symplectic homology}

Let $W$ be a strongly $R$-essential Weinstein domain of dimension $2n$. We fix an $R$-essential Morse function $K$, a Morse-Smale gradient-like vector field for $K$, continuation data from $-K$ to $K$, and denote this entire set of data by $\cD$. \emph{Shifted reduced symplectic homology} 
$$
\ol{S\H}_*(W)=\ol{SH}_{*+n}(W)
$$
carries the pair of pants product $\boldmu$ of degree $0$ with unit $\boldeta$, and it also carries the coproduct $\boldlambda_\cD$ of odd degree $1-2n$ defined above. 

\begin{theorem} \label{thm:uias_reduced} Let $W$ be a strongly $R$-essential Weinstein domain of dimension $2n\ge 6$. Shifted reduced symplectic homology 
$$
(\ol{S\H}_*(W),\boldmu,\boldlambda_\cD,\boldeta)
$$ 
is a unital infinitesimal anti-symmetric bialgebra which is commutative and cocommutative.  
\end{theorem}

The proof of Theorem~\ref{thm:uias_reduced} is spread over the next subsections. Coassociativity of $\boldlambda_\cD$ is proved in~\S\ref{sec:coass}. The unital infinitesimal relation is proved in~\S\ref{sec:infinitesimal}. The unital anti-symmetry relation is proved in~\S\ref{sec:unital_anti_symmetry}. Although this is implied in the closed string case by the unital infinitesimal relation in conjunction with commutativity and cocommutativity (Remark~\ref{rem:ccIAB}), we give a proof relying on an explicit analysis of moduli spaces which applies verbatim in the open string case, see~\S\ref{sec:open_strings}

\begin{remark}
We expect that the statement also holds in dimension $2$. The current restriction on the dimension arises from our method of proof of coassociativity.  
\end{remark}

\begin{remark}
In the case $H^{n-1}(W)=0$, it follows from Corollary~\ref{cor:lambdaeta=0} that the unital infinitesimal relation reduces to the infinitesimal-, or Sullivan's relation 
$$
\boldlambda\boldmu = (1\otimes\boldmu)(\boldlambda\otimes 1) + (\boldmu\otimes 1)(1\otimes\boldlambda).
$$
In this case the structure at hand is that of a (commutative and cocommutative) infinitesimal bialgebra in the sense of Aguiar~\cite{Aguiar} (see~\cite{CHO-algebra} for a comprehensive discussion of the origin of this notion). 
\end{remark}

\begin{remark} \label{rmk:dual} One can define a dual notion of \emph{reduced symplectic cohomology} 
$\ol{SH}^{\alex -*}(W)=\ker \, \eps_*$, with $\eps_*$ defined in~\S\ref{sec:reduced_homology}. It carries a cocommutative counital coproduct $\boldmu^\vee$ dual to $\boldmu$, a commutative product $\boldlambda_\cD^\vee$ dual to $\boldlambda_\cD$, and the resulting structure is that of a counital infinitesimal anti-symmetric bialgebra in the sense of~\cite{CHO-algebra}. 
\end{remark}

\subsection{Coassociativity} \label{sec:coass} 

\begin{proposition} \label{prop:coass_reduced}
The coproduct $\boldlambda_\cD$ is coassociative if $W$ has dimension $2n\ge 6$. 
\end{proposition}

\begin{proof} Let $H$ be an $R$-essential Hamiltonian with truncation $K$. 

In Figure~\ref{fig:Coassociativity} we depict a family of Floer problems parametrized by a $3$-dimensional polytope which we call {\sf The House}. The source Riemann surface for this Floer problem is a genus zero curve with four punctures constrained to lie on a circle, one positive with weight $1$, and three negative with respective weights equal to $3$. This Riemann surface degenerates to nodal curves along the boundary of the polytope as indicated in Figure~\ref{fig:Coassociativity}. {\sf The House} has 7 codimension 1 faces. We refer to each of the two triangles together with their adjacent quadrilaterals as {\sf The Back and Front Walls}. We refer to the 2 pentagonal faces and to the hexagonal face as {\sf The Longitudinal Walls}. Each of the faces of {\sf The House} is labeled by a possibly nodal Riemann surface, and parametrizes a Floer problem with that nodal Riemann surface at the source. We indicate in the figure the  Hamiltonians at the intermediate nodes (either $-K$, or multiples of $H$ indicated by numbers). 

\begin{figure}
\begin{center}
\includegraphics[width=.9\textwidth]{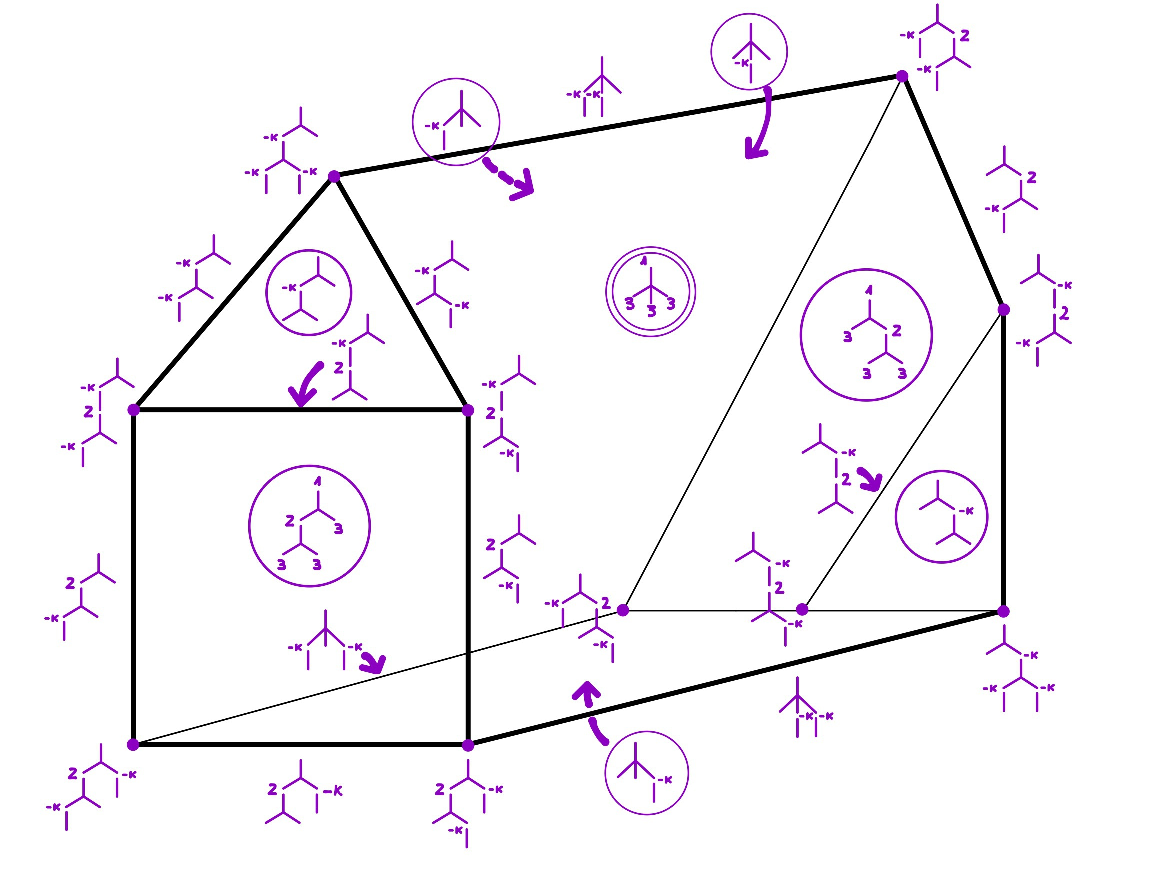}
\caption{{\sf The House}.}
\label{fig:Coassociativity}
\end{center}
\end{figure}

The Floer degeneration data is chosen such that the $0$-dimensional solutions of the Floer problems parametrized by the quadrilateral faces define the compositions $(\boldlambda_\cD\otimes 1)\boldlambda_\cD$ and $(1\otimes \boldlambda_\cD)\boldlambda_\cD$. 

Denote $\boldB_\cD$ the operation with three outputs obtained by dualizing the input of the operation $\boldbeta_\cD$ from the proof of Lemma~\ref{lem:lambda-imc}. The Floer problems parametrized by the triangular faces can then be interpreted as $(1\otimes 1\otimes \boldmu)(\boldB_\cD\otimes 1)$ and $(\boldmu\otimes1\otimes 1)(1\otimes \boldB_\cD)$. 

The $0$-dimensional solutions of the Floer problems parametrized by {\sf The Longitudinal Walls} define operations which, when they are applied to $FC_*(H)$, land in $I_\cD=\im c_\cD \, \otimes FC_*(3H)\otimes FC_*(3H) + FC_*(3H)\otimes \, \im c_\cD \, \otimes FC_*(3H) + FC_*(3H)\otimes FC_*(3H)\otimes \, \im c_\cD\, $. These operations vanish in the quotient $\left( FC_*(3H)/\im c_\cD\right)^{\otimes 3}$.

The count of elements of $0$-dimensional moduli spaces of solutions of the Floer problem parametrized by {\sf The House} defines an operation 
$$
\Theta : FC_*(H)\to FC_*(3H; \im c_\cD)^{\otimes 3}
$$
such that 
\begin{align}
[\p,\Theta]=(\boldlambda_\cD\otimes & 1)\boldlambda_\cD + (1\otimes \boldlambda_\cD)\boldlambda_\cD \label{eq:coass_first_line}\\
& + (1\otimes 1\otimes \boldmu)(\boldB_\cD\otimes 1)-(\boldmu\otimes1\otimes 1)(1\otimes \boldB_\cD). \nonumber
\end{align}
{\refone The minus sign in front of the term $(\boldmu\otimes1\otimes 1)(1\otimes \boldB_\cD)$ arises by taking into account the boundary orientation.} 

We now show that, under the assumption $2n\ge 6$, the map $\boldB_\cD$ vanishes. We first note that, since $\boldB_\cD$ has no inputs, all the outputs have energy close to zero and the map can be equivalently phrased in terms of a count of gradient $\mbox{\sf Y}$-graphs with three outputs $a,b,x\in MC^*(K)$ in a 2-parametric family. Referring to the proof of Lemma~\ref{lem:lambda-imc}, the condition for 0-dimensional moduli spaces is $\ind_K a + \ind_K b + \ind_K x -4n+2=0$, or $\ind_K a + \ind_K b + \ind_K x=4n-2$. By $R$-essentiality the left hand side of this equality is $\le 3n$, which is strictly less than $4n-2$ as soon as $n\ge 3$. This shows that there are no such rigid configurations and therefore $\boldB_\cD$ vanishes.

As a consequence, the chain level relation~\eqref{eq:coass_first_line} reduces to $[\p,\Theta]=(\boldlambda_\cD\otimes  1)\boldlambda_\cD + (1\otimes \boldlambda_\cD)\boldlambda_\cD$, which proves coassociativity. 
\end{proof}

\subsection{The unital infinitesimal relation} \label{sec:infinitesimal}
 
\begin{proposition} \label{prop:infinitesimal_reduced}
The unital infinitesimal relation holds: 
\begin{equation}\label{eq:unital_infinitesimal_relation_reduced}
\boldlambda_\cD\boldmu= (1\otimes\boldmu)(\boldlambda_\cD\otimes 1) + (\boldmu\otimes 1)(1\otimes \boldlambda_\cD) - (\boldmu\otimes \boldmu)(1\otimes \boldlambda_\cD\boldeta\otimes 1).
\end{equation} 
\end{proposition} 

\begin{proof}
Let $\cM_{0,4}^\circ$ be the moduli space of genus $0$ Riemann surfaces $\Sigma_{2;2}$ with 2 positive punctures and 2 negative punctures constrained to lie on a circle and ordered such that the negative punctures do not separate the positive punctures. This moduli space is diffeomorphic to an open interval and its compactification $\ol \cM_{0,4}^\circ$ is diffeomorphic to a closed interval whose ends correspond to nodal curves with two irreducible components and matching asymptotic markers at their common node, each containing two punctures: at one end the two punctures on each irreducible component have the same sign, and at the other end they have opposite signs. We choose a family of cylindrical ends over $\ol \cM_{0,4}^\circ$ which is compatible with splittings at the boundary.

\begin{figure}
\begin{center}
\includegraphics[width=1\textwidth]{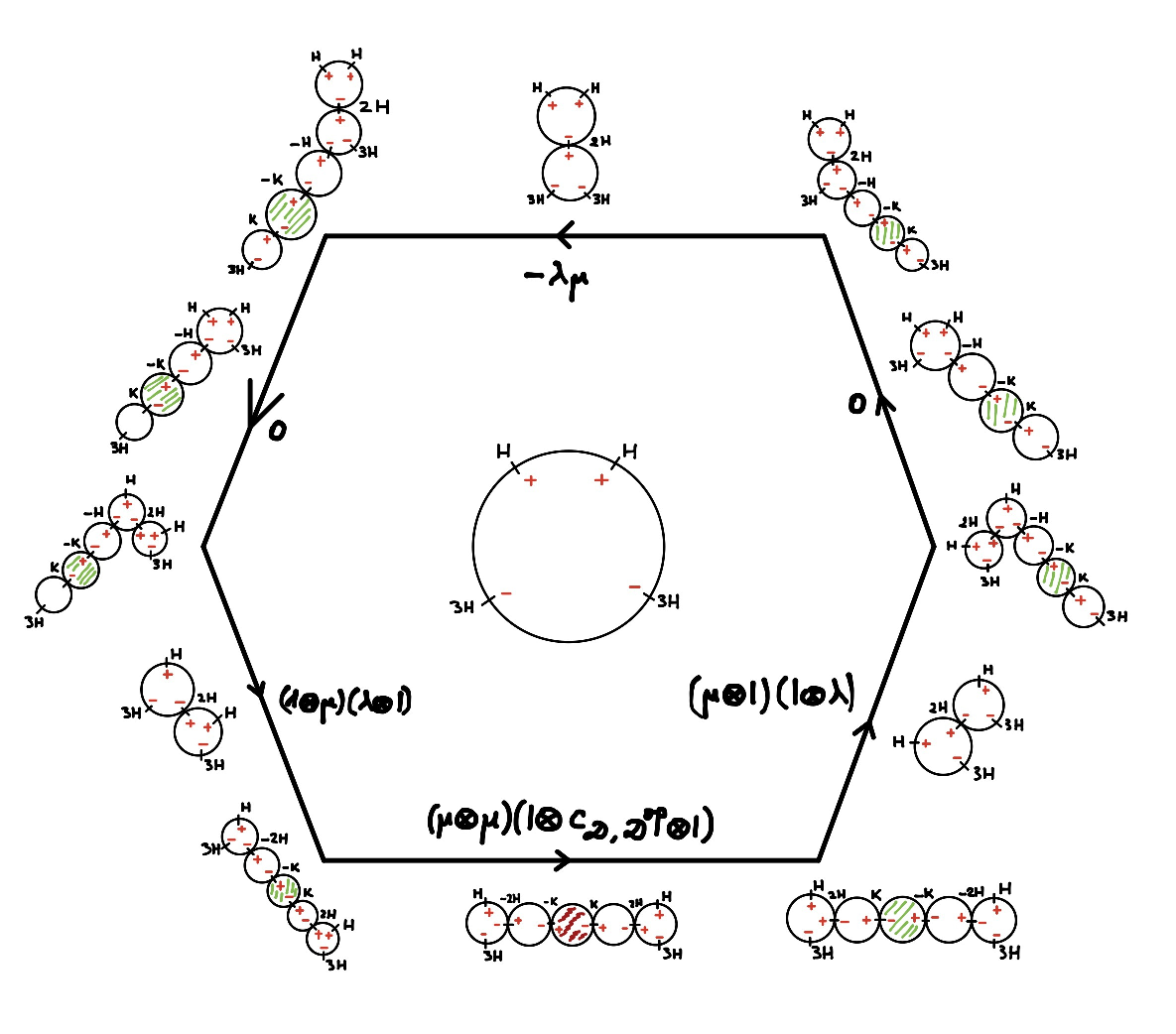}
\caption{Proof of the unital infinitesimal relation.}
\label{fig:Sullivan-RFH}
\end{center}
\end{figure}

We now consider an oriented $2$-dimensional hexagon $\cP$ which is fibered over $\ol \cM_{0,4}^\circ$ and which parametrizes Floer data $(H_\tau,\beta_\tau)$, $\tau\in\cP$ as described below. The projection map $\cP\to \ol \cM_{0,4}^\circ$ forgets the Floer data and contracts the unstable components of the underlying Riemann surfaces. We fix an $R$-essential Hamiltonian $H$ and refer to Figure~\ref{fig:Sullivan-RFH} for a pictorial description of the hexagon $\cP$. {\refone The appearance of this hexagon is a manifestation of the fact that moduli spaces of Riemann surfaces endowed with Floer data carry strictly more information than the ``plain'' moduli spaces of Riemann surfaces.\footnote{\refone This phenomenon manifests itself at an operadic level in the fact that the Deligne-Mumford-Knudsen operad is equivalent to the framed little 2-discs operad with trivialized $S^1$-action. This was conjectured by Kontsevich and proved by Drummond-Cole~\cite{Drummond-Cole}, see also~\cite{Oancea-Vaintrob} and~\cite{Deshmukh} for higher-genus generalizations, and~\cite{Oancea-Vaintrob} for a historical account of this problem.}}	
\begin{itemize}
\item
{\reftwo The Riemann surface that underlies 
the top horizontal side $\cP^{\text{top}}$ of $\cP$ is the endpoint of $\ol \cM_{0,4}^\circ$ for which the punctures on each irreducible component have the same sign.} The Floer data on the irreducible component that contains the two positive punctures is fixed and equal to $(H,\beta)$, with $\beta$ a closed $1$-form that has weights $1$ at the positive punctures and weight $2$ at the node. The Floer data on the irreducible component that contains the two negative punctures is the Floer data that defines the secondary coproduct $-\boldsymbol{\lambda}_\cD$. In particular, at the endpoints of $\cP^{\text{top}}$ the Floer data is defined on a Riemann surface with two unstable components and is non-zero on these, so that the solutions of the Floer equation are stable maps.  
\item The projection map $\cP^{\text{top, vert}}\to \ol \cM_{0,4}^\circ$ defined on each of the two top vertical sides of $\cP$ is a smooth diffeomorphism. The underlying Riemann surface for the Floer data has one stable component with 2 positive punctures, 1 negative puncture and one node, and 2 unstable components. The second negative puncture is located on the unstable component which is not adjacent to the stable component. The Floer data on the unstable components consists of homotopies from $-H$ to $-K$, from $-K$ to $K$ given by $\cD$, and from $K$ to $3H$. The Floer data on the stable component is $(H,\beta_\tau)$, where $\beta_\tau$ interpolates within the space of closed $1$-forms with weights $(1,1;3,-1)$ (resp. $(1,1;-1,3)$), between a split $1$-form with weights $(1,1;2)$ and $(2;3,-1)$ (resp. $(2;-1,3)$), and a split $1$-form with weights $(1;2,-1)$ and $(1,2;3)$ (resp. $(1;-1,2)$ and $(2,1;3)$). 
\item {\reftwo The Riemann surface that underlies 
the bottom vertical sides\break $\cP^{\text{bottom,vert}}$ and the bottom horizontal side $\cP^{\text{bottom}}$ is the endpoint of $\ol \cM_{0,4}^\circ$ for which the punctures on each irreducible component have different signs. }

Along the two sides $\cP^{\text{bottom,vert}}$, the Floer data is as follows: 
\begin{itemize}
\item on the irreducible component for which the node is labeled as positive, the Floer data is constant equal to $(H,\beta)$, with $\beta$ a closed $1$-form that has weights $(2,1;3)$ (resp. $(1,2;3)$).
\item on the irreducible component for which the node is labeled as negative, the Floer data defines the secondary coproduct $\boldsymbol{\lambda}_\cD$. In particular, at the endpoints of $\cP^{\text{bottom,vert}}$ the underlying nodal Riemann surface has two unstable components, on which the Floer data consists of homotopies $-H\to -K\stackrel{\cD}\longrightarrow K\to 3H$ or $-2H\to -K\stackrel{\cD}\longrightarrow K \to 3H$.
\end{itemize}

Along the bottom horizontal side $\cP^{\text{bottom}}$ the underlying Riemann surface has two stable components with 2 punctures and 1 node each, separated by an unstable component. At the endpoints of $\cP^{\text{bottom}}$ the unstable component is replaced by two unstable components. The Floer data is constant on the stable components along $\cP^{\text{bottom}}$, whereas on the unstable component it is given by a family of homotopies $H_{\tau,s}$ from $-2H$ to $2H$, interpolating between the broken homotopy $-2H\to -K\stackrel{\cD}\longrightarrow K\to 2H$ and the broken homotopy $-2H\to -K\stackrel{\cD^{op}}\longrightarrow K \to 2H$.  (In order to write this, one needs to reinterpret the Floer data at the endpoint of $\cP^{\text{bottom}}$ so that it fits with the Floer data at the starting point of the following bottom vertical side.)
\end{itemize}

Given orbits $x_0,x_1\in\cP(H)$ and $y_0,y_1\in\cP(3H)$, we define 
the moduli space $\cM^\cP(x_0,x_1;y_0,y_1)$ consisting of pairs $(\tau,u)$, $\tau\in\cP$, $u:\Sigma_\tau\to \wh W$ such that
\begin{equation}\label{eq:FloereqnP}
(du-X_{H_\tau}\otimes \beta_\tau)^{0,1}=0
\end{equation} 
and $u$ is asymptotic to $x_0,x_1$, resp. $y_0,y_1$ at the positive, resp. negative punctures. The anti-holomorphic part of $du-X_{H_\tau}\otimes \beta_\tau$ is considered with respect to a $\cP$-family of compatible almost complex structures $(J_\tau)$ which are cylindrical in the symplectization end $[1,\infty)\times \p W$. For a generic such choice the moduli space $\cM^\cP(x_0,x_1;y_0,y_1)$ is a smooth manifold with boundary, of dimension 
$$
\dim\cM^\cP(x_0,x_1;y_0,y_1) = \CZ(x_0)+\CZ(x_1)-\CZ(y_0)-\CZ(y_1)-2n+2. 
$$

The signed count of elements in $0$-dimensional such moduli spaces defines a degree $-2n+2$ map 
$$
\Gamma:FC_*(H)^{\otimes 2} \to FC_*(3H)^{\otimes 2}. 
$$
The $1$-dimensional moduli spaces $\cM^\cP(x_0,x_1;y_0,y_1)$, which correspond to $\CZ(x_0)+\CZ(x_1)-\CZ(y_0)-\CZ(y_1)-2n+1=0$, are compact up to Floer breaking. By examining their oriented boundary we obtain the equation
$$
[\p, \Gamma] = \Gamma_{\p\cP},
$$
where $\Gamma_{\p\cP}:FC_*(H)^{\otimes 2}\to FC_*(3H)^{\otimes 2}$ is the degree $-2n+1$ map obtained by counting solutions of the parametrized Floer equation~\eqref{eq:FloereqnP} parametrized by the oriented boundary of $\cP$.

We now project onto $FC_*(3H;\im c)^{\otimes 2}$ and obtain the equation 
\begin{equation*} 
\Gamma_{\p\cP}=-\boldlambda_\cD\boldsymbol{\mu} + (1\otimes\boldsymbol{\mu})(\boldlambda_\cD\otimes 1)+(\boldsymbol{\mu}\otimes 1)(1\otimes\boldlambda_\cD)+ (\boldmu\otimes \boldmu)(1\otimes \boldc_{\cD,\cD^{\mathrm{op}}}\otimes 1).
\end{equation*}
Indeed, each of the sides of $\cP$ contributes as follows to $\Gamma_{\p\cP}$: 
\begin{itemize}
\item The contribution of the top horizontal side of $\cP$ is $-\boldlambda_\cD\boldsymbol{\mu}$. 
\item The contribution of the bottom vertical sides of $\cP$ is $(1\otimes\boldsymbol{\mu})(\boldlambda_\cD\otimes 1)+(\boldsymbol{\mu}\otimes 1)(1\otimes\boldlambda_\cD)$. 
\item The contribution of the top vertical sides of $\cP$ is zero in the quotient complex. 
\item The contribution of the bottom horizontal side of $\cP$ is $(\boldmu\otimes \boldmu)(1\otimes \boldc_{\cD,\cD^{\mathrm{op}}}\otimes 1)$. 
\end{itemize}
This implies the equality
$$
[\p, \Gamma] =-\boldlambda_\cD\boldsymbol{\mu} + (1\otimes\boldsymbol{\mu})(\boldlambda_\cD\otimes 1)+(\boldsymbol{\mu}\otimes 1)(1\otimes\boldlambda_\cD)+ (\boldmu\otimes \boldmu)(1\otimes \boldc_{\cD,\cD^{\mathrm{op}}}\otimes 1)
$$
as maps $FC_*(H)^{\otimes 2}\to FC_*(3H;\im c)^{\otimes 2}$. As before this implies the relation 
$$
\boldlambda_\cD\boldmu = (1\otimes\boldsymbol{\mu})(\boldlambda_\cD\otimes 1)+(\boldsymbol{\mu}\otimes 1)(1\otimes\boldlambda_\cD)+ (\boldmu\otimes \boldmu)(1\otimes \boldc_{\cD,\cD^{\mathrm{op}}}\otimes 1)
$$
as maps $\ol{F\H}_*(H)^{\otimes 2}\to F\H_*(3H;\im c)^{\otimes 2}$. By passing to the limit over $H$ we find the same relation as maps $\ol{S\H}_*(W)^{\otimes 2}\to S\H_*(W;\im c)^{\otimes 2}$, and further as maps $\ol{S\H}_*(W)^{\otimes 2}\to \ol{S\H}_*(W)^{\otimes 2}$ in view of strong $R$-essentiality. We conclude using that 
$$
\boldc_{\cD,\cD^{\mathrm{op}}}=-\boldc_{\cD^{\mathrm{op}},\cD}=-\boldlambda_\cD\boldeta,
$$
where the first equality is a direct consequence of the definition and the second equality is the content of Lemma~\ref{lem:cDopDlambda}. 
\end{proof}

\subsection{The unital anti-symmetry relation}  \label{sec:unital_anti_symmetry}

\begin{proposition} \label{prop:anti-symmetry_reduced}
The unital anti-symmetry relation holds: 
\begin{align*}
&(1\otimes\boldmu) (\tau\boldlambda_\cD\otimes 1)  
- (\boldmu\tau\otimes 1)(1\otimes\boldlambda_\cD) \\
& \hspace{5cm}+ (\boldmu\tau\otimes \boldmu)(1\otimes\boldlambda_\cD\boldeta\otimes 1) \\
&= \tau(1\otimes\boldmu\tau)(\boldlambda_\cD\otimes 1) 
- \tau (\boldmu\otimes 1)(1\otimes\tau\boldlambda_\cD) \\ 
& \hspace{5cm} - \tau(\boldmu\otimes \boldmu\tau)(1\otimes\boldlambda_\cD\boldeta\otimes 1).
\end{align*}
\end{proposition}

\begin{proof}[Proof of Proposition~\ref{prop:anti-symmetry_reduced}]
The proof resembles that of Proposition~\ref{prop:infinitesimal_reduced}. The relevant moduli space $\cM$ is that of genus 0 Riemann surfaces with 2 positive punctures and 2 negative punctures, constrained to lie on a circle and ordered now such that the negative punctures and the positive punctures alternate. This moduli space is diffeomorphic to an open interval and its compactification $\ol \cM$ is diffeomorphic to a closed interval whose ends correspond to nodal curves with two irreducible components and matching asymptotic markers at their common node, each containing 2 punctures, one negative and one positive. We choose a family of cylindrical ends over $\ol\cM$ which is compatible with splittings at the boundary. We denote $\circled{1},\circled{2}$ the positive punctures (inputs) and $1,2$ the negative punctures (outputs). 

\begin{figure}
\begin{center}
\includegraphics[width=.9\textwidth]{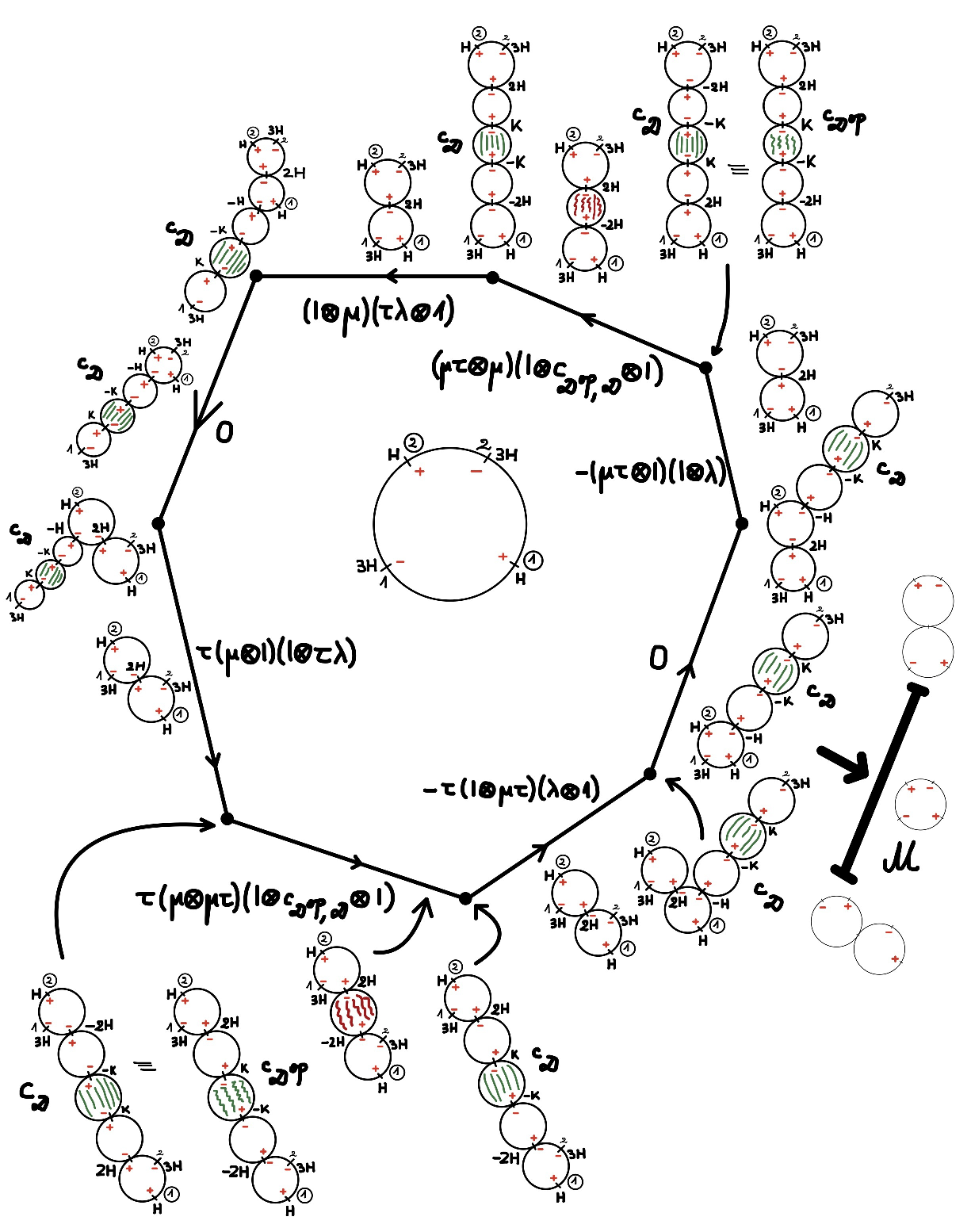}
\caption{Proof of {\sc (unital anti-symmetry)}.}
\label{fig:4-term-relation}
\end{center}
\end{figure}

We further consider an oriented octogon $\cP$ which parametrizes Floer data $(H_\tau,\beta_\tau)$, $\tau\in\cP$
on suitable nodal Riemann surfaces and which is fibered over $\ol\cM$ by the map which forgets the Floer data and contracts the unstable components of the underlying Riemann surfaces. The fibers of the projection $\cP\to\ol\cM$ are described as 
follows (Figure~\ref{fig:4-term-relation}):
\begin{itemize}
\item  The three sides above the ones marked ``0" project onto the endpoint of $\ol\cM$ for which the two irreducible components contain the punctures $\circled{2},2$ and respectively $\circled{1},1$.
\item The union of the interiors of $\cP$ and of the sides marked ``0" is fibered over $\cM$ with closed interval fibers.
\item The three sides below the ones marked ``0"  project onto the endpoint of $\ol\cM$ for which the two irreducible components contain the punctures $1,\circled{2}$ and respectively $2,\circled{1}$. 
\end{itemize}
The underlying Riemann surfaces for the points $\tau\in\cP$ are depicted in Figure~\ref{fig:4-term-relation}: in the interior of $\cP$ these are simply elements of $\cM$. On the boundary of $\cP$ these are nodal genus 0 Riemann surfaces with punctures and matching cylindrical ends at the punctures, and which may possibly have unstable components. These unstable components carry nontrivial Floer data and are stable if interpreted as Riemann surfaces with Floer data. 

The Floer data $(H_\tau,\beta_\tau)$, $\tau\in\cP$ consists of a constant (in $\tau$) $R$-essential Hamiltonian $H_\tau=H$ and of $1$-forms $\beta_\tau$, $\tau\in\cP$ which satisfy $d\beta_\tau\le 0$ and which have weights at the punctures as indicated in {\bf bold} in Figure~\ref{fig:4-term-relation}. E.g., on the top horizontal side of $\cP$, on the irreducible component which contains the punctures $\circled{2},2$ the weight is 1 at the positive puncture, 3 at the negative puncture, and 2 at the node. This field of $1$-forms is chosen to be compatible with splittings of Riemann surfaces over $\cP$. Moreover, each node is {\em directed}, meaning that it is labeled as a positive puncture (input) for one of the irreducible components which contain it, and as a negative puncture (output) for the other component. 

From this point on, the proof goes exactly as for Proposition~\ref{prop:infinitesimal_reduced}. The count of elements in $0$-dimensional moduli spaces of solutions to the Floer problem parametrized by $\cP$ defines a map $\Gamma:FC_*(H)^{\otimes 2}\to FC_*(3H;\im c_\cD)^{\otimes 2}$ of degree $-2n+2$. The count of elements in $0$-dimensional moduli spaces of solutions to the Floer problem parame\-trized by the oriented boundary of $\cP$ defines a map $\Gamma_{\p\cP}$ of degree $-2n+1$ with the same source and target, such that  
$$
[\p,\Gamma]=\Gamma_{\p\cP}. 
$$
Inspection of the boundary shows that, upon quotienting the target by $\im c_\cD$, we have  
\begin{align*}
\Gamma_{\p\cP} = &(1\otimes\boldmu) (\tau\boldlambda_\cD\otimes 1)  
- (\boldmu\tau\otimes 1)(1\otimes\boldlambda_\cD) \\
& \hspace{5cm}+ (\boldmu\tau\otimes \boldmu)(1\otimes \boldc_{\cD^{op},\cD}\otimes 1) \\
&- \tau(1\otimes\boldmu\tau)(\boldlambda_\cD\otimes 1) 
+ \tau (\boldmu\otimes 1)(1\otimes\tau\boldlambda_\cD) \\ 
& \hspace{5cm} + \tau(\boldmu\otimes \boldmu\tau)(1\otimes \boldc_{\cD^{op},\cD} \otimes 1).
\end{align*}
As in the previous proofs this implies the relation 
\begin{align*}
&(1\otimes\boldmu) (\tau\boldlambda_\cD\otimes 1)  
- (\boldmu\tau\otimes 1)(1\otimes\boldlambda_\cD) \\
& \hspace{5cm}+ (\boldmu\tau\otimes \boldmu)(1\otimes\boldc_{\cD^{op},\cD}\otimes 1) \\
&= \tau(1\otimes\boldmu\tau)(\boldlambda_\cD\otimes 1) 
- \tau (\boldmu\otimes 1)(1\otimes\tau\boldlambda_\cD) \\ 
& \hspace{5cm} - \tau(\boldmu\otimes \boldmu\tau)(1\otimes\boldc_{\cD^{op},\cD}\otimes 1)
\end{align*}
as maps $\ol{F\H}_*(H)^{\otimes 2}\to F\H_*(3H;\im c)^{\otimes 2}$. By passing to the limit over $H$ we find the same relation as maps $\ol{S\H}_*(W)^{\otimes 2}\to S\H_*(W;\im c)^{\otimes 2}$, and further as maps $\ol{S\H}_*(W)^{\otimes 2}\to \ol{S\H}_*(W)^{\otimes 2}$ in view of strong $R$-essentiality. We conclude using that $\boldc_{\cD^{\mathrm{op}},\cD}=\boldlambda_\cD\boldeta$ (Lemma~\ref{lem:cDopDlambda}).  
\end{proof}

\subsection{Cocommutativity} \label{sec:cocomm}

Graded commutativity of the product $\boldsymbol{\mu}$ is standard. We address cocommutativity of the coproduct $\boldsymbol{\lambda}_\cD$. 

\begin{proposition} \label{prop:cocomm-coprod}
The coproduct $\boldlambda_\cD$ is cocommutative on $\ol{S\H}_*(\p W)$, meaning that 
$$
\boldlambda_\cD + \tau\boldlambda_\cD = 0. 
$$
\end{proposition} 

\begin{proof}
Let $H$ be an $R$-essential Hamiltonian and let $\cD$ be continuation data. We set up a Floer problem parametrized by the square $[0,1]\times [0,1]$ as follows (see Figure~\ref{fig:cocommutativity-SH}). The underlying Riemann surface is a genus 0 curve with 3 punctures, 1 positive and 2 negative, together with cylindrical ends which depend on $(\sigma,\tau)\in[0,1]\times [0,1]$. The cylindrical end at the positive puncture is fixed with respect to some given parametrization of the Riemann surface, and this uniquely identifies its complement with the disc $D^2$. The two negative punctures are fixed along each of the vertical sides, and they move along a half-Dehn twist as we traverse the square along segments $[0,1]\times \{\tau\}$. Thus the two negative punctures get exchanged as we move from one vertical side to the other. One important point in the construction is that the negative punctures are ordered, i.e. labeled by $1$ and $2$. When traversing from the left vertical side to the right vertical side, this labeling is switched. We also choose cylindrical ends at the negative punctures over $[0,1]\times [0,1]$ in such a way that over the two vertical sides the cylindrical ends are switched from one puncture to the other.

\begin{figure}
\begin{center}
\includegraphics[width=.5\textwidth]{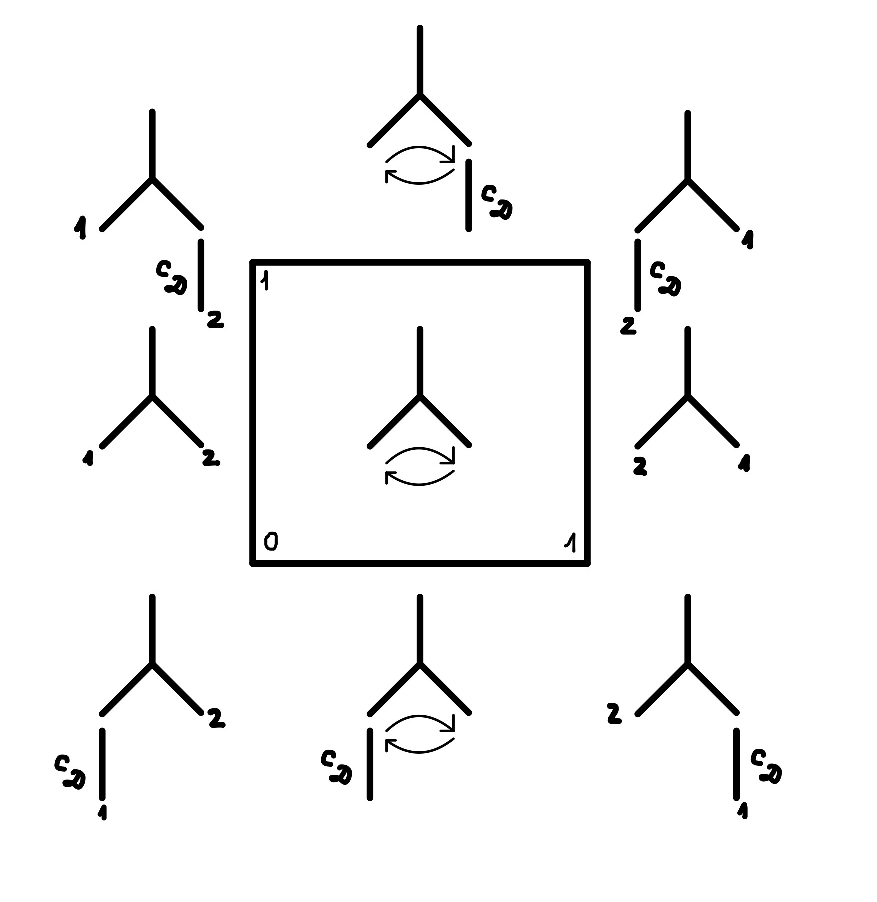}
\caption{Cocommutativity of $\boldlambda_\cD$.}
\label{fig:cocommutativity-SH}
\end{center}
\end{figure}

On the first vertical side the Floer data corresponds to the one defining the coproduct $\boldlambda_\cD$. On the second vertical side the Floer data corresponds to the one defining the composition $-\tau\boldlambda_\cD$ between the tensor product twist and the coproduct. 

Denoting $\Gamma, \Gamma_\p:FC_*(H)\to FC_*(2H)\otimes FC_*(2H)$ the operations defined by the count of the elements of 0-dimensional moduli spaces of the Floer problem parametrized by $[0,1]\times [0,1]$, respectively by $\p([0,1]\times [0,1])$, we obtain a relation 
$$
[\p,\Gamma] = \Gamma_\p.
$$

We then have 
$$
\Gamma_\p= - \tau\boldlambda_\cD -\boldlambda_\cD
$$
as maps $FC_*(H)\to FC_*(2H;\im c_\cD)^{\otimes 2}$, because the operations that are read along the two horizontal sides have image contained inside $\im c_\cD \otimes FC_*(2H) + FC_*(2H)\otimes \im c_\cD$. Passing to homology and in the limit over $H$ we obtain the equality $\boldlambda_\cD+ \tau\boldlambda_\cD=0$ as maps $\ol{S\H}_*(W)\to S\H_*(W;\im c_\cD)^{\otimes 2}$. The conclusion follows by $R$-essentiality. 
\end{proof}

\section{Splittings of Rabinowitz Floer homology}\label{sec:splittings}

The goal of this section is to exhibit splittings of the Rabinowitz Floer homology group which involve reduced homology and cohomology and are compatible with the product and coproduct. 

In~\S\ref{sec:RFH-cone} we recall from~\cite{CO-cones} the cone description of the product, {\alex and we explain the definition of the coproduct}, on Rabinowitz Floer homology of a Liouville domain. In~\S\ref{sec:RFH-colimit} we give an invariant definition of Rabinowitz Floer homology as a colimit in groupoids, which allows us to describe the various identifications involved in the proof of invariance for the cone description. In~\S\ref{sec:RFH-les} we discuss the properties of the long exact sequence of the pair $(V,\p V)$ for a Liouville domain $V$ from the perspective of product and coproduct structures. In~\S\ref{subsec:splittings} we specialize to strongly $R$-essential Weinstein domains and show that this long exact sequence, rephrased as a short exact sequence in terms of reduced homology, admits non-canonical splittings. In~\S\ref{sec:splittings_are_not_canonical} we discuss the failure of canonicity of the splitting maps. We conclude with~\S\ref{sec:unital-from-ass} where we show how the unital infinitesimal relation in reduced homology is implied by associativity of the product on the cone.

\subsection{Cone description of Rabinowitz Floer homology} \label{sec:RFH-cone}

Let $V$ be a Liouville domain of dimension $2n$ and denote $S\H_*(\p V)=SH_{*+n}(\p V)$ the shifted Rabinowitz Floer homology of $\p V$. We proved in~\cite{CHO-PD} that this is a biunital coFrobenius bialgebra with degree $0$ product $\boldmu$ and degree $1-2n$ coproduct $\boldlambda$. 

It was observed in~\cite{CO} that Rabinowitz Floer homology can be alternatively described as a cone, and Venkatesh~\cite{Venkatesh} went a step further by adopting the cone perspective as definition. This proved effective in contexts where action filtration arguments were not readily available. 

In~\cite{CO-cones} we described the product structure on Rabinowitz Floer homology from the cone perspective. Dualizing that construction yields a description of the coproduct, and we now summarize the two. 

Let $H$ be a Hamiltonian which is admissible for symplectic homology and denote $\cA=FC_{*+n}(H)$, $\cM=\cA^\vee[2n]=FC_{*+n}(-H)$, $\cA'=FC_{*+n}(2H)$, $\cM'=\cA^{\prime\vee}[2n]=FC_{*+n}(-2H)$. Given a choice of continuation data $\cD$ from (a Morse truncation of) $-H$ to $H$, we denote  
$$
c=c_\cD:\cM\to \cA
$$
the corresponding continuation map, and we use the same notation for the continuation maps $\cM\to \cA'$, $\cM'\to \cA$, and $\cM'\to \cA'$. Note that $c^\vee$ is canonically identified with $c_{\cD^{op}}$, the continuation map determined by the reverse continuation data $\cD^{op}$.

\subsubsection{Product on the cone} In~\cite{CO-cones} we constructed maps $\mu:\cA\otimes \cA\to \cA'$ (product), $m_R:\cM'\otimes \cA\to \cM$ and $m_L:\cA\otimes \cM'\to \cM$ (module structure) obtained by dualizing $\mu$ at the output and at one of the two inputs, $\sigma:\cM'\otimes\cM'\to \cM$ (secondary product), $\tau_L:\cM'\otimes \cA\to \cA'$ and $\tau_R:\cA\otimes \cM'\to \cA'$ (secondary module structure), and finally $\beta:\cM'\otimes\cM'\to \cA'$. We showed in~\cite{CO-cones} that, up to considering suitable action windows, these maps fit together into a linear map 
$$
{\alex \boldmu_{\cD}:Cone(c_\cD)\otimes Cone(c_\cD)\to Cone (c_\cD)}
$$
which induces in homology and in the limit a product $\boldmu_{{\alex \cD}}$ isomorphic to the product $\boldmu$ on $S\H_*(\p V)$. 
These maps are an action truncated version of the notion of $A_2$-structure studied in~\cite{CO-cones}. 

For action reasons the image of the map $\beta$ is contained in the 0-energy sector $FC_*^{=0}(2H)$. In case $V$ is an $R$-essential Weinstein domain of dimension $2n\ge 6$, 
the map $\beta$ vanishes identically. 
See the proof of Lemma~\ref{lem:lambda-imc}.

\subsubsection{Secondary continuation map and $A_2^+$-structures} \label{sec:A2+structure}
The formalism of $A_2^+$-structures studied in~\cite[\S7]{CO-cones} is designed to infer the previous operations from the following  more elementary ones. Given the continuation data $\cD$ we define the following maps:  
\[
\tag{\mbox{\sc continuation bivector}} c_0\in \cA\otimes \cA.
\]
This is the bivector dual to the continuation map ${\alex c_\cD}$. Note that $\tau c_0$ is then the bivector dual to the continuation map ${\alex c^\vee=c_{\cD^{op}}}$. 
\[
\tag{\mbox{\sc secondary continuation bivector}}Q_0\in \cA\otimes \cA.
\]
This is the bivector dual to the secondary continuation map $\vec\boldc$ and satisfies $\tau c_0-c_0=[\p,Q_0]$.  
\[
\tag{\mbox{\sc product}} \mu:\cA\otimes \cA\to \cA'
\]
of degree $0$. This is the usual pair of pants product in Floer theory. 
\[
\tag{\mbox{\sc coproduct}} \lambda{\alex =\lambda_\cD}:\cA\to \cA'\otimes \cA'
\]
of degree $1-2n$. This is defined from a parametrized Floer problem for pairs of pants with parameter space the 1-simplex by splitting off the continuation map $c=c_\cD$ at the negative punctures at the boundary of the simplex. 
\[
\tag{\mbox{\sc cubic vector}} B:R\to \cA'\otimes \cA'\otimes \cA'
\]
of degree $2-4n$ and cyclically symmetric. This is defined from a parametrized Floer problem for pairs of pants with parameter space the 2-simplex by splitting off the continuation map $c=c_\cD$ at the negative punctures at the boundary of the simplex.

{\alex In the rest of this subsection we will omit the distinction between $\cA$ and $\cA'$, and refer to~\cite{CO-cones} for a detailed discussion.}

This data determines operations $m_L$, $m_R$, $\sigma$, $\tau_L$, $\tau_R$, $\beta$ as in~\cite[\S7.2]{CO-cones}. The key idea is that, upon dualizing $\mu$, $\lambda$ or $B$ at some of their inputs or outputs, one needs to further interpolate between the continuation map ${\alex c^\vee=c_{\cD^{op}}}$ arising from the dualization and the continuation map ${\alex c=c_\cD}$ involved in the cone. This interpolation makes use of the secondary continuation bivector $Q_0$. The outcome is the product $\boldmu_\cD$ on $Cone({\alex c_\cD})$.

{\alex

\subsubsection{Poincaré duality on the cone} The above maps can be dualized in order to obtain a cone description of the coproduct on Rabinowitz Floer homology. As a preliminary step, we need to describe Poincaré duality for Rabinowitz Floer homology at the cone level. 

Following~\cite[\S7.4]{CO-cones} we have canonical chain isomorphisms 
$$
  Cone(c)\stackrel{\varphi}\longrightarrow Cone(c^\vee)\stackrel{\psi}\longrightarrow Cone(c)^\vee[-1].
$$
The first chain isomorphism is 
$$
  \varphi :  
  \left(\cA\oplus\cA^\vee[-1],\begin{pmatrix}\p & c \\ 0 & -\p^\vee\end{pmatrix}\right) 
  \stackrel{\tiny \begin{pmatrix}1 & -\vec\boldc \\ 0 & 1\end{pmatrix}}{\longrightarrow}
  \left(\cA\oplus\cA^\vee[-1],\begin{pmatrix}\p & c^\vee \\ 0 & -\p^\vee\end{pmatrix}\right)\!,
$$
and the second one simply interchanges the two direct summands,
$$
\psi:\left(\cA\oplus\cA^\vee[-1],\begin{pmatrix}\p & c^\vee \\ 0 & -\p^\vee\end{pmatrix}\right) 
  \stackrel{\tiny \begin{pmatrix}0 & 1 \\ 1 & 0\end{pmatrix}}{\longrightarrow}
  \left(\cA^\vee[-1]\oplus\cA,\begin{pmatrix}-\p^\vee & 0 \\ c^\vee & \p\end{pmatrix}\right).
$$

\begin{proposition} \label{prop:PD-cone}
The composition 
$$
I=\psi\circ\varphi=\tiny \begin{pmatrix} 0 & 1 \\ 1 & -\vec \boldc \end{pmatrix}:Cone(c)\to Cone(c)^\vee[-1]
$$ 
induces in the limit in homology the Poincaré duality isomorphism.
\end{proposition}

\begin{remark}
In~\cite[Theorem~7.14]{CO-cones} we gave an algebraic statement of duality in this framework. However, we did not explicitly prove in \emph{loc.\,cit.} that the algebraic duality statement was equivalent to the Poincaré duality theorem from~\cite{CHO-PD}. Proposition~\ref{prop:PD-cone} fills that gap. 
\end{remark}

\begin{proof} The proof relies on~\cite[\S5.1]{CO-cones}. Let $H_\vee=H_{-\mu,\mu}$ be a {\sc V}-shaped Hamiltonian as in~\cite[\S5.1]{CO-cones} of slopes $-\mu<0<\mu$, and $L=L_{-\mu}$ the linear Hamiltonian of slope $-\mu<0$. (For an illustration, see the second row of the diagram at the end of this proof.) The Poincaré duality isomorphism is proved in {\it loc.\,cit.} to be equivalent to an isomorphism of rings $S\H_*(\cH_\vee)\simeq S\H_*(\cH_\wedge)$, where $S\H_*(\cH_\vee)$ is the limit of homologies of the cones of continuation maps $c_{L,H_\vee}:L\to H_\vee$, and $S\H_*(\cH_\wedge)$ is the limit of homologies of the cones of continuation maps $c_{-H_\vee,-L}:-H_\vee\to -L$. Moreover, this isomorphism arises as a limit of chain homotopy equivalences between the cones $Cone(c_{L,H_\vee})\simeq Cone(c_{-H_\vee,-L})$ which are induced by continuation maps $c_{L,-H_\vee}:L\to -H_\vee$, $c_{H_\vee,-L}:H_\vee\to -L$, and a homotopy of homotopies between $c_{H_\vee,-L}\circ c_{L,H_\vee}$ and $c_{-H_\vee,-L}\circ c_{L,-H_\vee}$~\cite[Theorem~5.1]{CO-cones}. On the other hand, we have isomorphisms $Cone(c_{L,H_\vee})\simeq Cone (c_\cD)$ and $Cone(c_{-H_\vee,-L})\simeq Cone(c_{\cD^{op}})$, again by continuation maps and homotopies. The following three easy algebraic facts, also proved in~\cite[\S4.2]{CO}, will allow us to conclude:

Let $A$, $B$, $A'$, $B'$ be chain complexes. Given two chain maps $c:A\to B$ and $c':A'\to B'$, a morphism between them is a triple $(f,g,h)$ consisting of chain maps $f:A\to A'$, $g:B\to B'$ and a chain homotopy $h:A\to B'$ of degree 1 satisfying $[\p,h]=gc-c'f$. 
$$
\xymatrix{A \ar[r]^c \ar[d]_f \ar@{.>}[dr]^-h & B \ar[d]^g \\
A' \ar[r]_{c'} & B'
}
$$
The triple $(f,g,h)$ induces a chain map 
$
{\tiny \begin{pmatrix} g & h \\ 0 & f\end{pmatrix}} :Cone(c)\to Cone(c'). 
$
\begin{enumerate}
\item If $h$ is replaced by another homotopy $\tilde h$ such that $\tilde h-h$ is a boundary in $\Hom(A,B')$, then the induced morphisms between the cones are chain homotopic.
\item If $f':A\to A'$ is another chain map that is chain homotopic to $f$, and $k:A\to A'$ is a chain homotopy between them such that $f-f'=[\p,k]$, then the triple $(f',g,h+c'k)$ induces a map $Cone(c)\to Cone(c')$ that is chain homotopic to the one induced by $(f,g,h)$. Specifically, the homotopy is $\tiny \left(\begin{array}{cc} 0 & 0 \\ 0 & k \end{array}\right)$. 
$$
\xymatrix{A\ar@{=}[r] \ar[d]_{f'} \ar@{.>}[dr]^-k & A \ar[r]^c \ar[d]_f \ar@{.>}[dr]^-h & B \ar[d]^g \\
A \ar@{=}[r] & A' \ar[r]_{c'} & B'
}
$$
An analogous statement holds for $g$ replaced by a chain homotopic map, with a specified chain homotopy between the two.
\item Morphisms compose. Let $\tilde A$, $\tilde B$ be chain complexes, let $\tilde c:\tilde A\to \tilde B$ be a chain map, and consider a morphism $(\tilde f,\tilde g,\tilde h)$ from $\tilde c$ to $c$. The composition $(f,g,h)\circ (\tilde f,\tilde g,\tilde h)$ is defined as $(f\tilde f,g\tilde g,g\tilde h+h\tilde f)$.
$$
\xymatrix{\tilde A \ar[r]^{\tilde c} \ar[d]_{\tilde f} \ar@{.>}[dr]^-{\tilde h} & \tilde B \ar[d]^{\tilde g} \\
A \ar[r]^c \ar[d]_f \ar@{.>}[dr]^-h & B \ar[d]^g \\
A' \ar[r]_{c'} & B'
}
$$
\end{enumerate}

To finish the proof, we apply the previous three algebraic facts to the following diagram, where $H_\vee$ and $L$ are as in the previous paragraph, and $H=H_\mu$ is a Hamiltonian of slope $\mu$ that is admissible for symplectic homology as described at the beginning of this subsection. For a lighter notation in the diagram, the Hamiltonians stand for the corresponding Floer complexes. Also, in the same diagram we have illustrated the shapes of the various Hamiltonians and homotopies that give rise to continuation maps. The assumption from (1) that the difference between two homotopies between continuation maps obtained from homotopies of homotopies of Hamiltonians is a boundary is a general fact in Floer theory~\cite[Lemma~4.7]{CO}.
$$
\xymatrix
@C=20pt
@R=20pt
{\minusHammuaxis{$-\mu$} & -H \ar[rrrrr]^{c_\cD}_{\tiny \homotopyHamminusmutomu{}{}} \ar[d] \ar@{.>}[drrrrr] &&&&& H \ar[d] & \Hammuaxis{$\mu$} \\
\HamLlambda{$-\mu$} & L \ar[rrrrr]^{c_{L,H_\vee}}_{\tiny \homotopyHamLlambdatoHamlambdamu{}{}} \ar[d]^{c_{L,-H_\vee}} \ar@{.>}[drrrrr] &&&&& H_\vee \ar[d]_{c_{H_\vee,-L}} & \Hamlambdamuaxis{$-\mu$}{$\mu$} \\
\minusHamlambdamuaxis{$\mu$}{$-\mu$} & -H_\vee \ar[rrrrr]^{c_{-H_\vee,-L}}_{\tiny \homotopyminusHamlambdamutominusHamLlambda{}{}} \ar[d] \ar@{.>}[drrrrr] &&&&& -L \ar[d] & \minusHamLlambda{$\mu$} \\
\minusHammuaxis{$-\mu$} & -H \ar[rrrrr]^{c_{\cD^{op}}}_{\tiny \homotopyHamminusmutomu{}{}} &&&&& H & \Hammuaxis{$\mu$}
}
$$
This concludes the proof. 
\end{proof}

\subsubsection{Coproduct on the cone} \label{sec:coproduct-on-cone}

The chain level product $\boldmu_\cD$ on $Cone(c_\cD)$ is algebraically dual to a chain level
coproduct $\tau \boldmu_\cD^\vee$ on $Cone(c_\cD)^\vee$, which induces a shifted coproduct on $Cone(c_\cD)^\vee[-1]$. Under the chain isomorphism $I$ from Proposition~\ref{prop:PD-cone}, this corresponds to a coproduct $\boldlambda_\cD$ on $Cone(c_\cD)$. This coproduct is in turn algebraically dual to a product $\boldlambda_\cD^\vee\tau$ on $Cone(c_\cD)^\vee$, which induces a shifted product on $Cone(c_\cD)^\vee[-1]$. Under the chain isomorphism $I$, this corresponds to the original product $\boldmu_\cD$ on $Cone(c_\cD)$. 

The family of maps $\boldlambda_\cD$ on $Cone(c_\cD)$ induces in the limit on homology a coproduct that we also denote $\boldlambda_\cD$. 

\begin{proposition} \label{prop:boldlambdaD-restricts-to-lambdaD}
The coproduct $\boldlambda_\cD$ is isomorphic to the coproduct $\boldlambda$ on $S\H_*(\p V)$. 
\end{proposition}

\begin{proof}
Since $\boldlambda$ is also obtained by dualizing $\boldmu$ and transporting it back via the Poincaré duality map, the statement follows from Proposition~\ref{prop:PD-cone}, which states that the chain isomorphism $I$ induces in the limit on homology the Poincaré duality isomorphism. 
\end{proof}

By construction, the components of the product $\boldmu_\cD$ and the coproduct $\boldlambda_\cD$ on the cone are obtained by dualizing the operations $\mu$, $\lambda=\lambda_\cD$ of the $A_2^+$-structure at some of their inputs and outputs, together with contributions from the secondary copairing $Q_0$. The following two components will be relevant:
\begin{itemize}
\item the component $\cA^\vee[-1]\otimes \cA^\vee[-1]\to \cA^\vee[-1]$ of the product $\boldmu_\cD$ is $\lambda_{\cD^{op}}^\vee\tau$. This follows by comparing the formulas from~\cite[\S7.2]{CO-cones} with Corollary~\ref{cor:lambdaDprimeD}, and is also expressed by Theorem~\ref{thm:m-components-RFH}(2) below. 
\item the component $\cA\to\cA\otimes\cA$ of the coproduct $\boldlambda_\cD$ is $\lambda_\cD$. This is seen by tracing directly through the isomorphism $I$ above, and we now explain this computation. 
We seek to compute the component $\boldlambda^+_{++}$ of the coproduct $\boldlambda_\cD$ given by the following diagram 
$$
\xymatrix
@C=80pt
{
\cA \ar[r]^-{\boldlambda^+_{++}} \ar[d]^{incl} \ar@/_40pt/[dd]_{incl} & \cA\otimes \cA \\
\cA\oplus \cA^\vee[-1] \ar[r]^-{\boldlambda_\cD} \ar[d]^I
& 
(\cA\oplus \cA^\vee[-1])\otimes (\cA\oplus \cA^\vee[-1]) \ar[u]_{proj\otimes proj} \\
\cA^\vee[-1]\oplus \cA \ar[r]^-{\tau\boldmu^\vee} & (\cA^\vee[-1]\oplus \cA) \otimes (\cA^\vee[-1]\oplus \cA)
\ar[u]_{I^{-1}\otimes I^{-1}}
}
$$
We now refer to~\cite[\S7]{CO-cones}. 
Let $s:\cA^\vee[-1]\to\cA^\vee$ be the shift and let $Q:\cA^\vee\to \cA$ be obtained by turning into an input the first output of the secondary continuation bivector $Q_0$ from~\S\ref{sec:A2+structure}. We then have $\vec \boldc=Qs$,  
$I=\tiny \begin{pmatrix} 0 & 1 \\ 1 & -Qs \end{pmatrix}$ and $I^{-1}= \tiny \begin{pmatrix} Qs & 1 \\ 1 & 0 \end{pmatrix}$. (In~\cite[\S7]{CO-cones} we used the shift implicitly and wrote $Q$ instead of the current $Qs$.) 
 The $A_2^+$-structure consisting of $c_0$, $Q_0$, $\mu$, and $\lambda=\lambda_\cD$ determines operations $m_L:\cA\otimes\cA^\vee\to \cA^\vee$, $m_R:\cA^\vee\otimes\cA\to \cA^\vee$, $\sigma:\cA^\vee\otimes\cA^\vee\to \cA^\vee$ (and also $\tau_L$, $\tau_R$ which are irrelevant for this discussion). The operation $\sigma$ is dual to $-\tau\lambda_{\cD^{op}}=-\tau\big(\lambda + (\mu\otimes 1)(1\otimes Q_0) - (1\otimes\mu)(\tau Q_0\otimes1)\big)$, and the operations
$m_L$ and $m_R$ are obtained by dualizing $\mu$ at one of the inputs and at the output as in the next diagram. (In this diagram, inputs or outputs in $\cA$ are represented by downward pointing arrows, and inputs or outputs in $\cA^\vee$ are represented by upward pointing arrows. The labels ``1'' and ``2'' stand for the order of the inputs, which corresponds to reading the leaves of each tree clockwise starting from the output.) 
 
\smallskip 
 
\begin{center}

\begin{tikzpicture}[scale=.25]

\draw (0,0) -- (1,-1) -- (2,0);
\draw (1,-1) -- (1,-2.5);
\node at (1,-5.5) {\tiny $\cA \otimes \cA\to \cA$};
\node at (-1,-2)[scale=0.75] {\tiny $\mu$};
\draw [arrows={[sep] - Stealth[length=3pt]}]        (0,1) -- (0,.2);
\draw [arrows={[sep] - Stealth[length=3pt]}]        (2,1)   -- (2,.2);
\draw [arrows={[sep] - Stealth[length=3pt]}]        (1,-2.7) -- (1,-3.5);
\node at (2.5,0)[scale=0.5] {\tiny $2$};
\node at (-0.5,0)[scale=0.5] {\tiny $1$};

\pgftransformxshift{12cm}
\draw (0,0) -- (1,-1) -- (2,0);
\draw (1,-1) -- (1,-2.5);
\node at (1,-5.5) {\tiny $\cA \otimes \cA^\vee\to \cA^\vee$};
\node at (-1,-2)[scale=0.75] {\tiny $m_L$};
\draw [arrows={[sep] - Stealth[length=3pt]}]        (0,.2) -- (0,1);
\draw [arrows={[sep] - Stealth[length=3pt]}]        (2,1)   -- (2,.2);
\draw [arrows={[sep] - Stealth[length=3pt]}]        (1,-3.5) -- (1,-2.7);
\node at (2.5,0)[scale=0.5] {\tiny $1$};
\node at (1.5,-2.5)[scale=0.5] {\tiny $2$};

\pgftransformxshift{12cm}
\draw (0,0) -- (1,-1) -- (2,0);
\draw (1,-1) -- (1,-2.5);
\node at (1,-5.5) {\tiny $\cA^\vee \otimes \cA\to \cA^\vee$};
\node at (-1,-2)[scale=0.75] {\tiny $m_R$};
\draw [arrows={[sep] - Stealth[length=3pt]}]        (0,1) -- (0,.2);
\draw [arrows={[sep] - Stealth[length=3pt]}]        (2,.2)   -- (2,1);
\draw [arrows={[sep] - Stealth[length=3pt]}]        (1,-3.5) -- (1,-2.7);
\node at (2.5,0)[scale=0.5] {\tiny $1$};
\node at (1.5,-2.5)[scale=0.5] {\tiny $2$};

\end{tikzpicture}
\end{center}
 
\smallskip 

Consider the shifts 
$$
\xymatrix
@C=50pt 
{
\cA[1] \ar@<-.5ex>[d]_\omega & \cA^\vee \ar@<-.5ex>[d]_\omega \\
\cA \ar@<-.5ex>[u]_s & \cA^\vee[-1] \ar@<-.5ex>[u]_s
}
$$
of degrees $|\omega|=1$, $|s|=-1$. The components of the product $\boldmu_\cD$ that will appear in the computation of $\boldlambda^+_{++}$ are the following:
\begin{align*}
\cA\otimes \cA^\vee[-1]\to \cA^\vee[-1] \, : & \quad \omega m_L(1\otimes s),\\
\cA^\vee[-1]\otimes \cA \to \cA^\vee[-1] \, : & \quad \omega m_R (s\otimes 1),\\
\cA^\vee[-1]\otimes \cA^\vee[-1]\to \cA^\vee[-1]\, : & \quad -\omega\sigma(s\otimes s).
\end{align*}
(The minus sign in the last formula is~\cite[Equation~(12)]{CO-cones}.) The shift of the dual operation $\boldmu_\cD^\vee$ has components 
\begin{align*}
\cA\to \cA^\vee[-1]\otimes \cA \, : & \quad (\omega\otimes\omega)(\omega m_L(1\otimes s))^\vee s = -(\omega\otimes 1)m_L^\vee,\\
\cA\to \cA\otimes\cA^\vee[-1] \, : & \quad (\omega\otimes\omega)(\omega m_R (s\otimes 1))^\vee s = (1\otimes \omega)m_R^\vee,\\
\cA\to\cA\otimes \cA\, : & \quad (\omega \otimes\omega) (-\omega \sigma (s\otimes s ))^\vee s = -\sigma^\vee,
\end{align*}
and the 
corresponding coproduct on $Cone(c)^\vee[-1]$, given by the shift of $\tau\boldmu_\cD^\vee$, has components 
\begin{align*}
\cA\to \cA\otimes\cA^\vee[-1] \, : & \ -\tau(\omega\otimes 1)m_L^\vee = -(1\otimes\omega)\tau m_L^\vee,\\
\cA\to \cA^\vee[-1]\otimes \cA \, : & \ \tau(1\otimes \omega)m_R^\vee = (\omega\otimes 1)\tau m_R^\vee,\\
\cA\to\cA\otimes\cA\, : & \ -\tau\sigma^\vee \!=\! \lambda_{\cD^{op}}\!=\!\lambda + (\mu\otimes 1)(1\otimes Q_0) \\
& \qquad \qquad \qquad \qquad - (1\otimes\mu)(\tau Q_0\otimes1).
\end{align*}
The component $\boldlambda^+_{++}$ of $\boldlambda_\cD$ is then given by
\begin{align*}
\boldlambda^+_{++} & = \lambda + (\mu\otimes 1)(1\otimes Q_0) - (1\otimes\mu)(\tau Q_0\otimes1) \\
& \qquad \qquad - (1\otimes Qs)(1\otimes \omega)\tau m_L^\vee + (Qs\otimes 1)(\omega\otimes 1)\tau m_R^\vee \\
& = \lambda + (\mu\otimes 1)(1\otimes Q_0) - (1\otimes\mu)(\tau Q_0\otimes1) \\
& \qquad \qquad - (1\otimes Q)\tau m_L^\vee + (Q\otimes 1)\tau m_R^\vee \\
& = \lambda + (\mu\otimes 1)(1\otimes Q_0) - (1\otimes\mu)(\tau Q_0\otimes1) \\
& \qquad \qquad -(\mu\otimes 1)(1\otimes Q_0) + (1\otimes\mu)(\tau Q_0\otimes 1) \\
& = \lambda.
\end{align*}
The third equality is a consequence of $(1\otimes Q)\tau m_L^\vee=(\mu\otimes 1)(1\otimes Q_0)$ and $(Q\otimes 1)\tau m_R^\vee=(1\otimes\mu)(\tau Q_0\otimes1)$. The equality $(1\otimes Q)\tau m_L^\vee=(\mu\otimes 1)(1\otimes Q_0)$ is explained pictorially in the following diagram. (For each tree, the labels ``1'' and ``2'' stand for the order of the inputs or of the outputs.)

\begin{center}

\begin{tikzpicture}[scale=.25]

\draw (0,0) -- (1,-1) -- (2,0);
\draw (1,-1) -- (1,-2.5);
\node at (1,-5.5) {\tiny $\cA \otimes \cA^\vee\to \cA^\vee$};
\node at (-1,-2)[scale=0.75] {\tiny $m_L$};
\draw [arrows={[sep] - Stealth[length=3pt]}]        (0,.2) -- (0,1);
\draw [arrows={[sep] - Stealth[length=3pt]}]        (2,1)   -- (2,.2);
\draw [arrows={[sep] - Stealth[length=3pt]}]        (1,-3.5) -- (1,-2.7);
\node at (2.5,0)[scale=0.5] {\tiny $1$};
\node at (1.5,-2.5)[scale=0.5] {\tiny $2$};

\pgftransformxshift{9cm}
\draw (0,0) -- (1,-1) -- (2,0);
\draw (1,-1) -- (1,-2.5);
\node at (1,-5.5) {\tiny $\cA\to \cA^\vee\otimes\cA$};
\node at (-1,-2)[scale=0.75] {\tiny $m_L^\vee$};
\draw [arrows={[sep] - Stealth[length=3pt]}]        (0,1)   -- (0,.2);
\draw [arrows={[sep] - Stealth[length=3pt]}]        (2,.2) -- (2,1);
\draw [arrows={[sep] - Stealth[length=3pt]}]        (1,-2.7) -- (1,-3.5); 
\node at (2.5,0)[scale=0.5] {\tiny $1$};
\node at (1.5,-2.5)[scale=0.5] {\tiny $2$};

\pgftransformxshift{9cm}
\draw (0,0) -- (1,-1) -- (2,0);
\draw (1,-1) -- (1,-2.5);
\node at (1,-5.5) {\tiny $\cA\to \cA\otimes\cA^\vee$};
\node at (-1,-2)[scale=0.75] {\tiny $\tau m_L^\vee$};
\draw [arrows={[sep] - Stealth[length=3pt]}]        (0,1)   -- (0,.2);
\draw [arrows={[sep] - Stealth[length=3pt]}]        (2,.2) -- (2,1);
\draw [arrows={[sep] - Stealth[length=3pt]}]        (1,-2.7) -- (1,-3.5); 
\node at (2.5,0)[scale=0.5] {\tiny $2$};
\node at (1.5,-2.5)[scale=0.5] {\tiny $1$};

\pgftransformxshift{10cm}
\draw (0,0) -- (1,-1);
\draw [arrows={- Stealth[length=3pt]}] (1,-1) -- (2,0);
\draw (1,-1) -- (1,-2.5);
\draw [arrows={Stealth[length=3pt,reversed] - Stealth[length=3pt]}] (2.1,.2)  .. controls (3.1,1.5) .. (4.1,.2);
\node at (1,-5.5) {\tiny $\cA\to \cA\otimes\cA$};
\node at (-2,-2)[scale=0.75] {\tiny $(1\otimes Q)\tau m_L^\vee$};
\draw [arrows={[sep] - Stealth[length=3pt]}]        (0,1)   -- (0,.2);
\draw [arrows={[sep] - Stealth[length=3pt]}]        (1,-2.7) -- (1,-3.5); 
\node at (4.4,0)[scale=0.5] {\tiny $2$};
\node at (1.5,-2.5)[scale=0.5] {\tiny $1$};
\node at (3.1,1.7)[scale=.65] {\tiny $Q$};

\pgftransformxshift{11cm}
\draw (0,0) -- (1,-1);
\draw [arrows={- Stealth[length=3pt,reversed]}] (1,-1) -- (2,0);
\draw (1,-1) -- (1,-2.5);
\draw [arrows={Stealth[length=3pt] - Stealth[length=3pt]}] (2.1,.2)  .. controls (3.1,1.5) .. (4.1,.2);
\node at (1,-5.5) {\tiny $\cA\to \cA\otimes\cA$};
\node at (-2.5,-2)[scale=0.75] {\tiny $(\mu\otimes 1)(1\otimes Q_0)$};
\draw [arrows={[sep] - Stealth[length=3pt]}]        (0,1)   -- (0,.2);
\draw [arrows={[sep] - Stealth[length=3pt]}]        (1,-2.7) -- (1,-3.5); 
\node at (4.4,0.5)[scale=0.5] {\tiny $2$};
\node at (1.5,-2.5)[scale=0.5] {\tiny $1$};
\node at (3.1,1.7)[scale=.65] {\tiny $Q_0$};
\node at (1.5,-1.3)[scale=.65] {\tiny $\mu$};

\end{tikzpicture}
\end{center}
The equality $(Q\otimes 1)\tau m_R^\vee=(1\otimes\mu)(\tau Q_0\otimes1)$ is explained pictorially in the following diagram.

\begin{center}

\begin{tikzpicture}[scale=.25]

\draw (0,0) -- (1,-1) -- (2,0);
\draw (1,-1) -- (1,-2.5);
\node at (1,-5.5) {\tiny $\cA^\vee \otimes \cA\to \cA^\vee$};
\node at (-1,-2)[scale=0.75] {\tiny $m_R$};
\draw [arrows={[sep] - Stealth[length=3pt]}]        (0,1) -- (0,.2);
\draw [arrows={[sep] - Stealth[length=3pt]}]        (2,.2)   -- (2,1);
\draw [arrows={[sep] - Stealth[length=3pt]}]        (1,-3.5) -- (1,-2.7);
\node at (2.5,0)[scale=0.5] {\tiny $1$};
\node at (1.5,-2.5)[scale=0.5] {\tiny $2$};

\pgftransformxshift{9cm}
\draw (0,0) -- (1,-1) -- (2,0);
\draw (1,-1) -- (1,-2.5);
\node at (1,-5.5) {\tiny $\cA\to \cA\otimes\cA^\vee$};
\node at (-1,-2)[scale=0.75] {\tiny $m_R^\vee$};
\draw [arrows={[sep] - Stealth[length=3pt]}]        (0,.2)   -- (0,1);
\draw [arrows={[sep] - Stealth[length=3pt]}]        (2,1) -- (2,.2);
\draw [arrows={[sep] - Stealth[length=3pt]}]        (1,-3.5) -- (1,-2.7); 
\node at (-0.5,0)[scale=0.5] {\tiny $2$};
\node at (1.5,-2.5)[scale=0.5] {\tiny $1$};

\pgftransformxshift{9cm}
\draw (0,0) -- (1,-1) -- (2,0);
\draw (1,-1) -- (1,-2.5);
\node at (1,-5.5) {\tiny $\cA\to \cA^\vee\otimes\cA$};
\node at (-1,-2)[scale=0.75] {\tiny $\tau m_R^\vee$};
\draw [arrows={[sep] - Stealth[length=3pt]}]        (0,.2)   -- (0,1);
\draw [arrows={[sep] - Stealth[length=3pt]}]        (2,1) -- (2,.2);
\draw [arrows={[sep] - Stealth[length=3pt]}]        (1,-3.5) -- (1,-2.7); 
\node at (-0.5,0)[scale=0.5] {\tiny $1$};
\node at (1.5,-2.5)[scale=0.5] {\tiny $2$};

\pgftransformxshift{10cm}
\draw [arrows={Stealth[length=3pt] - }] (0,0) -- (1,-1);
\draw (1,-1) -- (2,0);
\draw (1,-1) -- (1,-2.5);
\draw [arrows={Stealth[length=3pt,reversed] - Stealth[length=3pt]}] (-0.1,.2)  .. controls (-1.1,1.5) .. (-2.1,.2);
\node at (1,-5.5) {\tiny $\cA\to \cA\otimes\cA$};
\node at (-2,-2)[scale=0.75] {\tiny $(Q\otimes 1)\tau m_R^\vee$};
\draw [arrows={[sep] - Stealth[length=3pt]}]        (2,1) -- (2,.2);
\draw [arrows={[sep] - Stealth[length=3pt]}]        (1,-2.7) -- (1,-3.5); 
\node at (-2.4,0)[scale=0.5] {\tiny $1$};
\node at (1.5,-2.5)[scale=0.5] {\tiny $2$};
\node at (-1.1,1.7)[scale=.65] {\tiny $Q$};

\pgftransformxshift{11cm}
\draw [arrows={Stealth[length=3pt,reversed] - }] (0,0) -- (1,-1);
\draw (1,-1) -- (2,0);
\draw (1,-1) -- (1,-2.5);
\draw [arrows={Stealth[length=3pt] - Stealth[length=3pt]}] (-0.1,.2)  .. controls (-1.1,1.5) .. (-2.1,.2);
\node at (1,-5.5) {\tiny $\cA\to \cA\otimes\cA$};
\node at (-2.8,-2)[scale=0.75] {\tiny $(1\otimes\mu)(\tau Q_0\otimes 1)$};
\draw [arrows={[sep] - Stealth[length=3pt]}]        (2,1) -- (2,.2);
\draw [arrows={[sep] - Stealth[length=3pt]}]        (1,-2.7) -- (1,-3.5); 
\node at (-2.4,0)[scale=0.5] {\tiny $1$};
\node at (1.5,-2.5)[scale=0.5] {\tiny $2$};
\node at (-1.1,1.7)[scale=.65] {\tiny $Q_0$};
\node at (1.5,-1.3)[scale=.65] {\tiny $\mu$};

\end{tikzpicture}
\end{center}
This proves that the component $\cA\to\cA\otimes\cA$ of the coproduct $\boldlambda_\cD$ is $\lambda_\cD$. \qed

\end{itemize}

}

\subsection{Rabinowitz Floer homology as a colimit in groupoids} \label{sec:RFH-colimit}

While the chain level components of the cone product and coproduct depend a priori on the choice of continuation data determining the continuation map $c$, the induced maps in homology do not depend on these choices. We now explain this invariance property. 

Recall that a set of \emph{continuation data} $\cD$ consists of a Morse function $K$ on $V$ and a homotopy $-K\to K$. This determines a continuation map $c_\cD:FC_*(-K)\to FC_*(K)$. We further choose a coherent set of Hamiltonian homotopies $K\to H_b$ and $H_{-b}\to -K$, where $H_b$, $b>0$ is an admissible Hamiltonian of slope $b$ extending $K$ and $H_{-b}=-H_b$. We define the \emph{cone Rabinowitz Floer homology group for the continuation datum $\cD$} to be
$$
S\H_{*,\cD}(\p V)=\lim\limits_{\stackrel{\longrightarrow}{b\to\infty}} 
\lim\limits_{\stackrel{\longleftarrow}{a\to-\infty}}
H_*(Cone(c_\cD:FC_{*+n}(H_a)\to FC_{*+n}(H_b))),
$$
where we denote $c_\cD:FC_*(H_a)\to FC_*(H_b)$ the continuation map determined by the concatenation of the Hamiltonian homotopies $H_a\to -K\to K\to H_b$. 
This group carries a product $\boldmu_{\cD}$ and a coproduct $\boldlambda_\cD$ as explained above. 
We omit from the notation the data of the Hamiltonian homotopies $H_a\to -K$ and $K\to H_b$ since the resulting structures do not depend on these at homology level.

Any two continuation data $\cD$, $\cD'$ can be connected by a homotopy, and such a homotopy induces a \emph{transition isomorphism} 
$$
\Phi_{\cD,\cD'}:S\H_{*,\cD}(\p V)\stackrel\simeq\longrightarrow S\H_{*,\cD'}(\p V)
$$
which intertwines the products 
$$
\Phi_{\cD,\cD'}\boldmu_{\cD} = \boldmu_{\cD'}  (\Phi_{\cD,\cD'}\otimes \Phi_{\cD,\cD'})
$$
and the coproducts 
$$
\boldsymbol{\lambda}_{\cD'} \Phi_{\cD,\cD'} = (\Phi_{\cD,\cD'}\otimes \Phi_{\cD,\cD'}) \boldsymbol{\lambda}_\cD.
$$
The isomorphism $\Phi_{\cD,\cD'}$ is independent of the choice of homotopy from $\cD$ to $\cD'$, and it satisfies the cocycle conditions $\Phi_{\cD',\cD''}\circ \Phi_{\cD,\cD'} = \Phi_{\cD,\cD''}$ and $\Phi_{\cD,\cD}=\mathrm{Id}$. In other words, the collection of objects 
$$
\{S\H_{*,\cD}(\p V)\}_\cD,
$$
together with the collection of isomorphisms 
$$
\{\Phi_{\cD,\cD'}\}_{\cD,\cD'}
$$
forms a groupoid. We define the \emph{cone Rabinowitz-Floer homology} as the colimit of this groupoid,
$$
S\H_*^{cone}(\p V)=\colim_\cD \, S\H_{*,\cD}(\p V).
$$
In very explicit terms $S\H_*^{cone}(\p V)$ can be described as the quotient of the direct sum $\oplus _\cD S\H_{*,\cD}(\p V)$ by the submodule generated by the relations $a_\cD-\Phi_{\cD,\cD'}a_\cD$ for all $\cD,\cD'$ and $a_\cD\in S\H_{*,\cD}(\p V)$. See also~\cite{stacks-project-colimit}. The colimit inherits a canonical product 
$$
\boldmu:S\H_*^{cone}(\p V)\otimes S\H_*^{cone}(\p V)\to S\H_*^{cone}(\p V)
$$
and coproduct 
$$
\boldsymbol{\lambda}:S\H_*^{cone}(\p V)\to S\H_*^{cone}(\p V)\otimes S\H_*^{cone}(\p V).
$$ 

The canonical maps $S\H_{*,\cD}(\p V)\to S\H_*^{cone}(\p V)$ fit into commutative diagrams 
$$
\xymatrix
@R=10pt 
{
S\H_{*,\cD}(\p V) \ar[drr] \ar[dd]_{\Phi_{\cD,\cD'}} & & \\
& & S\H_*^{cone}(\p V) \\
S\H_{*,\cD'}(\p V) \ar[urr] & & 
}
$$
We proved in~\cite{CO-cones} that we have a canonical isomorphism 
$$
S\H_*^{cone}(\p V)\stackrel\cong\longrightarrow S\H_*(\p V),
$$ 
and this further gives rise to the diagram 
$$
\xymatrix{
S\H_{*,\cD}(\p V) \ar[drr] \ar[dd]_{\Phi_{\cD,\cD'}} \ar@/^10pt/[drrr] & & & \\
& & S\H_*^{cone}(\p V) \ar[r]^\cong & S\H_*(\p V) \\
S\H_{*,\cD'}(\p V) \ar[urr] \ar@/_10pt/[urrr] & & &
}
$$

Now the point is that, even in cases where different choices of continuation data $\cD$ and $\cD'$ define \emph{equal} chain level continuation maps $c_\cD=c_{\cD'}$ and hence \emph{equal} homology groups $S\H_{*,\cD}(\p V)=S\H_{*,\cD'}(\p V)$, the transition isomorphism $\Phi_{\cD,\cD'}$ may be nontrivial. To see this explicitly, let $W$ be an $R$-essential Weinstein domain and $K$ a fixed $R$-essential Morse function. We proved in Lemma~\ref{lem:R-essential-continuation-maps-are-equal} that, given continuation data $\cD$, $\cD'$ with Morse function $K$, the continuation maps $c_\cD$ and $c_{\cD'}$ are equal. A choice of homotopy between $\cD$ and $\cD'$ determines a degree $1$ secondary continuation map $\vec\boldc_{\cD,\cD'}:FC_*(-K)\to FC_{*+1}(K)$, well-defined up to chain homotopy. 
We choose the homotopies $K\to H_b$ for $b>0$ such that the resulting map $FC_*(K)\to FC_*(H_b)$ is an inclusion. 
The isomorphism $\Phi_{\cD,\cD'}$ is then induced at chain level by the matrix 
$$
\begin{pmatrix}
1 & {\alex -}\vec\boldc_{\cD,\cD'} \\ 0 & 1
\end{pmatrix}
$$
with respect to the splitting $Cone(c_\cD) = FC_*(H_b)\oplus FC_{*-1}(-H_b)$.
If the map $\vec\boldc_{\cD,\cD'}$ is nontrivial in homology when restricted to $\ker c$, then the isomorphism $\Phi_{\cD,\cD'}$ is also nontrivial.

\begin{example} \label{exple:continuation} 
The previous phenomenon concretely occurs already for $W=T^*S^1$. Given a perfect Morse function on $S^1$, extend it to a Morse function $K$ on $T^*S^1=\R\times S^1$ by adding a positive quadratic function in the $\R$-variable, so that $K$ has a single minimum and a single critical point of index $1$. Choose continuation data in the form of a small nowhere vanishing $1$-form $\eta$ on $S^1$ {\reftwo (cf.~Remark~\ref{rem:cont-data})}, extended to $\R\times S^1$ as a vector field $\hat \eta$ which is constant along the $\R$-coordinate. We know that the resulting continuation map $c_{\hat \eta}$ is zero because the Euler characteristic of $S^1$ is zero. We can see this explicitly as follows. Denote $\varphi_{\hat\eta}$ the time one flow of $\hat \eta$. The map $c_{\hat \eta}:MC^*(-K)\to MC^*(K)$ counts pairs $(\gamma_1,\gamma_2)$ with $\gamma_{1,2}:[0,\infty)\to T^*S^1$, $\dot \gamma_{1,2}=\nabla K(\gamma_{1,2})$, $\gamma_{1,2}(+\infty)\in\mathrm{Crit}(K)$, and $\gamma_1(0)=\varphi_{\hat \eta}^{-1}(\gamma_2(0))$. By our choice of function $K$ we necessarily have $\gamma_{1,2}(0)\in S^1$, so that $\gamma_1(0)$ and $\gamma_2(0)$ must be equal and coincide with a zero of $\hat \eta$. Since $\hat \eta$ has no zeroes (which reflects the vanishing of the Euler characteristic of $S^1$), we conclude that $c_{\hat\eta}$ is zero. 

Consider now the continuation maps induced by $\hat\eta$ and $-\hat\eta$. While they both vanish, an interpolation between $\hat \eta$ and $-\hat\eta$ determines a chain map $\vec\boldc:MC^*(-K)\to MC^{*-1}(K)$ which acts nontrivially in homology. More precisely, if we denote $p,q$ the two critical points of $K$, which are also critical points of $-K$, an explicit computation shows that the map $\vec\boldc$ sends $p$ to $\pm q$ and $q$ to $\pm p$. The induced automorphism 
$$
\begin{pmatrix} 1 & {\alex -}\vec\boldc \\ 0 & 1 \end{pmatrix}
$$
of  $Cone(c_{\hat\eta})=Cone(c_{-\hat\eta})$ is therefore nontrivial in homology.  
\end{example}

\subsection{The long exact sequence of the pair $(V,\p V)$} \label{sec:RFH-les}

We have already used in the definition of reduced homology the following long exact sequence of the pair $(V,\p V)$ from~\cite{CO}:
$$
{\scriptsize 
\xymatrix
@C=20pt
{
S\H_*(V,\p V) \simeq S\H^{-*-2n}(V) \ar[r]^-\eps & S\H_*(V) \ar[r]^-{\iota} & 
   S\H_*(\p V) \ar[r]^-{\pi}  & 
   S\H^{1-2n-*}(V) \ar[r]^-\eps & \\
   }
}   
$$
Our next result is a refinement of this long exact sequence to take into account products and coproducts.

\begin{proposition} \label{prop:les+}
Let $V$ be a Liouville domain. There exists a commuting diagram with exact row
\begin{equation}\label{eq:les+}
\xymatrix
@C=17pt
{
   & & & (S\H^{1-2n-*}_{>0}(V),\lambda^\vee{\alex \tau}) \ar[dl]_-{i} \ar[d]^{j} & \\
   \ar[r]^-\eps & (S\H_*(V),\mu) \ar[r]^-{\iota} \ar[d]^-q & 
   (S\H_*(\p V),\boldsymbol{\mu},\boldsymbol{\lambda}) \ar[r]^-{\pi} \ar[dl]_-{p} & 
   (S\H^{1-2n-*}(V),{\alex \tau}\mu^\vee) \ar[r]^-\eps & \\
   & (S\H_*^{>0}(V),\lambda) &
}
\end{equation}
in which
\vspace{-5pt}
\begin{itemize} 
\item the maps $\iota$ and $i$ intertwine the products $\mu$,
  $\boldsymbol{\mu}$ and $\lambda^\vee{\alex \tau}$;
\item the maps $p$ and $\pi$ intertwine the coproducts $\lambda$,
  $\boldsymbol{\lambda}$ and ${\alex \tau}\mu^\vee$.
\end{itemize} 
Dualizing the diagram, replacing degree $*$ by $1-2n-*$ and applying
Poincar\'e duality $S\H_*(\p V)\cong S\H^{1-2n-*}(\p V)$ reproduces
the same diagram reflected at its center. 
\end{proposition}

\begin{proof}
We use the cone description of Rabinowitz Floer homology and its product structures given in the previous section. 

The long exact sequence of the pair is an instance of {\reftwo the} homology exact sequence of a cone, with $\iota$ induced by the inclusions $\cA\hookrightarrow Cone(c)$ and $\pi$ induced by the projections $Cone(c)\to \cM[-1]$. Thus $\iota$ intertwines the products $\mu$ and $\boldmu$, whereas $\pi$ intertwines the coproducts $\boldlambda$ and $\mu^\vee$.  

Denote $\cA^{=0}$ the 0-energy subcomplex of $\cA$, and $\cA^{>0}=\cA/\cA^{=0}$. The map $\lambda:\cA\to \cA'\otimes \cA'$ is a chain map modulo $\im c$ and, because $\im c\subset \cA^{=0}$, it induces a coproduct $\lambda$ on $S\H_*^{>0}(V)$. The map $p$ is induced in the limit by the composition of projections $Cone(c)\to \cA/\im c\to \cA^{>0}$, and it follows from the description of the coproduct on the cone that $p$ intertwines the coproducts $\boldlambda$ and $\lambda$. 

Similarly, denote $\cM^{<0}$ the negative energy subcomplex of $\cM$, which corresponds under algebraic duality to the positive energy subcomplex defining $SH^*_{>0}(V)$. The map $c^\vee$ vanishes on $\cM^{<0}$ and therefore $\lambda^\vee:\cM'\otimes \cM'\to \cM$ defines a chain map $\cM^{\prime,<0}\otimes \cM^{\prime,<0} \to \cM^{<0}$ which induces a product $\lambda^\vee$ on $S\H^*_{>0}(V)$. The map $i$ is induced by the composition of inclusions $\cM^{<0}\to \ker c^\vee\to Cone(c^\vee)$, and it follows from the description of the product on the cone that $i$ intertwines the products. {\alex There is no dependency on the continuation data for the cohomology product $\lambda^\vee\tau$, respectively for the homology coproduct $\lambda$, because we restrict them to positive action.}  

The final claim is a consequence of the proof of Poincar\'e duality for cones from~\cite{CO-cones}. 
\end{proof}

\begin{remark}
In Proposition~\ref{prop:les+} we have $pi=0$ because the map $pi$ factors as the composition
$$
  S\H^{1-2n-*}_{>0}(V) \cong S\H_*^{<0}(\p V) \to S\H_*(\p V) \to S\H_*^{\geq 0}(\p V)\to S\H_*^{>0}(V) 
$$
where the three terms in the middle are part of the action truncation exact sequence of $\p V$. However, we do not necessarily have {\reftwo $\ker p=\im i$.} 
\end{remark}

The next Corollary is straightforward. For the statement, recall that reduced homology is $\ol{S\H}_*(V)=\coker \eps$ and carries an induced product $\bar\mu$, whereas reduced cohomology is $\ol{SH}^*(V)=\ker \eps$ and carries an induced coproduct ${\bar\mu}^\vee$. 

\begin{corollary} \label{cor:les+} Let $V$ be a Liouville domain. 
There exists a commuting diagram with short exact row
\begin{equation}\label{eq:les+reduced}
\xymatrix
@C=14pt
{
   & & & (S\H^{1-2n-*}_{>0}(V),\lambda^\vee{\alex \tau}) \ar[dl]_-{i} \ar[d]^{\bar j} & \\
   0 \ar[r] & (\ol{S\H}_*(V),\bar \mu) \ar[r]^-{\iota} \ar[d]^-{\bar q} & 
   (S\H_*(\p V),\boldsymbol{\mu},\boldsymbol{\lambda}) \ar[r]^-{\pi} \ar[dl]_-{p} & 
   (\ol{S\H}^{1-2n-*}(V),{\alex \tau}{\bar \mu}^\vee) \ar[r] & 0 \\
   & (S\H_*^{>0}(V),\lambda) &
}
\end{equation}
in which
\vspace{-5pt}
\begin{itemize} 
\item the maps $\iota$ and $i$ intertwine the products $\bar\mu$,
  $\boldsymbol{\mu}$ and $\lambda^\vee{\alex \tau}$;
\item the maps $p$ and $\pi$ intertwine the coproducts $\lambda$,
  $\boldsymbol{\lambda}$ and ${\alex \tau}{\bar \mu}^\vee$.
\end{itemize} 
Dualizing the diagram, replacing degree $*$ by $1-2n-*$ and applying
Poincar\'e duality $S\H_*(\p V)\cong S\H^{1-2n-*}(\p V)$ reproduces
the same diagram reflected at its center.  \qed
\end{corollary}

\subsection{Splittings} \label{subsec:splittings} 

The results of the previous section were valid for arbitrary Liouville domains. In this section we specialize to strongly $R$-essential Weinstein domains, in which case we show that the short exact sequence from Corollary~\ref{cor:les+} is split, albeit not canonically. 

\begin{proposition} \label{prop:splittings}
Let $W$ be a strongly $R$-essential Weinstein domain of dimension $2n\ge 6$. {\alex Continuation data $\cD$ induce coproducts $\bar\lambda_\cD$ and $\bar\lambda_{\cD^{op}}$ on $\ol{S\H}_*(W)$ (see~\S\ref{sec:RFH-colimit}) and a splitting of the short exact sequence~\eqref{eq:les+reduced} }
\begin{equation}\label{eq:les+reducedsplit}
{\footnotesize 
\xymatrix
@C=12pt
{
   & & & (S\H^{1-2n-*}_{>0}(W),\lambda^\vee{\alex \tau}) \ar[dl]_-{i} \ar[d]^{\bar j} & \\
   0 \ar[r] & (\ol{S\H}_*(W),\bar \mu,\bar\lambda_{\alex \cD}) \ar[r]_-{\iota} \ar[d]^-{\bar q} & \ar@/_/[l]_-{\bar p}
   (S\H_*(\p W),\boldsymbol{\mu},\boldsymbol{\lambda}) \ar[r]_-{\pi} \ar[dl]^-{p} & \ar@/_/[l]_-{\bar i}
   (\ol{S\H}^{1-2n-*}(W),{\bar \lambda}^\vee_{\alex \cD^{op}}{\alex \tau}, {\alex \tau}{\bar \mu}^\vee) \ar[r] & 0 \\
   & (S\H_*^{>0}(W),\lambda) &
}
}
\end{equation}
via maps $\bar p,\bar i$ satisfying the following conditions:
\begin{itemize}
\item $\bar p\iota=1$, $\pi \bar i=1$, $\bar i \bar j = i$, and $\bar q \bar p = p$. 
\item $\im\bar i=\ker\bar p$.
\item $\bar i$ is a ring map with respect to the product ${\bar \lambda}^\vee_{\alex \cD^{op}}{\alex\tau}=\pi \boldmu (\bar i \otimes \bar i)$  
on $\ol{S\H}^*(W)$. 
\item $\bar p$ intertwines the coproduct on $S\H_*(\p W)$ with the coproduct $\bar \lambda_{\alex \cD}=(\bar p\otimes \bar p) \boldsymbol{\lambda} \iota$ on $\ol{S\H}_*(W)$. 
\end{itemize} 
{\alex
With respect to the splitting $S\H_*(\p W)=\ol{S\H}_*(W)\oplus \ol{S\H}^{1-2n-*}(W)$ and the algebraically dual splitting $S\H^{1-2n-*}(\p W)=\ol{S\H}^{1-2n-*}(W)\oplus \ol{S\H}_*(W)$, the Poincaré duality isomorphism is given in matrix form by
$$
I=\begin{pmatrix} 0 & 1 \\ 1 & -\vec \boldc_{\cD,\cD^{op}} \end{pmatrix},
$$
with $\vec \boldc_{\cD,\cD^{op}}$ the secondary continuation map from~\S\ref{sec:secondary_cont}.
}
\end{proposition}

Note that, if $\bar i$ is a ring map with respect to \emph{some} product on $\ol{S\H}^*(W)$, that product is necessarily given by the formula $\pi \boldmu (\bar i \otimes \bar i)$.
So this assertion is equivalent to $\im\bar i$ being a subalgebra. 
A similar observation pertains to the map $\bar p$.

\begin{proof}
We will describe in detail only the splitting $\bar p$, the discussion of $\bar i$ being analogous. 

We expand the left half of the diagram~\eqref{eq:les+reducedsplit} as follows 
$$
{\scriptsize
\xymatrix
@C=12pt
{
(S\H_*(\p W),\boldmu,\boldlambda) \ar[rr]^-{coalg.}_-{\mathrm{proj}} \ar@/^{20pt}/[rrrr]^-{p} \ar@/_{10pt}/@{-->}[ddrrrr]_-{\bar p} 
&& (S\H_*^{\ge 0}(\p W),\lambda^{\ge 0}) \ar[rr]^-{coalg.} \ar[dr]_-{P_\cD}^-{coalg.} & & (S\H_*^{>0}(W),\lambda) \\
& & & (S\H_*(W;\im c),\lambda_\cD) \ar[ur]^-{coalg.} & \\ 
(S\H_*(W),\mu) \ar[uu]^-{alg.} \ar[rrrr]^-{alg.}_-{\mathrm{proj}} &&&& (\ol{S\H}_*(W),\bar \mu)  \ar[uu]_-{\bar q}  \ar[ul]_\simeq
}
}
$$
The arrows labeled ``coalg." are coalgebra maps, and those labeled ``alg." are algebra maps. All the groups except $S\H_*(W;\im c)$ are defined for arbitrary Liouville domains. The group $S\H_*(W;\im c)$ is defined for $R$-essential Weinstein domains, and the map $\ol{S\H}_*(W)\to S\H_*(W;\im c)$ is an isomorphism under the assumption of strong $R$-essentiality. The coproduct $\lambda_\cD$ depends a priori on the choice of continuation data $\cD$, see~\S\ref{sec:dependence_on_D}. We will now describe the map $P_\cD$, and we define $\bar p$ to be the composition of the maps $S\H_*(\p W)\to S\H_*^{\ge 0}(\p W)\to S\H_*(W;\im c)\stackrel\simeq\longleftarrow \ol{S\H}_*(W)$, as in the diagram.

As part of the continuation data $\cD$ we have an $R$-essential Morse function $K$, and a homotopy from $-K$ to $K$ which determines a continuation map $c_\cD:FC_*(-K)\to FC_*(K)$. Denote $H_b$ an $R$-essential Hamiltonian of non-critical slope $b>0$ which extends $K$, and denote still by $c_\cD$ the continuation map $FC_*(-K)\to FC_*(H_b)$.
{\reftwo The cone description of Rabinowitz Floer homology in~\S\ref{sec:RFH-colimit} yields
$$
S{\refone \H}_*^{\ge 0}(\p W)= \lim\limits_{\stackrel{\longrightarrow}{b\to\infty}} H_*(Cone(c_\cD:FC_{\refone *+n}(-K)\to FC_{\refone *+n}(H_b))).
$$
(Here the nonnegative action part is picked out by using in the domain the fixed Morse function $-K$ rather than Hamiltonians $H_a$ of arbitrarily negative slope.)}
Recall also that  
$$
S{\refone \H}_*(W;\im c) = \lim\limits_{\stackrel{\longrightarrow}{b\to\infty}} H_{\refone *+n}(FC_*(H_b)/\im c_\cD). 
$$
The projection from the homotopy cokernel to the cokernel
$$
\xymatrix{
Cone(c_\cD:FC_{\refone *+n}(-K)\to FC_{\refone *+n}(H_b))\ar[r]^-{P_{\cD,b}} \ar@{=}[d] & FC_{\refone *+n}(H_b)/\im c_\cD \\
FC_{\refone *+n}(H_b)\oplus FC_{\refone *+n-1}(-K) & 
}
$$
is a chain map, and it acts by $(a,\bar x)\mapsto a \ \mathrm{mod} \ \im c_\cD$. It induces in the limit $b\to\infty$ the map 
$P_\cD:S{\refone \H}_*^{\ge 0}(\p W)\to S{\refone \H}_*(W;\im c)$. 

We now prove that $P_\cD$ is a coalgebra map. Fix a Hamiltonian $H=H_b$ as in the definition of the map $P_{\cD,b}$, which we denote for simplicity $P_\cD$. Denote 
$[{\alex \lambda_\cD}]:FC_*(H)/\im c_\cD\to FC_*(2H)/\im c_\cD \otimes FC_*(2H)/\im c_\cD$ 
the reduction of ${\alex \lambda_\cD}$ modulo the coideal pair $(\im c_\cD,\im c_\cD)$ (Lemma~\ref{lem:lambda-imc}). {\alex Let $\boldlambda_\cD$ be the coproduct on $Cone(c_\cD)$ as discussed in~\S\ref{sec:coproduct-on-cone}.} Using the definition of $P_\cD$ and the vanishing of the operation $\beta$ involved in the cone description of the coproduct for dimensions $2n\ge 6$ {\refone (proof of Lemma~\ref{lem:lambda-imc})}, we obtain for $a\in FC_*(H)$ and $\bar x\in FC_{*-1}(-K)$:
\begin{align*}
(P_\cD\otimes P_\cD)\boldlambda_{\alex \cD} & (a,\bar x) \\
& = (P_\cD\otimes P_\cD) {\alex \lambda_\cD}(a) + (P_\cD\otimes P_\cD)\beta(\bar x) \\
& = (P_\cD\otimes P_\cD) {\alex \lambda_\cD}(a) \\
& = {\alex \lambda_\cD}(a) \ \mathrm{mod} \ \im c_\cD\otimes FC_*(2H)+FC_*(2H)\otimes \im c_\cD \\
& = [{\alex \lambda_\cD}] (a \ \mathrm{mod} \ \im c_\cD) \\
& = [{\alex \lambda_\cD}] P_\cD(a,\bar x).
\end{align*}
{\alex The first equality uses the fact that the component $\cA\to\cA\otimes\cA$ of the cone coproduct $\boldlambda_\cD$ is equal to $\lambda_\cD$ (Proposition~\ref{prop:boldlambdaD-restricts-to-lambdaD}).}
By passing to the limit over $b\to\infty$ we obtain 
\begin{equation*} 
(P_\cD\otimes P_\cD)\circ \boldsymbol{\lambda}_{\alex \cD}= [{\alex \lambda_\cD}]\circ P_\cD
\end{equation*}
where $\boldsymbol{\lambda}_{\alex \cD}$ is the cone coproduct on $S{\refone \H}_*^{\ge 0}(W)$ and $[{\alex \lambda_\cD}]$ is the coproduct on $S{\refone \H}_*(W;\im c)$ determined by the continuation data $\cD$. 

The relations $\bar q \bar p = p$ and $\bar p \iota =\mathrm{Id}$ are immediate from the definition. 

We now prove the equality $\im \bar i = \ker \bar p$. 
We start from the relations $\bar p[(a,x)]=[a \ \mathrm{mod}\ \im c]$ and $\bar i [x]=[(0,x)]$. Then clearly $\bar p\bar i=0$, so that  $\im \bar i \subset \ker \bar p$. To prove equality, consider an element $[(a,x)]\in \ker \bar p$. This is represented by a cycle $(a,x)$ such that $a$ is a boundary mod $\im c$, or equivalently $(a,0)$ is a boundary. Thus $(a,x)$ is homologous to $(0,x)$ and $(0,x)$ is therefore a cycle. Hence $x\in \ker c$ and $[(a,x)]\in\im\bar i$.  

{\alex The last statement regarding the action of the Poincaré duality isomorphism is a direct consequence of Proposition~\ref{prop:PD-cone}.}
\end{proof}

We discuss the non-canonicity of the splittings in the next subsection. 

\subsection{Dependence of the splittings on choices} \label{sec:splittings_are_not_canonical}

We continue with the strongly $R$-essential Weinstein domain $W$ of dimension $2n\ge 6$ from the previous subsection. The map  
$$
P_\cD:S{\refone \H}_*^{\ge 0}(\p W)\to S{\refone \H}_*(W;\im c)
$$
depends on the choice of continuation data $\cD$ through the intermediate secondary continuation maps. To see this, fix an $R$-essential Morse function $K$, denote $\cA_*=FC_{*+n}(K)$ and $\cM_*=FC_{*+n}(-K)$, choose continuation data $\cD$, $\cD'$ for $K$ and let $\vec\boldc_{\cD,\cD'}:FC_{\refone *+n}(-K)\to FC_{\refone *+n+1}(K)$ be the chain map determined by a homotopy between $\cD$ and $\cD'$. Although $c_\cD=c_{\cD'}$, the following diagram in which the diagonal arrows are projections \emph{does not} commute in general
$$
\xymatrix{
Cone(c_{\cD})=\cA_*\oplus \cM_{*-1} \ar[drr] \ar[dd]_{\begin{pmatrix} 1 & {\alex -} \vec\boldc_{\cD,\cD'}  \\ 0 & 1 \end{pmatrix}} & & \\
& \!\!\!\!\! \!\!\!\!\! \!\!\!\!\! \!\!\!\!\! \!\!\!\!\! \!\!\!\!\! \circlearrowleft \!\!\!\!\! \!\! \raisebox{-2pt}{\mbox{\LARGE \sf X}} & \cA_*/\im c_\cD=\im c_{\cD'} \\
Cone(c_{\cD'}) = \cA_*\oplus \cM_{*-1} \ar[urr]
}
$$
Equivalently, the two maps 
$$
\xymatrix{
S{\refone \H}_*^{\ge 0,cone}(\p W) \ar@<-.5ex>[r]_-{P_{\cD'}} \ar@<.5ex>[r]^-{P_\cD} & S{\refone \H}_*(W;\im c),
}
$$ 
obtained by inverting the diagonal isomorphisms in the diagram below, are in general different. 
$$
\xymatrix
{
S{\refone \H}_{*,\cD}^{\ge 0}(\p W) \ar[drr]^\simeq  \ar[dd]_{\Phi_{\cD,\cD'}} \ar@/^10pt/[drrrr]& & & \\
& \!\!\!\!\! \!\!\!\!\! \!\!\!\!\! \!\!\!\!\! \!\!\!\!\! \!\!\!\!\! \circlearrowleft  & S{\refone \H}_*^{\ge 0,cone} \ar@<-1ex>[rr]_-{P_{\cD'}}^-\neq \ar@<1ex>[rr]^-{P_\cD} & &  S{\refone \H}_*(W;\im c) \\
S{\refone \H}_{*,\cD'}^{\ge 0}(\p W) \ar[urr]_\simeq \ar@/_10pt/[urrrr] & & &  
}
$$ 
The map $P_\cD$, and hence the splitting from Proposition~\ref{prop:splittings}, fails to be canonical precisely to the extent to which $\lambda_\cD$ also fails to be canonical. This is coherent with the fact that each map $P_\cD$ is a coalgebra map for the corresponding coproduct $\lambda_\cD$ on $S{\refone \H}_*(W;\im c)$. The previous discussion shows that the dependence on choices is governed by the secondary continuation maps, and in view of Corollary~\ref{cor:independence} and its proof we obtain

\begin{proposition}\label{prop:splittings-can}
Let $W$ be a strongly $R$-essential Weinstein domain of dimension $2n\ge 6$. The splitting from Proposition~\ref{prop:splittings} is canonical whenever $H^{n-1}(W)=0$. \qed
\end{proposition}

On the other hand, the splitting and the coproduct can be seen to be non-canonical in situations where $H^{n-1}(W)$ does not vanish. This already happens for $W=T^*S^1$, cf. Example~\ref{exple:continuation} and the explicit computations of splittings from~\cite{CHO-PD}. Although $T^*S^1$ does not satisfy the dimension condition $2n\ge 6$, similar computations can be carried out in higher dimensional tori.

\subsection{The unital infinitesimal relation from associativity}  \label{sec:unital-from-ass}

We show in this subsection how the unital infinitesimal relation in reduced symplectic homology is implied by the associativity of the product on the cone. We use an $A_2$-structure determined by an $A_2^+$-structure as in~\S\ref{sec:RFH-cone}, with product $\mu$ and coproduct $\lambda$. We work with a strongly $R$-essential Weinstein domain $W$ of dimension $2n\ge 6$. 

{\reftwo
The following theorem summarizes Propositions~\ref{prop:splittings} and~\ref{prop:splittings-can} in a simplified form.

\begin{theorem} \label{thm:splittings}
Let $W$ be a strongly $R$-essential Weinstein domain of dimension $2n\ge 6$. 
Continuation data $\cD$ induce coproducts $\bar\lambda_\cD$ and $\bar\lambda_{\cD^{op}}$ on $\ol{S\H}_*(W)$ (see~\S\ref{sec:RFH-colimit}) and a splitting of the short exact sequence~\eqref{eq:les+reduced}
\begin{equation}\label{eq:splitting-RFH}
   S\H_*(\p W) = \ol{S\H}_*(W)\oplus \ol{S\H}^{1-2n-*}(W),
\end{equation}
such that the product $\boldmu$ on $S\H_*(\p  W)$ restricts to the product $\bar\mu$ on the subring $\ol{S\H}_*(W)$, and to the cohomology product $\bar\lambda^\vee_{\cD^{op}}$ on the subring (not containing the unit) $\ol{S\H}^{1-2n-*}(W)$. 

With respect to the splitting~\eqref{eq:splitting-RFH} and the algebraically dual splitting on $S\H^{1-2n-*}(\p W)=\ol{S\H}^{1-2n-*}(W)\oplus \ol{S\H}_*(W)$, the Poincaré duality isomorphism is given in matrix form by
$\tiny \begin{pmatrix} 0 & 1 \\ 1 & -\vec \boldc_{\cD,\cD^{op}} \end{pmatrix}$, with $\vec \boldc_{\cD,\cD^{op}}$ the secondary continuation map (see~\S\ref{sec:secondary_cont}).

If $H^{n-1}(W)=0$, then the splitting and the coproducts are canonical, i.e., they do not depend on the choice of continuation data $\cD$, and the Poincaré duality isomorphism simply exchanges the factors in the splitting. \qed
\end{theorem}
}
More generally, let us denote the components of the product $\boldmu$ with respect to the splitting~\eqref{eq:splitting-RFH} by 
$$
   m^{+-}_+:\ol{S\H}_*(W)\otimes \ol{S\H}^{1-2n-*}(W)\to \ol{S\H}_*(W)
$$ 
etc, where the upper indices denote the inputs, the lower index the
output, $+$ corresponds to $\ol{S\H}_*(W)$, and $-$ to $\ol{S\H}^{1-2n-*}(W)$. Each of these maps has degree $0$. 

{\alex Given continuation data $\cD$,} the following result expresses the product on
$S\H_*(\p W)$ in terms of 
{\alex $\bar\mu$, $\bar\lambda =\bar\lambda_\cD$}, 
and the continuation bivector $\boldc=\boldc_{\cD,\cD^{op}}$. In the statement we use the following notation: given $\bar f\in \ol{S\H}^{1-2n-*}(W)$, we denote $f$ the same element with degree shifted down by $1$, i.e. $|f|=|\bar f|-1$. 
{\alex Also, to simplify the notation we write $\mu$ instead of $\bar\mu$, and $\lambda$ instead of $\bar\lambda$.}

\begin{theorem} \label{thm:m-components-RFH}
Let $W$ be a strongly $R$-essential Weinstein domain of dimension $2n\ge 6$. The components of the product $\boldmu$ on $S\H_*(\p W)$ with respect to the splitting~\eqref{eq:splitting-RFH} are, for $a\in\ol{S\H}_*(W)$ and $\bar f, \bar g\in \ol{S\H}^{1-2n-*}(W)$, given by
\begin{enumerate}
\item $m^{++}_+=\mu$;
\item $\la m^{--}_- (\bar f \otimes \bar g),a\ra \\ =  
 (-1)^{|g|+|g||f|+1}\langle g\otimes f,\lambda(a) + (-1)^{|a|}(\mu\otimes 1)(a\otimes \boldc) + (1\otimes \mu)(\boldc \otimes a)\rangle$;
\item $m^{++}_-=0$ and $m^{--}_+=0$;
\item $m^{-+}_+(\bar f,a) = (-1)^{|f|+1}\la f\otimes 1,\lambda(a)+(1\otimes\mu)(\boldc\otimes a)\ra$, and\\ 
$m^{+-}_+(b,\bar f) = (-1)^{|b|}\la \lambda(b)+(-1)^{|b|}(\mu\otimes 1)(b\otimes \boldc),1\otimes f \ra$;
\item $\la m^{-+}_-(\bar f,a),b\ra = \la f,\mu(a,b)\ra$, and \\
$\la a,m^{+-}_-(b,\bar f)\ra = (-1)^{|b|}\la\mu(a,b),f\ra$. 
\end{enumerate}
\end{theorem}

\begin{proof}
These computations rely on results from~\cite{CO-cones}, and we first recall some relevant notions. 
The splitting of $S\H_*(\p W)$ arises from its description in terms of the cone of the continuation map $c=c_\cD:\cA^\vee[2n]=FC_{*+n}(-H)\to \cA=FC_{*+n}(H)$, cf.~\S\ref{sec:RFH-cone} and~\S\ref{subsec:splittings}. The product is determined by operations $\mu$, $m_L$, $m_R$, $\sigma$, $\tau_R$, $\tau_L$ as detailed in~\S\ref{sec:RFH-cone} (under our assumptions the operation $\beta$ vanishes). These operations are determined by the $A_2^+$-structure, consisting of the product $\mu$, the coproduct $\lambda=\lambda_\cD$, and the secondary continuation bivector $\boldc=\boldc_{\cD,\cD^{op}}$, according to the formulas in~\cite[\S7.2]{CO-cones}. More precisely 
\begin{align*}
\langle a,m_L(b,f)\rangle & = \langle \mu(a,b),f\rangle, \\
\langle m_R(f,a),b\rangle & = \langle f,\mu(a,b)\rangle,\\
\langle \sigma(f,g),a\rangle & = (-1)^{(|f|+1)(|g|+1)}\langle g\otimes f,\lambda_{\cD^{op}}(a)\rangle \\
 & = (-1)^{(|f|+1)(|g|+1)} \\
 & \qquad \quad \langle g\otimes f,\bigl(\lambda + (\mu\otimes 1)(1\otimes \boldc) + (1\otimes \mu)(\boldc\otimes 1)\bigr) (a)\rangle, \\
\tau_R(b,f) & = \langle \lambda_{\cD,\cD^{op}}(b),1\otimes f\rangle \\
		  & = \langle \bigl(\lambda+ (\mu\otimes 1)(1\otimes \boldc)\bigr) (b),1\otimes f\rangle \\
\tau_L(f,a) & = (-1)^{|f|+1}\langle f\otimes 1,\lambda_{\cD^{op},\cD}(a)\rangle \\
		& = (-1)^{|f|+1}\langle f\otimes 1,\bigl(\lambda + (1\otimes\mu)(\boldc\otimes 1)\bigr)(a)\rangle.
\end{align*}
Additional signs arise from the fact that one of the factors of the cone is shifted. More precisely, following~\cite[\S2.3]{CO-cones}  we define shifted operations $\um_L$, $\um_R$, $\utau_R$, $\utau_L$, $\usigma$ by the relations 
$$
m_L=\um_L[0,1;1],\qquad m_R=\um_R[1,0;1],
$$
$$
\tau_R=\utau_R[0,1;0],\qquad \tau_L=\utau_L[1,0;0],
$$
$$
\sigma=\usigma[1,1;1]. 
$$
Here $[i,j;k]$ indicates a shift by $i$ on the first component of the source, a shift by $j$ on the second component of the source, and a shift by $k$ in the target, see~\cite[Appendix~A]{CO-cones}. We then have $m^{++}_+=\mu$, $m^{+-}_-=\um_L$, $m^{-+}_-=\um_R$, $m^{--}_-=\usigma$, $m^{+-}_+=\utau_R$, $m^{-+}_+=\utau_L$. Denoting $\bar f$, $\bar g$ the shifts of elements $f$, $g$, we find 
\begin{align*}
\langle a,m^{+-}_-(b\otimes\bar f)\rangle & = \langle a,\um_L(b,\bar f)\rangle 
 = (-1)^{|b|}\langle a,m_L(b,f)\rangle, \\
\langle m^{-+}_-(\bar f\otimes a),b\rangle & = \langle \um_R(\bar f\otimes a),b\rangle  = \langle m_R(f,a),b\rangle, \\
\langle m^{--}_-(\bar f\otimes \bar g),a\rangle & = \langle \usigma(\bar f\otimes \bar g),a \rangle 
 = (-1)^{|f|}\langle \sigma(f\otimes g),a\rangle, \\
m^{+-}_+(b\otimes \bar f) & = \utau_R(b\otimes \bar f) 
 = (-1)^{|b|} \tau_R(b,f), \\
m^{-+}_+(\bar f\otimes a) & = \utau_L(\bar f\otimes a) 
 = \tau_L(f,a).
\end{align*}
Combining these formulas with the ones for $m_L$, $m_R$, $\sigma$, $\tau_L$, $\tau_R$ yields the conclusion. 
\end{proof}

\begin{remark}
The product $\boldmu$ on $S\H_*(\p W)$ is graded commutative, i.e. $\boldmu=\boldmu\circ \tau$. From this perspective, the two formulas in Theorem~\ref{thm:m-components-RFH}(4) are equivalent to each other as a consequence of $m^{-+}_+=m^{+-}_+\circ \tau$, and so are the two formulas in Theorem~\ref{thm:m-components-RFH}(5) as a consequence of $m^{-+}_-=m^{+-}_-\circ \tau$. 
\end{remark}

\begin{proposition}\label{prop:unit-inf-from-ass}
The unital infinitesimal relation 
{\alex
$$
\bar\lambda \bar\mu= (1\otimes\bar\mu)(\bar\lambda \otimes 1) + (\bar\mu\otimes 1)(1\otimes \bar\lambda ) - (\bar\mu\otimes \bar\mu)(1\otimes \bar\lambda \bar\eta\otimes 1)
$$ }
on $\ol{S\H}_*(W)$ is implied by the associativity of $\boldmu$ on $S\H_*(\p W)$. 
\end{proposition} 
 
\begin{proof}
We examine the $^{++-}_+$ component of the associativity relation $\boldmu(1\otimes \boldmu-\boldmu\otimes 1)=0$, which reads
$$
   0 = m^{++}_+(1\otimes m^{+-}_+) + m^{+-}_+(1\otimes m^{+-}_-) 
   - m^{+-}_+(m^{++}_+\otimes 1) - m^{--}_+(m^{++}_-\otimes 1),
$$
where the first and the third summands correspond to splitting along $\ol{S\H}_*(W)$, and the second and the fourth summands to splitting along $\ol{S\H}^{1-2n-*}(W)$.
By Theorem~\ref{thm:m-components-RFH}(3) the last summand vanishes. 
Now we evaluate each term
on inputs $a,b\in \ol{S\H}_*(W)$ and $\bar f\in\ol{S\H}^{1-2n-*}(W)$.
{\alex Again, in order to simplify the notation we write $\mu$ instead of $\bar\mu$, and $\lambda$ instead of $\bar\lambda$.}

Using Theorem~\ref{thm:m-components-RFH}(1) and (4) the first term becomes
\begin{align*}
    m^{++}_+(1\otimes m^{+-}_+)&(a\otimes b\otimes \bar f)  
   = \mu(a\otimes m^{+-}_+(b\otimes\bar f)) \\
    & = \mu(a\otimes\la\lambda(b)(-1)^{|b|}+(\mu\otimes 1)(b\otimes\boldc),1\otimes f\ra) \\
    & = (-1)^{|b|+|a|} \la (\mu\otimes 1)(1\otimes\lambda)(a\otimes b),1\otimes f\ra \\
    & \qquad + \la \mu(1\otimes\mu)\otimes 1(a\otimes b\otimes \boldc),1\otimes f\ra.
\end{align*}

Using Theorem~\ref{thm:m-components-RFH}(4) and (5) the second term becomes
\begin{align*}
   m^{+-}_+(1\otimes m^{+-}_-)&(a\otimes b\otimes \bar f) 
    = m^{+-}_+(a\otimes m^{+-}_-(b\otimes\bar f)) \\
   & = \la \lambda(a)(-1)^{|a|} + (\mu\otimes 1)(a\otimes\boldc), 1\otimes m^{+-}_-(b\otimes\bar f)\ra \\
   & = (-1)^{|a|+|b|}\la (1\otimes \mu)(\lambda\otimes 1)(a\otimes b),1\otimes f\ra \\
   & \qquad + (1\otimes\mu)(\mu\otimes 1)(a\otimes\boldc\otimes b)(-1)^{|b|}.
\end{align*}

Using Theorem~\ref{thm:m-components-RFH}(1) and (4) the third term becomes
\begin{align*}
   m^{+-}_+(m^{++}_+\otimes 1)(a\otimes b\otimes \bar f) & =
   m^{+-}_+(\mu(a\otimes b)\otimes\bar f) \\
   &= \la \lambda\mu(a\otimes b)(-1)^{|a|+|b|},1\otimes f\ra \\
   & \qquad  + \la (\mu\otimes 1)(\mu(a\otimes b)\otimes \boldc),1\otimes f\ra.
\end{align*}

Associativity of $\mu$ implies that $\mu(1\otimes\mu)\otimes 1(a\otimes b\otimes \boldc)$ equals $(\mu\otimes 1)(\mu(a\otimes b)\otimes \boldc) = \mu(\mu\otimes 1)\otimes 1(a\otimes b\otimes \boldc)$. Summing up the three terms we find precisely the unital infinitesimal relation applied to $a\otimes b$ and inserted into $1\otimes f$. 
\end{proof}

{\reftwo
\begin{remark}
Similarly to the proof of Proposition~\ref{prop:unit-inf-from-ass}, the other relations of the unital infinitesimal anti-symmetric bialgebra structure on $\ol{S\H}_*(W)$ can be recovered from other components of the associativity relation for $\boldmu$: associativity of $\mu$ from the $^{+++}_+$ component (as well as $^{++-}_-$, $^{+-+}_-$ and $^{-++}_-$), coassociativity of $\lambda$ from the $^{---}_--$ component (as well as $^{+--}_+$, $^{-+-}_+$ and $^{--+}_+$), the unital infinitesimal relation from the $^{++-}_+$ component (as well as $^{-++}_+$, $^{+--}_-$ and $^{--+}_-$), and unital anti-symmetry from the $^{+-+}_+$ component (as well as $^{-+-}_-$). This provides an alternative proof of Theorem~\ref{thm:uias_reduced} building on the results in~\cite{CO-cones}.
\end{remark}
}

\section{The Lagrangian case}\label{sec:open_strings}

We consider in this section Maslov zero exact Lagrangian submanifolds $L$ with Legendrian boundary in a Liouville domain $W$. We denote $SH_*(L)$ Lagrangian symplectic homology, or wrapped Floer homology, graded by the Conley-Zehnder index of Hamiltonian chords. 

The discussion of reduced symplectic homology in the open string case can be developed similarly to the closed string case, with only minor modifications. The starting point is the observation that $SH_*(L)$ carries a unital associative product $\boldmu$ of degree $-n$ and $SH_*^{>0}(L)$ carries a coassociative coproduct $\boldlambda$ of degree $1$. Accordingly, the shifted group $S\H_*(L)=SH_{*+n}(L)$ carries a unital associative product $\boldmu$ of degree $0$ and $S\H_*^{>0}(L)=SH_{*+n}^{>0}(L)$ carries a coassociative coproduct $\boldlambda$ of degree $1-n$. The key difference compared to the case of closed strings is that $\boldmu$ is in general not commutative, and $\boldlambda$ is in general not cocommutative. Another difference is that the degree of the coproduct $\boldlambda$ on $S\H_*^{>0}(L)$ may be even (when $n$ is odd). We are seeking a common domain of definition for the product and for the coproduct. 

Just like in the closed case we have a long exact sequence~\cite[\S8.3]{CO}  
\begin{equation} \label{eq:les-Lagrangian}
\longrightarrow SH_*(L,\p L)\cong SH^{\alex n-*}(L)\stackrel{\eps_*}\longrightarrow SH_*(L)\stackrel {\iota_*}\longrightarrow SH_*(\p L) \longrightarrow
\end{equation}
where the map $\eps_*$ factors through the zero-energy sector as 
$$
\xymatrix{
SH_*(L,\p L) \ar[r]^{\eps_*} \ar[d]&  SH_*(L) \\
SH_*^{=0}(L,\p L) = H^{n-*}(L,\p L)\ar[r]&  H^{n-*}(L) \simeq SH_*^{=0}(L) \ar[u] 
}
$$
and the middle map is identified with the canonical map in singular cohomology. Since $\iota_*$ is a ring map, $\ker\iota_*=\im\eps_*\subset SH_*(L)$ is an ideal and therefore $\boldmu$ descends to $\coker \eps_*$. This motivates the following definition: 

\begin{definition}
\emph{Reduced Lagrangian symplectic homology, or reduced wrapped Floer homology}, is defined as 
$$
\ol{SH}_*(L)=\coker \eps_*.
$$
It carries an induced unital product, denoted $\boldmu$, with unit denoted $\boldeta$.
\end{definition}

In analogy with the definition of strong $R$-essential Weinstein domains, we single out the following class of Lagrangians. Following our standing convention, all homology and cohomology groups in the sequel are understood to be defined with coefficients in $R$.

\begin{definition} Let $L$ be a Lagrangian submanifold of dimension $n$. We say that $L$ is \emph{strongly $R$-essential} if the following conditions hold: 
\begin{enumerate}
\item $L$ admits a defining Morse function whose critical points have degree $\le \lfloor \frac{n}2\rfloor$ and, if $n$ is even, such that the number of critical points of index $n/2$ equals the rank of $H_{n/2}(L)$. We call such a Morse function \emph{$R$-essential}; 
\item if $n$ is even, the map $H^{n/2}(L)\to SH_{n/2}(L)$ is injective.  
\end{enumerate}  
\end{definition} 

\begin{example} \label{example:Lag-strongly-R-ess}
(i) Any Lagrangian disc is strongly $R$-essential for all $R$. In particular Lefschetz thimbles of Lefschetz fibrations, and Lagrangian cocore discs in Weinstein manifolds, are strongly $R$-essential. 

(ii) If $n$ is odd, then any Lagrangian which admits a Morse function which is increasing towards the boundary and whose critical points have degree $\le \lfloor \frac{n}2\rfloor$ is strongly $R$-essential. 
\end{example}

The first condition ensures that there is a well-defined \emph{Lagrangian symplectic homology group relative to the continuation map}, denoted $SH_*(L;\im c)$, and the second condition ensures that the canonical map $\ol{SH}_*(L)\to SH_*(L;\im c)$ is an isomorphism. Note that the case $n$ odd is significantly less constrained than the case $n$ even. 

Given an $R$-essential Morse function $K:L\to \R$ and a choice of gradient-like vector field $\xi$, by \emph{continuation data} $\cD$ we mean a homotopy of Morse functions and pseudo-gradient vector fields from $(-K,-\xi)$ to $(K,\xi)$. Given such data for a  strongly $R$-essential Lagrangian $L$ we obtain an extension of the coproduct $\boldlambda:SH_*^{>0}(L)\to SH_*^{>0}(L)^{\otimes 2}$ to a coproduct  
$$
\boldlambda_\cD:\ol{SH}_*(L)\to \ol{SH}_*(L)^{\otimes 2}. 
$$

In the proof of coassociativity for $\boldlambda_\cD$ we see a slight difference with the closed string case. 

\begin{proposition}
Let $L$ be a strongly $R$-essential Lagrangian of dimension $n\ge 2$. If $n=2$ we assume further that $L$ is diffeomorphic to a disc, and if $n=4$ we assume further that $L$ has no homology in degree $2$. Then $\boldlambda_\cD$ is coassociative. 
\end{proposition}

\begin{proof}
The proof is identical to the one of Proposition~\ref{prop:coass_reduced}, except for the last argument involving indices. Given an $R$-essential Morse function $K:L\to \R$, we claim that the Lagrangian counterpart of the operation $\boldB_\cD$ from equation~\eqref{eq:coass_first_line} vanishes.  Denoting again this operation $\boldB_\cD$, the absence of inputs implies that all the outputs have energy close to zero and the map can be equivalently phrased in terms of a count of gradient $\mbox{\sf Y}$-graphs on $L$ with three outputs $a,b,x\in MC^*(K)$ in a 2-parametric family. The condition for 0-dimensional moduli spaces is now $\ind_K a + \ind_K b + \ind_K x -2n+2=0$, or $\ind_K a + \ind_K b + \ind_K x=2n-2$. By $R$-essentiality the left hand side of this equality is $\le 3\lfloor \frac{n}2\rfloor$, which is strictly less than $2n-2$ for $n=3$ and $n\ge 5$. For $n=2$ the equality can hold only if one of the outputs has index $0$ and the two others have index 1, and this case is excluded by the assumption that $L$ is a disc. For $n=4$ the equality can hold only if all three outputs have index 2, which is again excluded by $R$-essentiality and the vanishing of $H_2(L)$. This shows that, under our assumptions, there are no such rigid configurations and $\boldB_\cD$ vanishes.

As a consequence, the Lagrangian chain level relation~\eqref{eq:coass_first_line} reduces to $[\p,\Theta]=(\boldlambda_\cD\otimes  1)\boldlambda_\cD + (1\otimes \boldlambda_\cD)\boldlambda_\cD$. This proves coassociativity.
\end{proof}

The resulting algebraic structure on $\ol{S\H}_*(L)$ is similar to that in the closed string case (Theorem~\ref{thm:uias_reduced}). 

\begin{theorem} \label{thm:uias_reduced_Lag}
Let $L$ be a strongly $R$-essential Lagrangian of dimension $n\ge 2$. If $n=2$ we assume further that $L$ is diffeomorphic to a disc, and if $n=4$ we assume further that $L$ has no homology in degree $2$. Shifted reduced Lagrangian symplectic homology 
$$
(\ol{S\H}_*(L),\boldmu,\boldlambda_\cD,\boldeta)
$$ 
is a unital infinitesimal anti-symmetric bialgebra. \qed
\end{theorem}

\begin{remark}
As seen in~\S\ref{sec:odd-spheres} below, this structure is in general not involutive. 
\end{remark}

{\reftwo
All the preceding discussion carries over, with identical proofs, to the Lagrangian setting. Thus there exists a dual construction of reduced Lagrangian symplectic cohomology 
$\ol{SH}^{n-*}(L)=\ker\eps_*$, and with this definition the long exact sequence~\eqref{eq:les-Lagrangian} becomes short exact 
$$
0 \longrightarrow \ol{SH}_*(L)\stackrel {\iota_*}\longrightarrow SH_*(\p L) \longrightarrow \ol{SH}^{n-*+1}(L)\longrightarrow 0.
$$
See also Remark~\ref{rmk:dual}. We now define degree shifted groups 
$$
\ol{S\H}_*(L)=\ol{SH}_{*+n}(L),\ S\H_*(\p L)=SH_{*+n}(\p L),\ \ol{S\H}^*(L)=\ol{SH}^{*+n}(L),
$$ 
and we find the short exact sequence
\begin{equation}\label{eq:les+reduced-Lag}
0 \longrightarrow \ol{S\H}_*(L)\stackrel {\iota_*}\longrightarrow S\H_*(\p L) \longrightarrow \ol{S\H}^{1-n-*}(L)\longrightarrow 0.
\end{equation}
Statements analogous to Proposition~\ref{prop:lambdaDprimeDC} and Corollary~\ref{cor:lambdaDprimeD} hold in the Lagrangian setting, and a discussion of splittings similar to that of~\S\ref{sec:splittings} leads to the following Lagrangian analogue of Theorem~\ref{thm:splittings}. 

\begin{theorem} \label{thm:splittings-Lag}
Let $L$ be a strongly $R$-essential Lagrangian of dimension $n\ge 2$. If $n=2$ we assume further that $L$ is diffeomorphic to a disc, and if $n=4$ we assume further that $L$ has no homology in degree $2$. 
Continuation data $\cD$ induce coproducts $\bar\lambda_\cD$ and $\bar\lambda_{\cD^{op}}$ on $\ol{S\H}_*(L)$ (see~\S\ref{sec:RFH-colimit}) and a splitting of the short exact sequence~\eqref{eq:les+reduced-Lag}
\begin{equation*}\label{eq:splitting-RFH-Lag}
   S\H_*(\p L) = \ol{S\H}_*(L)\oplus \ol{S\H}^{1-n-*}(L),
\end{equation*}
such that the product $\boldmu$ on $S\H_*(\p L)$ restricts to the product $\bar\mu$ on the subring $\ol{S\H}_*(L)$, and to the cohomology product $\bar\lambda^\vee_{\cD^{op}}$ on the subring (not containing the unit) $\ol{S\H}^{1-n-*}(L)$. 

With respect to the splitting~\eqref{eq:splitting-RFH-Lag} and the algebraically dual splitting on $S\H^{1-n-*}(\p L)=\ol{S\H}^{1-n-*}(L)\oplus \ol{S\H}_*(L)$, the Poincaré duality isomorphism is given in matrix form by
$\tiny \begin{pmatrix} 0 & 1 \\ 1 & -\vec \boldc_{\cD,\cD^{op}} \end{pmatrix}$, with $\vec \boldc_{\cD,\cD^{op}}$ the secondary continuation map (see~\S\ref{sec:secondary_cont}).

If $H^{n-1}(L)=0$, then the splitting and the coproducts are canonical, i.e., they do not depend on the choice of continuation data $\cD$, and the Poincaré duality isomorphism simply exchanges the factors in the splitting. 
\qed
\end{theorem}

}

\section{Application to string topology}\label{sec:examples_reduced}

\subsection{Unital anti-symmetric infinitesimal bialgebra structure} 

Let $M$ be a closed $n$-dimensional $R$-orientable manifold with a Riemannian metric. We denote $D^*M$ the unit disc cotangent bundle and $D^*_qM$ its fiber at a basepoint $q\in M$. We also denote $\Lambda=\Lambda M$ the space of free loops and $\Omega=\Omega_qM$ its fiber at $q$ for the evaluation map, i.e., the space of loops based at $q$. We denote $\Lambda_0\subset \Lambda$ the subspace of constant loops, canonically identified with $M$. Recall the degree shifted loop homology $\H_*\Lambda=H_{*+n}\Lambda$ and its reduced version $\ol\H_*\Lambda=\H_*\Lambda/\chi(M)[q]$, where $\chi(M)$ is the Euler characteristic of $M$ and $[q]$ is the point class.

Our structural results for reduced symplectic homology of $R$-essential Weinstein domains have applications to string topology in view of the Viterbo isomorphisms~\cite{Viterbo-cotangent,AS,AS-corrigendum,AS2,SW,Ritter,Cieliebak-Latschev,Kragh,Abouzaid-cotangent,CHO-MorseFloerGH}
$$
S\H_*(D^*M)\simeq \H_*\Lambda,\qquad S\H_{*}(D^*_qM)\simeq H_*\Omega, 
$$
with their variants 
$$
S\H_*^{>0}(D^*M)\simeq \H_*(\Lambda,\Lambda_0),\qquad S\H_*^{>0}(D^*_qM)\simeq H_*(\Omega,q).
$$
These isomorphisms hold when, on the symplectic homology side, one uses the local system of coefficients on $\Lambda$, resp. $\Omega$, obtained by transgressing the second Stiefel-Whitney class~\cite{Kragh,Abouzaid-cotangent,AS-corrigendum}. The first isomorphism intertwines the unital BV-algebra structures~\cite{AS,AS-corrigendum,AS2,Ritter,Abouzaid-cotangent}, where on the loop homology side one uses the
loop product~\cite{CS}. The second isomorphism intertwines the unital algebra structures~\cite{AS-corrigendum,AS2,Ritter,Abouzaid-cotangent}, where on the based loop homology side one uses the Pontryagin product induced by the composition of loops. The third and fourth isomorphisms intertwine the coalgebra structures~\cite{CHO-MorseFloerGH}, where on the right hand side one uses the
loop coproducts~\cite{Sullivan-open-closed,Goresky-Hingston}. 

We have seen that disc cotangent bundles $D^*M$ of $R$-orientable manifolds are strongly $R$-essential Weinstein domains (Example~\ref{ex:strongly_R_essential}), and the fibers $D^*_qM$ are strongly $R$-essential exact Lagrangian submanifolds (Example~\ref{example:Lag-strongly-R-ess}). In this case we have corresponding Viterbo isomorphisms in reduced symplectic homology 
$$
\ol{S\H}_*(D^*M)\simeq \ol{\H}_*\Lambda,\qquad \ol{S\H}_*(D^*_qM)\simeq H_*\Omega.
$$
Indeed, it was shown in~\cite{Cieliebak-Frauenfelder-Oancea} that the map $\eps$ in the long exact sequence~\eqref{eq:les-intro} lives only in degree zero, where it is given by multiplication with the Euler characteristic $\chi(M)$ of $M$.
The Viterbo isomorphisms for reduced symplectic homologies intertwine the bialgebra structures~\cite{CHO-MorseFloerGH}, and in combination with the results proved in this paper this implies the following structural results for reduced loop homology and for based loop homology. 

\begin{theorem} \label{thm:main3} 
Assume $\dim M\ge 3$. The loop product {\alex $\mu$} on $\H_*\Lambda$ descends to {\alex a product $\bar\mu$ on} $\ol\H_*\Lambda$, and the loop coproduct {\alex $\lambda$} on $\H_*(\Lambda,\Lambda_0)$ extends to {\alex a coproduct $\bar\lambda$} on $\ol\H_*\Lambda$ (canonically if we have $H_1M=0$). Each such extension 
{\alex $\bar\lambda$} 
defines together with the loop product 
{\alex $\bar\mu$} 
the structure of a commutative cocommutative unital infinitesimal anti-symmetric
bialgebra on $\ol \H_*\Lambda$. In particular, the following relation holds 
{\alex 
$$
\bar\lambda\bar\mu = (\bar\mu\otimes 1)(1\otimes\bar\lambda) + (1\otimes\bar\mu)(\bar\lambda\otimes 1) - (\bar\mu\otimes\bar\mu)(1\otimes\bar\lambda\bar\eta \otimes 1),
$$ 
}
where $1$ denotes the identity map and {\alex $\bar\eta$} the unit for the product {\alex $\bar\mu$}. 
\end{theorem}

\begin{proof} This follows from Theorem~\ref{thm:uias_reduced} combined with the Viterbo isomorphisms.
\end{proof}

\begin{theorem} \label{thm:main-based} 
Assume $\dim M\ge 2$. The loop coproduct on $H_*(\Omega,q)$ extends canonically to a coproduct $\lambda$ on $H_*\Omega$. Together with the Pontryagin product $\mu$ this defines the structure of a unital infinitesimal anti-symmetric
bialgebra on $H_*\Omega$ such that $\lambda\eta=0$, where $\eta$ is the unit for the product $\mu$. In particular, the following relation holds 
$$
\lambda\mu = (\mu\otimes 1)(1\otimes\lambda) + (1\otimes\mu)(\lambda\otimes 1).
$$ 
\end{theorem}

\begin{proof} This follows from Theorem~\ref{thm:uias_reduced_Lag} combined with the Viterbo isomorphisms. That $\lambda\eta=0$ follows from the fact that the coproduct has degree $1-n$ and the unit is supported in degree $0$, so that $\lambda\eta$ has negative degree if $n\ge 2$.
\end{proof}

\begin{remark}
(a) The bialgebra structure in Theorem~\ref{thm:main-based} is in general not commutative or cocommutative. For example, the Pontryagin algebra with rational coefficients of an even-dimensional sphere is a polynomial algebra on a generator of odd degree, so it is not commutative in the graded sense (see~\cite[\S16 and~\S21]{Felix-Halperin-Thomas} for this and more general results on simply connected spaces). \\
(b) We have stated our theorems under the assumptions $n\ge 3$, respectively $n\ge 2$, for two reasons: (i) in the case of free loops, we have established Theorem~\ref{thm:uias_reduced} under the assumption $2n\ge 6$, and (ii) in the case of based loops, as seen {\alex in~\S\ref{sec:odd-spheres}}, for $n=1$ (i.e.~$M=S^1$) the coproduct does not necessarily vanish on the unit. As a matter of fact, in this last case the extension of the coproduct is not anymore canonical either.  
\end{remark}

{\reftwo
\subsection{Splitting of Rabinowitz loop homology} \label{sec:Rab-loop}

We retain the setting of the previous subsection. Denote by $S^*M=\p D^*M$ the sphere cotangent 
bundle. Following~\cite{CHO-PD}, we define the (degree shifted) {\em Rabinowitz loop homology}
$$
  \wh\H_*\Lambda := S\H_*(D^*M)
$$
with its canonical product $\boldmu$. In view of the Viterbo isomorphisms above, the middle row of the short exact sequence~\eqref{eq:les+reduced} becomes 
\begin{equation}\label{eq:les+reduced-loop}
\xymatrix
@C=15pt
{
  0 \ar[r] & \ol\H_*\Lambda \ar[r]^-{\iota} 
  & \wh\H_*\Lambda \ar[r]^-{\pi} 
  & \ol\H^{1-2n-*}\Lambda \ar[r] & 0 
}
\end{equation}
Denote $\lambda^\vee$ the loop cohomology product on $H^*(\Lambda,\Lambda_0)$, dual to the loop homology coproduct $\lambda$ on $H_*(\Lambda,\Lambda_0)$. Theorem~\ref{thm:splittings} specializes to

\begin{corollary} \label{cor:splittings-loop}
Let $M$ be a closed $R$-orientable manifold of dimension $n\ge 3$. 
The short exact sequence~\eqref{eq:les+reduced-loop} admits a splitting 
\begin{equation}\label{eq:splitting-RFH-loop}
   \wh\H_*\Lambda = \ol\H_*\Lambda \oplus \ol\H^{1-2n-*}\Lambda
\end{equation}
such that the product $\boldmu$ on $\wh\H_*\Lambda$ restricts to the loop product $\bar\mu$ on the subring $\ol\H_*\Lambda$, 
and to an extension of the loop cohomology product $\lambda^\vee$ on the subring (not containing the unit) $\ol\H^{1-2n-*}\Lambda$.

\noindent
The splitting and the extension of $\lambda^\vee$ are canonical if $H_1(M)=0$.
 \qed
\end{corollary}

Note that, while Rabinowitz loop homology was defined in symplectic terms, equation~\eqref{eq:splitting-RFH-loop} provides a purely topological description as the direct sum of reduced loop homology and cohomology. However, this description is less canonical, and the product $\boldmu$ and its associativity are harder to see in this description (see~\S\ref{thm:m-components-RFH}). 

As in the previous subsection, the results of this subsection have analogues for the based loop space. In the case $L=D^*_qM$ we have $SH_{*+n}(L)\simeq H_*\Omega$ and $SH^{*+n}(L)\simeq H^*\Omega$. The map $\eps_*:SH^{n-*}(L)\to SH_*(L)$ described at the beginning of~\S\ref{sec:open_strings} acts as $\eps_*:H^{-*}\Omega\to H_{*-n}\Omega$ and therefore vanishes for degree reasons. Recalling from~\S\ref{sec:open_strings} the notation $S\H_*(\p L)=SH_{*+n}(\p L)$, we define Rabinowitz based loop homology as\footnote{We do not use the notation $\wh \H_*\Omega$ in order to signify that it contains the subring $H_*\Omega$ endowed with the Pontrjagin product of degree $0$.} 
 $$
\wh H_*\Omega:= S\H_*(S^*_qM).
$$
By Viterbo's isomorphisms, the short exact sequence~\eqref{eq:les+reduced-Lag} specializes to
\begin{equation}\label{eq:les+reduced-based-loop}
\xymatrix
@C=30pt
{
   0 \ar[r] & H_*\Om \ar[r]^-\iota & \wh{H}_*\Om \ar[r]^-\pi & H^{1-n-*}\Om \ar[r] & 0, 
}
\end{equation}
and Theorem~\ref{thm:splittings-Lag} specializes to

\begin{corollary} \label{cor:splittings-based-loop}
Let $M$ be a closed $R$-orientable manifold of dimension $n\ge 2$. 
The short exact sequence~\eqref{eq:les+reduced-based-loop} admits a canonical splitting
\begin{equation*}
   \wh H_*\Omega = H_*\Omega \oplus H^{1-n-*}\Omega
\end{equation*}
such that the product $\boldmu$ on $\wh H_*\Omega$ restricts to the Pontrjagin product on the subring $H_*\Omega$, and to the product $\lambda^\vee$ dual to the extension $\lambda$ of the based loop coproduct on the subring (not containing the unit) $H^{1-n-*}\Omega$. \qed
\end{corollary}

}

\subsection{Odd-dimensional spheres} \label{sec:odd-spheres}

We illustrate the algebraic structures of this paper arising from loop spaces of odd-dimensional spheres $S^n$. Denoting $\Lambda S^n$ the free loop space and $\Omega S^n$ the based loop space, we have 
$$
\ol{S\H}_*(DT^*S^n)\simeq \H_*\Lambda S^n, \qquad \ol{S\H}_*(DT^*_qS^n)\simeq H_*\Omega S^n, 
$$
where $\H_*\Lambda S^n=H_{*+n}\Lambda S^n$. 
{\reftwo The following results for the free loop space were proved in~\cite[\S8.2]{CHO-MorseFloerGH}. As noted in~\cite{CHO-PD}, they imply the results for the based loop space by dropping the variable $A$ and the symmetrization. See also~\cite[\S10]{CHO-algebra}. }

{\bf The case of odd $n\ge 3$}. 

{\it Loop homology $\H_*\Lambda S^n$. }
As a ring with respect to the loop product $\boldmu$, ordinary loop homology is given by
$$
   \H_*\Lambda S^n = \Lambda[A,U],\qquad |U|=n-1,\ |A|=-n.
$$
Thus $\boldmu$ has degree $0$, it is associative, commutative, and unital with unit $\boldeta=\bold1$.  
The loop coproduct is given by
\begin{align*}
   \boldlambda(AU^k) &= \sum_{i+j=k-1\atop i,j\geq 0} AU^i\otimes AU^j,\cr
   \boldlambda(U^k) &= \sum_{i+j=k-1\atop i,j\geq 0} (AU^i\otimes U^j-U^i\otimes AU^j).
\end{align*}
The coproduct $\boldlambda$ has odd degree $1-2n$, it is (skew-)coassociative and (skew-)cocommutative, but it has no counit. 
One can verify by direct computation that $(\H_*\Lambda S^n,\boldmu,\boldlambda,\boldeta)$ is a unital infinitesimal anti-symmetric bialgebra.
Since
$$
   \boldlambda\boldeta = \boldlambda(\bold1) = 0,
$$
the {\sc (unital infinitesimal relation)} simplifies to the ``Sullivan's relation" 
$$
  \boldlambda\boldmu = (1\otimes\boldmu)(\boldlambda\otimes 1) + (\boldmu\otimes 1)(1\otimes\boldlambda).
$$

{\it Based loop homology $H_*\Om S^n$. }
As a ring with respect to the Pontrjagin product $\boldmu$, ordinary based loop homology is given by
$$
   H_*\Om S^n = \Lambda[U],\qquad |U|=n-1.
$$
Thus $\boldmu$ has degree $0$, it is associative, commutative, and unital with unit $\boldeta=\bold1$.  
The based loop coproduct is given by
\begin{align*}
   \boldlambda(U^k) &= \sum_{i+j=k-1\atop i,j\geq 0} U^i\otimes U^j.
\end{align*}
The coproduct $\boldlambda$ has even degree $1-n$, it is coassociative and cocommutative, but it has no counit. 
One can verify by direct computation that $(H_*\Om S^n,\boldmu,\boldlambda,\boldeta)$ is a unital infinitesimal anti-symmetric bialgebra. 
Since
$$
   \boldlambda\boldeta = \boldlambda(\bold1) = 0,
$$
the {\sc (unital infinitesimal relation)} simplifies to Sullivan's relation
$$
  \boldlambda\boldmu = (1\otimes\boldmu)(\boldlambda\otimes 1) + (\boldmu\otimes 1)(1\otimes\boldlambda).
$$

Note that the structure is \emph{not} involutive in the based loop case. 

\bigbreak 

{\bf The case $n=1$. }

{\it Loop homology $\H_*\Lambda S^1$. }
As a ring with respect to the loop product $\boldmu$, ordinary loop homology is given by
$$
   \H_*\Lambda S^1 = \Lambda[A,U,U^{-1}],\qquad |U|=0,\ |A|=-1.
$$
Thus $\boldmu$ has degree $0$, it is associative, commutative, and unital with unit $\boldeta=\bold1$.  

The loop coproduct in this case depends on the choice of a nowhere vanishing vector field on $S^1$ {\reftwo (cf.~Remark~\ref{rem:cont-data})}. 
Up to homotopy there are two such choices $v^\pm(x) = \pm 1$, giving rise to two coproducts 
\begin{align*}
  \boldlambda_+(AU^k) &= \begin{cases}
     \sum_{i=0}^{k}AU^i\otimes AU^{k-i} & k\geq 0, \\
     -\sum_{i=k+1}^{-1}AU^i\otimes AU^{k-i} & k<0
  \end{cases} \cr 
  \boldlambda_+(U^k) &= \begin{cases}
     \sum_{i=0}^{k}(AU^i\otimes U^{k-i} - U^i\otimes AU^{k-i}) & k\geq 0, \\
     -\sum_{i=k+1}^{-1}(AU^i\otimes U^{k-i} - U^i\otimes AU^{k-i}) & k<0
  \end{cases}  
\end{align*}
\begin{align*}
  \boldlambda_-(AU^k) &= \begin{cases}
     \sum_{i=1}^{k-1}AU^i\otimes AU^{k-i} & k > 0, \\
     -\sum_{i=k}^{0}AU^i\otimes AU^{k-i} & k\leq 0,
  \end{cases} \cr 
  \boldlambda_-(U^k) &= \begin{cases}
     \sum_{i=1}^{k-1}(AU^i\otimes U^{k-i} - U^i\otimes AU^{k-i}) & k> 0, \\
     -\sum_{i=k}^{0}(AU^i\otimes U^{k-i} - U^i\otimes AU^{k-i}) & k\leq 0.
  \end{cases}  
\end{align*}
The coproduct $\boldlambda_\pm$ has odd degree $-1$, it is (skew-)coassociative and (skew-)cocommutative, but it has no counit. 
One can verify by direct computation that $(\H_*\Lambda S^1,\boldmu,\boldlambda_\pm,\boldeta)$ is a unital infinitesimal anti-symmetric bialgebra. 
Since
$$
   \boldlambda_\pm\boldeta = \boldlambda_\pm(\bold1) = \pm(A\otimes \bold1-\bold1\otimes A),
$$
the {\sc (unital infinitesimal)} and {\sc (unital anti-symmetry)} relations contain nontrivial terms involving $\boldlambda_\pm\boldeta$.
One verifies directly the relation from Corollary~\ref{cor:lambdaDprimeD} in the form 
$$
\boldlambda_- = \boldlambda_+ + (\boldmu\otimes 1)(1\otimes\boldc) + (1\otimes\boldmu)(\boldc\otimes 1),
$$
with $\boldc=1\otimes A - A\otimes 1$ the continuation bivector (here $\boldc=\boldlambda_-(1)$). 

{\it Based loop homology $H_*\Om S^1$. }
As a ring with respect to the Pontrjagin product $\boldmu$, ordinary based loop homology of $S^1$ is given by
$$
  H_*\Om S^1 = \Lambda[U,U^{-1}],\qquad |U|=0.
$$
Thus $\boldmu$ has degree $0$, it is associative, commutative, and unital with unit $\boldeta=\bold1$.  

Again, the two choices of nowhere vanishing vector fields $v^\pm(x) = \pm 1$ give rise to two coproducts 
\begin{align*}
  \boldlambda_+(U^k) &= \begin{cases}
     \sum_{i=0}^{k}U^i\otimes U^{k-i} & k\geq 0, \\
     -\sum_{i=k+1}^{-1}U^i\otimes U^{k-i} & k<0,
  \end{cases} \cr 
  \boldlambda_-(U^k) &= \begin{cases}
     \sum_{i=1}^{k-1}U^i\otimes U^{k-i} & k > 0, \\
     -\sum_{i=k}^{0}U^i\otimes U^{k-i} & k\leq 0
  \end{cases} 
\end{align*}
taking values in the ordinary (uncompleted) tensor product.
The coproduct $\boldlambda_\pm$ has even degree $0$, it is coassociative and cocommutative, but it has no counit. 
One can verify by direct computation that $(H_*\Om S^1,\boldmu,\boldlambda_\pm,\boldeta)$ is a unital infinitesimal anti-symmetric bialgebra. 
Since
$$
   \boldlambda_\pm\boldeta = \boldlambda_\pm(\bold1) = \pm(\bold1\otimes \bold1),
$$
the {\sc (unital infinitesimal relation)} simplifies to
$$
  \boldlambda_\pm\boldmu = (1\otimes\boldmu)(\boldlambda_\pm\otimes 1) + (\boldmu\otimes 1)(1\otimes\boldlambda_\pm) \mp 1\otimes 1.
$$
Thus $(H_*\Om S^1,\boldmu,\boldlambda_+,\boldeta)$ 
and $(H_*\Om S^1,\boldmu,-\boldlambda_-,\boldeta)$ are ``unital infinitesimal bialgebras'' in the sense of Loday--Ronco~\cite{Loday-Ronco}. 
One verifies directly the relation from Corollary~\ref{cor:lambdaDprimeD}
(see also~\S\ref{sec:open_strings})
in the form 
$$
\boldlambda_- = \boldlambda_+ + (\boldmu\otimes 1)(1\otimes\boldc) + (1\otimes\boldmu)(\boldc\otimes 1),
$$
with $\boldc=-1\otimes 1$ the continuation bivector (here $\boldc=\boldlambda_-(1)$).

\appendix

\section{Grading conventions} \label{sec:grading}

Our standing convention is to grade Hamiltonian Floer homology by the Conley-Zehnder index of orbits, and Lagrangian Floer homology by the Conley-Zehnder index of chords. 

Given a preferred trivialization, 
the Conley-Zehnder index $\CZ(\gamma)$ of a Hamiltonian orbit $\gamma$ is such that the Fredholm indices of the capping operators $o_\pm(\gamma)$ defined on the sphere with one positive/negative puncture, and with asymptotic behavior at the puncture given by the linearized Hamiltonian flow along the periodic orbit $\gamma$, are given by (see~\cite[\S4.1]{Ekholm-Oancea} and references therein)
\begin{equation} \label{eq:capping-operators-orbits}
\mathrm{index}(o_-(\gamma))=n-\CZ(\gamma),\qquad \mathrm{index}(o_+(\gamma))=n+\CZ(\gamma). 
\end{equation}

Given a preferred trivialization, 
the Conley-Zehnder index $\CZ(c)$ of a Hamiltonian chord $c$ is such that the Fredholm indices of the capping operators $o_\pm(c)$ defined on the disc with one boundary puncture which is positive/negative, and with asymptotic behavior at the boundary puncture given by the linearized Hamiltonian flow along the Hamiltonian chord $c$ with endpoints on the Lagrangian, are given by (see~\cite[\S4.1]{Ekholm-Oancea} and references therein) 
\begin{equation} \label{eq:capping-operators-chords}
\mathrm{index}(o_-(c))=n-\CZ(c),\qquad \mathrm{index}(o_+(c))=\CZ(c). 
\end{equation}

The case in point is the cotangent fiber $L=D^*_qM$. This is a Maslov 0 exact Lagrangian. Moreover, there is a preferred trivialization of $TT^*M$ along Hamiltonian chords given by the complexification of the Lagrangian vertical distribution, and this determines a canonical grading on the symplectic homology and cohomology groups.

The degree of an operation determined by a punctured Riemann surface is calculated by gluing negative/positive capping operators at the positive/negative punctures. For example the half pair of pants product on $SH_*(L)$ is defined by the disc with 2 positive boundary punctures and one negative boundary puncture, so that it has degree $-n$. The half pair of pants primary coproduct on $SH_*(L)$ is defined by the disc with one positive boundary puncture and 2 negative boundary punctures, so that it has degree $0$. The half pair of pants secondary coproduct on $SH_*^{>0}(L)$ has degree one higher, i.e. $+1$, because it is defined by a parametrized Floer problem with one-dimensional parameter space. Its (algebraic) dual half pair of pants product on $SH^*_{>0}(L)$ has opposite degree, i.e. $-1$.

{\alex Given an orbit $\gamma$ of a Hamiltonian $H_t$, denote $\bar\gamma(t)=\gamma(1-t)$ the corresponding orbit for the Hamiltonian $\bar H_t=-H_{1-t}$. Similarly, given a chord $c$ for $H_t$, denote $\bar c(t)=c(1-t)$ the corresponding chord for $\bar H_t$. We have $\mathrm{index}(o_-(\gamma))= \mathrm{index}(o_+(\bar\gamma))$ and $\mathrm{index}(o_-(c))= \mathrm{index}(o_+(\bar c))$, so that equations~\eqref{eq:capping-operators-orbits} and~\eqref{eq:capping-operators-chords} imply the equalities
$$
\CZ(\bar \gamma)=-\CZ(\gamma),\qquad \CZ(\bar c)=n-\CZ(c).
$$
As a reflection of these formulas, the Poincaré duality isomorphism for a Liouville domain $W$ and for an exact Lagrangian $L\subset W$ writes~\cite{CO}
$$
SH_*(W,\p W)\simeq SH^{-*}(W), \qquad SH_*(L,\p L)\simeq SH^{n-*}(L).
$$ 
}

\bibliographystyle{abbrv}
\bibliography{000_SHpair}

\end{document}